\documentclass[11pt]{amsart}
 \usepackage{amsmath,amsthm,amsfonts,amssymb}
\usepackage[mathscr]{euscript}
 \usepackage{hyperref} 
 \usepackage{mathrsfs} 
\usepackage{graphicx}
\usepackage{color}
  \usepackage{epsfig}

\begin{document}
 
\title{ A hierarchic multi-level energy method  \\
for the control of bi-diagonal and mixed 
n-coupled \\
cascade systems of PDE's by 
a reduced number of controls }
\thanks{Accepted for publication: May 2013.}
\date{}
\maketitle     
 
\vspace{ -1\baselineskip}

{\small
\begin{center}
 {\sc Fatiha Alabau-Boussouira} \\
Universit\'e de Lorraine, IECL UMR 7502 and Inria Equipe-projet Corida
\\
 57045 Metz Cedex 1, France \\[10pt]
 (Submitted by: Jean-Michel Coron)  
\end{center}
}

\numberwithin{equation}{section}
\allowdisplaybreaks

 \smallskip

 \begin{quote}
\footnotesize
{\bf Abstract.}  
This work is concerned with the exact controllability/observability  of
abstract cascade hyperbolic systems by a reduced number of controls/observations. We prove that the observation of the last component of the vector state allows to recover the initial energies of all of its components in suitable functional spaces under a necessary and sufficient condition on the coupling operators for cascade bi-diagonal systems. The approach is based on a multi-level energy method which
involves $n$-levels of weakened energies. 
 We establish this result for the case of bounded as well as unbounded dual control operators and under the hypotheses of partial coercivity of the  $n-1$ coupling operators on the sub-diagonal of the system.  
 We further extend our observability result to mixed bi-diagonal and non bi-diagonal $n+p$-coupled cascade systems by $p+1$ observations. Applying the HUM method, we derive the corresponding exact controllability results for
$n$-coupled bi-diagonal cascade and $n+p$-coupled mixed cascade systems. Using the transmutation method for the wave operator, we prove that the corresponding heat 
(resp. Schr\"odinger)  multi-dimensional cascade systems are null-controllable for control regions and coupling regions which are disjoint from each other and for any positive time for 
$n \le 5$ for dimensions larger than $2$, and for any $n \ge 2$ in the one-dimensional case. The controls can be localized on a subdomain or on the boundary and in the one-dimensional case the coupling coefficients can be supported in any non-empty subset of the domain. 
\end{quote}

\newtheorem{Theorem}{Theorem}[section]
\newtheorem{Definition}[Theorem]{Definition}
\newtheorem{Proposition}[Theorem]{Proposition}
\newtheorem{Lemma}[Theorem]{Lemma}
\newtheorem{Corollary}[Theorem]{Corollary}
\newtheorem{Hypothesis}[Theorem]{Hypothesis}
\newtheorem{Remark}[Theorem]{Remark}
\newtheorem{Example}[Theorem]{Example}
\newtheorem{Examples}[Theorem]{Examples}
\newtheorem{Problem}[Theorem]{Problem}

\def\aa{\alpha}
\def\t{\tau}
\def\R{\mathbb R}
\def\N{\mathbb N }
\def\O{\mathbb \Omega}
\def\eps{\varepsilon }

\newcounter{prop}

\section{Introduction}
The control of reaction-diffusion systems, of coupled hyperbolic systems, or of more complex systems involved in medical, biological, chemical or mechanical applications,  has become a more and more challenging issue for more than a decade. From the point of view of control theory, it is important for practical applications and for cost reasons to control these systems by a reduced number of controls, that is by a number of controls which is strictly smaller than the number of unknowns/equations involved in the system. 
Indirect controllability questions for coupled systems appear  for instance naturally as soon as  one wants to build insensitizing controls for scalar equations or simultaneous control for certain classes of coupled systems in parallel.  We shall describe more precisely some examples for these two applications.

\smallskip

\noindent {\bf Insensitizing controls.}
 J. -L. Lions has introduced this definition in~\cite{lions89} for heat type equations, to describe
controls which are robust to small perturbations on the initial data with respect to a given measurement of the solutions.  Let us describe formally 
this notion for the scalar wave equation in a bounded open set $\Omega \subset \mathbb{R}^d$ with a smooth boundary $\Gamma$ (see \cite{Dager06, alabau2cascade} for the mathematical justification)
$$
\begin{cases}
y_{tt} - \Delta y =b v \mbox{ in } (0,T) \times \Omega \,,\\
y=0 \mbox{ in } (0,T) \times \Gamma  ,\\
(y, y_t)(0,.)=(y^0 +\tau_0 z^0, y^1 + \tau_1 z^1) \mbox{ in } \Omega\,, 
\end{cases}
$$
 where $(y^0,y^1)$ are given known initial data, whereas $(z^0,z^1)$ are unknown perturbations of the initial data and of norm $1$
in appropriate functional spaces, $\tau_0$ and $\tau_1$ are real (small) numbers measuring the
amplitude of the perturbations $(z^0, z^1)$, $v$ is the control and $b$ is the control coefficient, which may vanish in some region of $\Omega$. We associate
to the solution $y$ the following measurement
$$
\phi(y;\tau_0, \tau_1)=\frac{1}{2} \int_0^T\int_{\Omega} c(x) y^2 \,dx\, dt \,,
$$
 where $c \ge 0$ is the observation coefficient, which may also vanish in some sub-regions of $\Omega$ and
may have a disjoint support from that of the control coefficient $b$.
One says that the control $v$ insensitizes $\phi$ if the following property holds
$$
\frac{\partial \phi}{\partial \tau_0}(y;0,0)=\frac{\partial \phi}{\partial \tau_1}(y;0,0)=0 \,,
$$
 for all $(z^0,z^1)$ of norm $1$ in the appropriate spaces. We have formally
$$
\frac{\partial \phi}{\partial \tau_0}(y;0,0)=\int_0^T\int_{\Omega} c(x) y_2 w \,dx\, dt \,, \ 
\frac{\partial \phi}{\partial \tau_1}(y;0,0)=\int_0^T\int_{\Omega} c(x) y_2 z \,dx\, dt \,,
$$
 where $y_2$, $w$ respectively solve
$$
\begin{cases}
y_{2,tt} - \Delta y_2 =b v \mbox{ in } (0,T) \times \Omega \,,\\
y_2=0 \mbox{ in } (0,T) \times \Gamma \,,\\
(y_2, y_{2,t})(0,.)=(y^0, y^1) \mbox{ in } \Omega\,,
\end{cases}
\begin{cases}
w_{tt} -\Delta w=0\mbox{ in } (0,T) \times \Omega \,,\\
w=0 \mbox{ in } (0,T) \times \Gamma \,,\\
(w, w_t)(0,.)=(z^0, 0) \mbox{ in } \Omega\,,
\end{cases}
$$
and $z$ solves
$$
\begin{cases}
z_{tt} -\Delta z=0\mbox{ in } (0,T) \times \Omega \,,\\
z=0 \mbox{ in } (0,T) \times \Gamma \,,
\end{cases}
(z, z_t)(0,.)=(0, z^1) \mbox{ in } \Omega\,.
$$
 We introduce the auxiliary equation
$$
\begin{cases}
y_{1,tt} - \Delta y_1 +c(x) y_2 =0 \mbox{ in } (0,T) \times \Omega \,,\\
y_1=0 \mbox{ in } (0,T) \times \Gamma \,, \\
(y_1, y_{1,t})(0,.)=(y_1^0, y_1^1) \mbox{ in } \Omega\,.
\end{cases}
$$
 Then, multiplying the above equation by $w$, integrating
over $(0,T)\times \Omega$ and using the equation in $w$, we have formally
$$
\frac{\partial \phi}{\partial \tau_0}(y;0,0)=\int_0^T\int_{\Omega} cy_2 w=\int_0^T\int_{\Omega}
(-y_{1,tt} +\Delta y_1)w= \Big[\int_{\Omega} -y_{1,t}w +y_1 w_t
\Big]_0^T \,.
$$
 In a similar way, we have
$$
\frac{\partial \phi}{\partial \tau_1}(y;0,0)=\int_0^T\int_{\Omega} cy_2 z=\int_0^T\int_{\Omega}
(-y_{1,tt} +\Delta y_1)z= \Big[\int_{\Omega}- y_{1,t}z +y_1 z_t
\Big]_0^T \,.
$$
 The insensitizing property will hold as soon as the control $v$ is such that the solution of the $2$ order cascade system
\begin{equation}\label{modex}
\begin{cases}
y_{1,tt} - \Delta y_1 +c(x) y_2 =0 \mbox{ in } (0,T) \times \Omega \,,\\
y_{2,tt} - \Delta y_2 =b v \mbox{ in } (0,T) \times \Omega \,,\\
y_1=y_2=0 \mbox{ in } (0,T) \times \Gamma \,,\\
(y_1, y_{1,t})(0,.)=(y_1^0, y_1^1)  ,\\
(y_2, y_{2,t})(0,.)=(y^0, y^1) \mbox{ in } \Omega\,,
\end{cases}
\end{equation}
 satisfies the following property, 
for any $(y^0,y^1)$ given in appropriate energy space
$ 
(y_1,y_{1,t})(0,.)=(y_1,y_{1,t})(T,.)=0$  in $ \Omega  .
$ 
 A similar formal analysis can be performed 
for the building of insensitizing controls on a subset $\Gamma_1$ of the
boundary $\Gamma$. 
Hence, the existence of insensitizing controls for the scalar wave equation is directly linked to an exact controllability result for
a cascade system of order $2$ by a single control. 
This problem can be reformulated as 
$$
Y^{\prime\prime} + \mathcal{M}_2Y= \mathcal{B}_2v\,, \  \ (Y,Y^{\prime})(0)=(y_1^0,y^0,y_1^1,y^1) \,,
$$
 where $Y=(y_1,y_2)^t$, $\mathcal{B}_2v=(0,bv)^t$ and where the matrix operator $\mathcal{M}_2$ involved in \eqref{modex} has the following upper triangular form
$ 
\mathcal{M}_2=
\left(\begin{matrix} 
A &  c I\\
0 & A  
\end{matrix}
\right),
$ 
 where $I$ stands for the identity operator in $L^2(\Omega)$ and $A=-\Delta$ stands for the homogeneous Dirichlet Laplacian. We can remark that for cascade systems arising from insensitizing questions, the coupling operators are partially coercive (due to the above properties of
the coefficient $c$).

Note also that for other examples the coupling terms may be of higher order. This is the case for the mechanical model of Timoshenko beams for instance, 
which couples two wave equations by first order terms. Also one may consider in applications two controls (see e.g. ~\cite{kimren, taylor}) or a single control (see e.g. \cite{soufyane, ABRR, nodeaala}).

\smallskip

\noindent 
{\bf Simultaneous control for coupled systems.}
Another application is the simultaneous control of systems of hyperbolic equations coupled in parallel (see e.g.~\cite{lions, tucsnakweiss}). 
We denote by $\mathcal{L}$ a uniformly elliptic operator on $\Omega \subset \mathbb{R}^n$ with smooth coefficients, subjected to homogeneous Dirichlet boundary conditions. We set $p=(p_1, p_2, p_3)^t$, and use the notation $p_{tt}=(p_{1,tt}, \ldots, p_{3,tt})^t$ ,
$\mathcal{L}p=(\mathcal{L}p_1, \ldots, \mathcal{L}p_3)^t$.  Let $\alpha$ and $\beta$ be given functions on the set $\Omega$.
We consider the following control problem for $t \in (0,T)$ and $x \in \Omega$
\begin{equation}\label{simulta}
\begin{cases}
p_{1,tt} - \mathcal{L}p_1 - (3 \alpha + \beta)p_1 +(\alpha+\beta)p_2 + (2 \alpha+\beta)p_3= v_1\,, \\
p_{2,tt} - \mathcal{L}p_2 - (3 \alpha -\beta)p_1 +(-\alpha+\beta)p_2 + (-2 \alpha +\beta)p_3= v_2\,, \\
p_{3,tt} - \mathcal{L}p_3 - 6 \alpha p_1 +2\alpha p_2 + 4\alpha p_3= v_3\,, 
\end{cases}
\end{equation}
 where the initial conditions for $p$ are known and where
$v_1, v_2, v_3 \in L^2((0,T) \times \Omega)$ are the controls. We shall consider the simultaneous control problem, that is we look for controls
$(v_1,v_2,v_3)=(\eta_1 h, \eta_2 h, h)$ which depend on a single scalar control $h \in L^2((0,T) \times \Omega)$, where $\eta_1, \eta_2$
are fixed real coefficients.
Hence, for each given initial data we look for a scalar control $h$ which could simultaneously drive back to equilibrium at time $T>0$ each component of the system, i.e. which is such that $p_i(T)=p_{i,t}(T)=0$, $i=1,2,3$ for a sufficiently large time $T$. Indeed, making an appropriate change of unknowns, we can transform this simultaneous control problem into a control problem for a bi-diagonal cascade problem by a single control. More precisely, set
\begin{align*}
y_1 & =\frac{1}{4}(-p_1 +p_2 -2p_3)  ,  \\
y_2 & =\frac{1}{4}(3p_1 -p_2 -2p_3) , \\
 y_3 & =(-p_1 +p_2 +p_3) ,
\end{align*}
 and $\eta_1=2\,,\eta_2=4$.
 Then $y=(y_1,y_2,y_3)^t$ is the solution of the following bi-diagonal cascade system
$$
\begin{cases}
y_{1,tt} - \mathcal{L}y_1+ 6 \alpha(x) y_2=0 \,, t \in (0,T) \,, x \in \Omega \,,\\
y_{2,tt} - \mathcal{L}y_2+ \frac{\beta(x)}{2} y_3= 0\,, t \in (0,T) \,, x \in \Omega \,,\\
y_{3,tt} - \mathcal{L}y_3=3h \,, t \in (0,T) \,, x \in \Omega \,.
\end{cases}
$$
 Therefore, if we can solve the exact controllability of the above bi-diagonal cascade system, we can find a simultaneous control $h$
which drives back the solution of \eqref{simulta} at equilibrium at a sufficiently large time $T$. The above example can be generalized to simultaneous control systems of order $n$.

Thus, a natural question raising from these applications and from the case of $2$-coupled cascade systems given in \eqref{modex} is to extend the controllability results
for $2$-coupled cascade systems to cascade systems coupling at least $3$ equations
or more. Moreover we are motivated by giving sharp conditions on the geometry of the control and coupling regions, and in particular by cases for which the control and coupling regions do not meet and for boundary controls as well. This is an open problem for the corresponding parabolic systems.
We shall restrict our analysis to two subclasses of cascade coupled systems: the bi-diagonal and mixed cascade systems. We generalize the two-level energy method into a {\em hierarchic multi-level energy method} to study the indirect controllability of such systems. This {\em constructive} method uses the property that one can derive from the original system set in the natural energy space a {\em hierarchy of related systems} similar to the original one, but set in weakened energy spaces. The solutions of these hierarchic systems are linked to each other and  this {\em rich structure} allows us to get positive controllability results.
These subclasses can be seen as a toy model to understand and capture essential properties which guarantee controllability by a reduced number of controls. If one cannot understand controllability questions by a reduced number of controls for such subclasses then there is no hope for more general structures.  
We  give applications to explicit coupled cascade systems of PDE's, namely the wave, heat and Schr\"odinger ones.  Our abstract results apply as well to coupled cascade wave systems with variable (smooth) coefficients, to
cascade systems of Petrowsky equations \ldots and also under other sufficient geometric conditions derived by other tools such as frequency and spectral methods, even though we do not present such applications in this paper for reasons due to the length of the paper.

Let us describe the two subclasses of cascade systems we consider in this paper, namely
the bi-diagonal $n$-coupled cascade system 
$$
Y^{\prime\prime} + \mathcal{M}_nY =\mathcal{B}_n{\bf v}\,,
$$
 with $Y=(y_1, y_2, \ldots y_n)^t$, $\mathcal{B}_n{\bf v}=(0, \ldots, B_n v)^t$
and for which the matrix operator $\mathcal{M}_n$ have the following form 
\begin{equation}\label{Mn}
\mathcal{M}_n=
\begin{pmatrix}
A &  c_{21}I &0 & \ldots\\
0 & A & c_{32} I & 0 & \ldots \\
\vdots \\
0 & 0 & \ldots & A & c_{n n-1}I \\
0 & 0& \ldots & 0 & A
\end{pmatrix} \,,
\end{equation}
 and the mixed bi-diagonal and non bi-diagonal $n+p$-coupled cascade system 
$$
Y^{\prime\prime} + \mathcal{M}_{n+p}Y =\mathcal{B}_{n+p}{\bf v} \,,
$$
 with $Y=(y_1, y_2, \ldots y_{n+p})^t$, $\mathcal{B}_{n+p}{\bf v}=(0, \ldots, B_{n} v_{n}, \ldots, B_{n+p}v_{n+p})$,
where ${\bf v}=(v_n, \ldots, v_{n+p})$ are the $p+1$ controls, and
where the matrix operator is given by
\begin{align} \label{Mn+p}
&
\mathcal{M}_{n+p}=
\notag \\
&  
\begin{pmatrix}
A &  c_{21}I &0 & \ldots & 0 & 0 & 0 & 0 & 0\\
0 & A & c_{32}I & 0 & \ldots  &0 & 0 & 0 &0 \\
\vdots \\
0 & 0 & \ldots & A & c_{n-1 n-2}I & 0 & \ldots & 0 & 0\\
0 & \ldots & 0 & 0& A & c_{n n-1}I & c_{n+1 n-1} I & \ldots & c_{n+p \,n-1} I\\
\vdots\\
0 & 0 \ldots & 0 & \ldots & 0 & 0& 0& A & c_{n+p\, n+p-1} I\\
0& 0 & 0 & \ldots & 0 & 0 & 0 & \ldots & A
\end{pmatrix}
\end{align}
 where $n \ge 3$ and $p \ge 0$.  One can note that the above matrix operator is bi-diagonal up to the $n-1$ line and is no longer bi-diagonal starting from the $n$ line.

We first study these two classes in an abstract form, that is for general unbounded self-adjoint and coercive operators $A$ and
bounded and partially coercive coupling operators $C_{i i-1}$ in a Hilbert space $H$.
We give  a necessary and sufficient {\em observability-type condition}, on the coupling operators $C_{i i-1}$, so that the bi-diagonal $n$-coupled cascade systems are controllable by a single
control acting on the last equation, either by bounded or unbounded control operators. We then prove that mixed bi-diagonal and non bi-diagonal $n+p$-coupled cascade
systems are controllable by $p+1$ controls acting on the last $p+1$ equations, some of the controls being bounded
and the other unbounded. The necessary and sufficient abstract condition allows us to get results under sharp geometric conditions, issued from micro-local analysis, or spectral and frequency approaches (or also multipliers methods) for wave coupled systems
under additional compatibility assumptions on some of the coupling coefficients if the number of equations is larger or equal to $4$. 

We give applications of these
abstract results to hyperbolic, parabolic and Schr\"odinger cascade systems of order larger than $3$. These results solve partially
some conjectures on coupled parabolic cascade systems allowing to exhibit coupled systems with variable coupling coefficients for which
the control regions do not meet the coupling regions (for less than $5$ equations). 

The question of positive controllability results for parabolic systems for arbitrary non empty open coupling and control regions is still widely open.

The assumption of partially coercive coupling operators can also be discussed. On one side, this assumption is naturally satisfied for several applications as for instance for cascade systems arising from insensitizing control or for examples arising from applications such as Timoshenko beams for instance. On the other side, from a mathematical point of view, one may want to relax this assumption. If for instance, the coupling operators
$C_{i i-1}$ are given by $C_{i i-1}u=c_{i i-1}u$ for $u \in L^2(\Omega)$ in the case of systems of cascade wave equations, the necessary conditions $(A3)_n$ with $\Pi_i=C_{i i-1}$ established in Theorem~\ref{nNEC},  only require that the $supp\{c_{i i-1}\}$ satisfy the Geometric Control Condition. It is an interesting open question to obtain positive or negative control results for $2$-coupled cascade wave systems  for a sign varying coupling coefficient $c$ where $supp \{c\}$ satisfies the Geometric Control Condition. All the positive results~\cite{sicon03, alaleaucont, alaleau11, alabau2cascade, Dager06, DLRL, RdT11, tebou2011} for coupled systems assume that the coupling coefficient has a constant sign over the whole domain. To our knowledge, there are no results for coupled systems with sign varying coupling coefficients.Thus, this is an open problem for $2$-coupled cascade systems, and further for $n$-coupled systems for which the dynamics is being more complex.  

\smallskip

\noindent
{\bf On the literature on control of coupled systems by a reduced number of controls.}
Let us give some picture of the literature on the subject of coupled hyperbolic as well as parabolic systems. The question of controllability of symmetric weakly $2$-coupled hyperbolic systems by a single control is adressed by the author in~\cite{alacras01, sicon03}. These systems have the form
$$
Y^{\prime\prime} + \mathcal{S}_2Y= \mathcal{B}_2v\,, 
\ \ \ (Y,Y^{\prime})(0)=(y_1^0,y^0,y_1^1,y^1) \,,
$$
 where $Y=(y_1,y_2)^t$, $\mathcal{B}_2v=(Bv,0)^t$ and where the matrix operator $\mathcal{S}_2$ has the following symmetric form
$ 
\mathcal{S}_2=
\left(
\begin{matrix}
A &  C\\
C^{\ast} & A
\end{matrix}
\right) ,
$ 
 where $C$ stands for a bounded operator in $L^2(\Omega)$ and $A=-\Delta$ stands for the homogeneous Dirichlet Laplacian.
We prove in~\cite{alacras01, sicon03}, positive indirect controllability/observability results by means of a {\em two-level energy method}.  We introduce this method in a general abstract setting under a coercivity assumption  of $C$ (case of globally distributed couplings) and unbounded control operators (case of boundary control). 

The author and L\'eautaud in~\cite{alaleaucont, alaleau11} extend and simplify these results to the case of symmetric weakly $2$-coupled hyperbolic systems with partially coercive couplings (case of locally supported couplings) and for localized as well as boundary control. Thanks to the transmutation method
~\cite{seidman, miller, phung, EZ}, we also give applications to symmetric $2$-coupled systems of parabolic and diffusive equations. These results are valid for geometric situations for which the control and coupling regions do not meet and in the case of smooth coupling coefficients. 

Using D\"ager's~\cite{Dager06} approach relying on the periodicity  of the semigroup for the single free wave equation, Rosier and de Teresa~\cite{RdT11} prove positive results for $2$-coupled system of cascade hyperbolic equations with partially coercive operators under this periodicity assumption, hence for one dimensional $2$-coupled cascade wave systems. They give applications to $2$-coupled systems of cascade one-dimensional heat equations and to $2$-coupled systems of cascade Schr\"odinger equations in a $n$-dimensional interval with empty intersection between the control and coupling regions.  The coupling coefficient is assumed to be a nonnegative function, and does not need to be a smooth function as in the two-level energy method.

On the other hand, Dehman, Le Rousseau and L\'eautaud~\cite{DLRL}  (see also \cite{these-leautaud}) consider a $2$-coupled cascade wave system 
in a $\mathcal{C}^{\infty}$ compact connected riemannian manifold without boundary. They further assume that the coupling coefficient is nonnegative, which is equivalent to assuming that the coupling operator is partially coercive. Under these assumptions, they prove that the system can be controlled by a locally distributed control and further give
a {\em characterization of the minimal control time} using micro-local analysis and the idea to work in weakened energy spaces for the unobserved component (see~\cite{alacras01, sicon03}). 

We give in~\cite{alabau2cascade}, under the assumption of partially coercive operators, a necessary and sufficient condition for the controllability/observability of  $2$-coupled cascade hyperbolic systems by a single, either locally distributed control or a boundary control, in particular for geometric situations for which the coupling and control regions do not meet. We answer to the question of existence of insensitizing controls for the scalar wave equation, generalizing the one-dimensional results and approximate insensitivity results in~\cite{Dager06} to the multi-dimensional case and to exact insensitivity.

The generalization of the above controllability result for $2$-cascade systems to $n$-coupled cascade systems with arbitrary $n\ge 2$ in the present paper is nontrivial and relies on a tricky induction argument on $n$ given in Theorem~\ref{inductionobsNSHn}. One has to establish by induction on $n$ intermediate estimates stated in the property $\mathcal{P}_n$. These successive estimates explain in some way how
the information given by the observation of the last component can be transferred to the unobserved equations through the successive coupling coefficients.

There is a prolific literature on the control of parabolic coupled systems. The survey paper~\cite{AKBGBT11} and the references therein give an overview of the recent results on the null-controllability for coupled parabolic systems.  It is devoted to observability results for the adjoint system based on Carleman estimates and generalizations of the Kalman rank condition in infinite dimension. It also stresses fundamental differences between scalar and vectorial parabolic equations, in particular for boundary control. 

As for the case of the insensitive controls for the scalar wave equation, the question of existence of insensitizing controls for the scalar heat equation is equivalent to a controllability result for an associated forward-backward cascade system. We refer to~\cite{lions89, bodart-fabre95, DeT00, BGBPG04CPDE, BGBPG04SICON, DeTZ, cannarsaLT09} for results on insensitizing control for the heat equation. These studies treat the cases $\omega \cap O\neq \emptyset$, that is the cases for which the control region meets the coupling region, which corresponds to the region on which one wants to insensitive the measurement. We also refer to~\cite{guerrero}  for an insensitizing analysis in the case of fluids models, that is the case of the existence of insensitizing controls for the Stokes equations and~\cite{gueyecras, gueyeNS} for the existence of insensitizing controls for the Navier-Stokes equations.
For  null controllability results on coupled parabolic systems by a single control force for: either constant coupling operators and locally distributed control, or localized coupling operators and locally distributed control regions with a non-empty intersection between control and coupling regions, we refer the reader to~\cite{DeT00, AKBD06, AKBDGB09, FCGBDeT10, GBdT10, leautaud, CGR10, olive, mauffrey}. These results are based on Carleman estimates for the observability of the adjoint system. In the case $\omega \cap O= \emptyset$, Kavian and de Teresa~\cite{DeTK10} proved a unique continuation result for a $2$-coupled cascade systems of parabolic equations. De Teresa and Zuazua~\cite{DeTZ} give further results concerning the determination of the initial data for which insensitizing controls of the heat equation can be built. 

Coron, Guerrero and Rosier~\cite{CGR10} prove local null controllability results for nonlinearly coupled $2$-systems of parabolic equations with a {\em nonlinear coupling} term arising in control of  chemical reaction-diffusion models. The results are based on the Coron's return method~\cite{coronbk07}.  
One should note that for most of these results, a condition is that the coupling region meets the control/observation region.  It is therefore a challenging issue to determine whether it is possible to control/observe the full vector-state solution of coupled parabolic systems by a reduced number of controls/observations, in the cases for which the control/observation and coupling regions do not meet. This is also one of  the question which motivates this paper. 

This paper is organized as follows.  In Section 2, we give the results and proofs for observability (resp. controllability) by a single observation (resp. control) for $n$-coupled abstract bi-diagonal cascade systems. Section 3 is devoted to the control of mixed bi-diagonal and non bi-diagonal $n+p$-coupled abstract cascade systems by $p+1$ controls. 
Section 4 gives applications of our abstract results to $n+p$-coupled cascade wave, parabolic and Schr\"odinger systems.We give the proofs of the applicative results of Section 2 in Section 5. We discuss our results and indicate open problems in Section 6. We provide the proofs of the results of section 3 in an appendix.

\section{Controllability and observability of 
bi-diagonal $n$-coupled cascade hyperbolic systems by a single control/observation}

\subsection{Observability of bi-diagonal $n$-coupled cascade hyperbolic systems by a single observation}

In~\cite{alabau2cascade}, we have considered $2$-coupled cascade system, that is coupled systems of order $2$,
which are lower  triangular systems . We  consider in this paper, $n$-coupled cascade systems, that is lower triangular hyperbolic systems of order $n\in \mathbb{N}$
with $n \ge 2$.
\begin{equation}\label{NSHFn}
\begin{cases}
u_1^{\prime\prime} + A u_1 = 0 \,,\\
u_2^{\prime\prime} + A u_2+ C_{21}u_1  = 0 \,,\\
u_3^{\prime\prime} + A u_3+ C_{31}u_1 + C_{32}u_2 = 0 \,,\\
\vdots \\
u_n^{\prime\prime} + A u_n+ C_{n1}u_1 + C_{n2}u_2 + \ldots
C_{nn-1}u_{n-1}=0 \,,\\
(u_i,u_i^{\prime})(0)=(u_i^0,u_i^1) \mbox{ for }
i=1, \ldots n\,,
\end{cases}
\end{equation}
 where $A$ satisfies 
\begin{equation*}\label{A1}
(A1)\ 
\begin{cases}
 A : D(A) \subset H \mapsto H \,, A^{\ast}=A\,,\\
\exists \; \omega>0\,,  |A u| \ge \omega |u|  \quad \forall \ u \in D(A) \,,\ \
A \mbox{  has a compact resolvent}\,,
\end{cases}
\end{equation*}
 and where
the operators $C_{ij}$ are bounded in $H$ for $i \in \{2,\ldots, n\}\,, \ 
j \in \{1, \ldots,  i-1\}$. We recover $2$-coupled cascade systems considered in~\cite{alabau2cascade}, when $n=2$.
We set $H_k=D(A^{k/2})$ for $k \in \N$, with the convention
$H_0=H$. The set $H_k$ is equipped with the norm $|\cdot|_k$
defined by $|A^{k/2} \cdot|$ and the associated scalar product. It
is a Hilbert space. We denote by $H_{-k}$ the dual space of
$H_k$ with the pivot space $H$. We equip $H_{-k}$ with the norm
$|\cdot|_{-k}=|A^{-k/2} \cdot|$. 

We define the energy space associated
to \eqref{NSHFn} by $\mathcal{H}_n=H_1^n\times
H_0^n$.
The system \eqref{NSHFn} can then be reformulated as the first order abstract system
\begin{equation}\label{ANSHn} 
U^{\prime} = \mathcal{A}_n U \,,\quad  
U(0)=U^0=(u_1^0,u_2^0 ,\ldots u_n^0,u_1^1,u_2^1,\ldots, u_n^1) \,, 
\end{equation}
 where $U=(u_1,u_2,\ldots, u_n, v_1,v_2, \ldots, v_n)$ and
$\mathcal{A}_n$ is the unbounded operator
in $\mathcal{H}_n$ with domain
$D(\mathcal{A}_n)=H_2^n\times H_1^n$ defined by
\begin{align}\label{An}
\mathcal{A}_nU=(v_1,v_2,\ldots, v_n, & -Au_1,-Au_2 - C_{21}u_1, 
\notag 
\\
\ldots, & -Au_n -C_{n1}u_1 \ldots - C_{nn-1}u_{n-1}) \,.
\end{align}
 The well-posedness of the abstract system
\eqref{ANSHn} for initial data $U_0 \in
\mathcal{H}_n$ using semigroup theory is easy to establish and the usual regularity results for smoother initial data also hold. Moreover, for any $k \in \mathbb{Z}^{\ast}$, if $C_{ij}, C_{ij}^{\ast} \in \mathcal{L}(H_{k-1})$ for all $i \in \{2, \ldots, n\}$ and all $j \in \{1, \ldots, i-1\}$, the problem \eqref{ANSHn} is also well-posed in
$H_k^n \times H_{k-1}^n$. For a solution
$U=(u_1,\ldots, u_n,v_1,\ldots, v_n)$ of \eqref{ANSHn}, we have
$v_i=u_i^{\prime}$ for $i=1,\ldots,n$. 

For this, it is convenient to introduce some
further notations. For a solution
$U=(u_1, \ldots,u_n,u_1^{\prime},\ldots, u_n^{\prime})$ of \eqref{ANSHn}, we set
$ \label{Uin}
U_i=(u_i,u_i^{\prime}) $ for $ i=1,\ldots,n\,.
$
 Moreover, for $U_i \in H_k \times H_{k-1}$, we define the energy of level $k$ as
\begin{equation*}\label{enkU}
e_k(U_i)(t)=\tfrac{1}{2} \Big(
|A^{k/2}u_i|^2 + |A^{(k-1)/2}u_i^{\prime}|^2\Big) \,, \ 
k \in \mathbb{Z} \,, i=1,2\,.
\end{equation*}
Similarly to the case of $2$-coupled cascade systems, 
we can easily deduce the following results (their proofs are left to the reader).

\begin{Proposition}\label{invertAn}
Assume that $A$ satisfies $(A1)$ and define $\mathcal{A}_n$
as in \eqref{An}. Then $\mathcal{A}_n$ is invertible from
$D(\mathcal{A}_n)$ on $\mathcal{H}_n$. Moreover,
for any solution $U=(u_1,\ldots,u_n,v_1,\ldots,v_n)$ of \eqref{ANSHn}, the equation
$ \label{AWn}
\mathcal{A}_nW=U
$
 admits a unique solution $W=(w_1, \ldots, w_n, z_1,\ldots,z_n)$ given by
\begin{equation*}\label{Wn}
\begin{cases}
w_1=-A^{-1}u_1^{\prime} \,,\\
w_2=-A^{-1}u_2^{\prime} + A^{-1}C_{21}A^{-1}u_1^{\prime} \,,\\
\vdots \\
w_n= - A^{-1}u_n^{\prime}- A^{-1}C_{n1}w_1 - \ldots -  A^{-1}C_{nn-1}w_{n-1} \,, \\
z_1=w_1^{\prime}=u_1 \,, \ldots \,, \  z_n=w_n^{\prime}=u_n \,.
\end{cases}
\end{equation*}
 Also, $W$ is then the solution of \eqref{ANSHn}, associated
to the initial data $W(0)=\mathcal{A}_n^{-1}U^0$.
\end{Proposition}

\begin{Corollary}\label{coroinducn}
We assume the hypotheses of Proposition~$\ref{invertAn}.$ Let $l \in \mathbb{N}^{\ast}$ be given.
Then the equation 
$ \label{eqWkn}
\mathcal{A}_n^lW^l=U
$
 admits a unique solution $W^l=\mathcal{A}_n^{-l}U \in D(\mathcal{A}_n^l)$. This solution can be defined by induction as follows
$ \label{inducWn}
W^0=U \,, W^{i+1}=\mathcal{A}_n^{-1}W^i \,, i \in \{0, \ldots, l-1\}\,.
$
 We set
\begin{equation}\label{Wi}
W_i^l=(w_i^l,(w_i^{l})^{\prime})\,, \quad  i=1, \ldots, n \,, \, l \in \mathbb{N} \,.
\end{equation}
 Then we have moreover 
\begin{equation}\label{wil}
(w_i^l)^{\prime}=w_i^{l-1} \,, \quad  \forall \ i=1, \ldots, n \,, \forall \ l  \in \mathbb{N}^{\ast} \,.
\end{equation}
\end{Corollary}
We will use the above notations and properties in all the sequel.

\subsubsection{Main results for the observability of $n$-coupled bi-diagonal cascade systems}
Let $n \ge 2$ be a fixed integer. We will generalize our results on  $2$-coupled cascade systems to $n$-coupled cascade systems of bi-diagonal form, that is under the assumption that
\begin{equation} \label{bidiagonal}
(HC)_n \ \ \ \ \ \ \ \  \  \  \  \  \  \  \  \  \  \  \ 
C_{i j}\equiv 0 \mbox{ for } i \in \{2,\ldots n\}\,, j \in \{1, \ldots, i-2\} \,. 
\ \ \ \ \  \  \  \  \  \  \  \  \  \  \ 
\end{equation}
 Therefore, we discuss cascade systems of the form
\begin{equation}\label{NSHn}
\begin{cases}
u_1^{\prime\prime} + A u_1 = 0 \,,\\
u_i^{\prime\prime} + A u_i+ C_{i i-1}u_{i-1}  = 0 \,, 2 \le i \le n \,,\\
(u_i,u_i^{\prime})(0)=(u_i^0,u_i^1) \mbox{ for }
i=1, \ldots n\,,
\end{cases}
\end{equation}
 For bi-diagonal $n$-coupled cascade system, we shall assume that the coupling operators $C_{i i-1}$ for $i=2, \ldots,n$ satisfy
\begin{equation*}\label{hypCi}
(A2)_n
\begin{cases} 
\mbox{ For all } i \in \{2, \ldots, n\} \mbox{ we have }\\
C_{i i-1}^{\ast} \in \mathcal{L}(H_{k}) \mbox{ for } k \in \{0,1,\ldots n-i+1\}\,, \\
||C_{i i-1}||=\beta_{i}\,,\,
|C_{i i-1}w|^2 \le \beta_i \langle C_{i i-1}w\,,w\rangle \quad \forall \ w \in H \,,\\
\exists \alpha_{i}>0 \mbox{ such that }
\alpha_{i}\, |\Pi_{i} w|^2 \le \langle C_{i i-1}w,w\rangle \quad
\forall \ w \in H \,.
\end{cases}
\end{equation*}
 where the operators $\Pi_i$, $i \in \{2, \ldots,n\}$, satisfy the assumptions
\begin{equation*}\label{Pii}
(A3)_n
\begin{cases}
\mbox{ For all } i \in \{2, \ldots, n\}, \ \
\Pi_{i} \in \mathcal{L}(H)\,,\\
\exists \ T_{0,i}>0, \forall \ T>T_{0,i}, \exists \ \gamma_i(T)>0, \\
\mbox{ such that all the solutions } w \mbox{ of }
w'' + A w = 0       
\mbox{ satisfy }  \\
\int_0^T |\Pi_{i}w^{\prime}|^2 dt \geq \gamma_i(T)e_1(W)(0) \,.
\end{cases}
\end{equation*}
For a given $i \in \{2, \ldots, n\}$, we denote by $G_i$  given Hilbert spaces with norm $||\ ||_{G_i}$ and
scalar product $\langle \,, \rangle_{G_i}$. The spaces
$G_{i}$, $i=2,\ldots, n$ will be identified to their dual spaces in all the sequel. Let
$\mathcal{\mathbf{B^{\ast}_{n}}}$ for $n \ge 2$ be an arbitrary observability operator satisfying the following assumptions:
\begin{equation*}\label{admissibilityi}
(A4)_n
\begin{cases}
\mathcal{\mathbf{B^{\ast}_{n}}} \in \mathcal{L}(H_2\times H;G_{n}),\\
\forall \ T > 0 \ \exists \ D_{n}=D_{n}(T) >0 ,\mbox{ such that all the solutions }
w \mbox{ of }\\
w'' + A w = f  \in L^2([0,T];H)     
\mbox{ satisfy }  \\
\int_0^T \| \mathcal{\mathbf{B^{\ast}_{n}}}(w,w^{\prime}) \|_{G_{n}}^2dt \leq D_{n}(T)  \Big( e_1(W)(0) + e_1(W)(T) + \\
\int_0^T e_1(W)(t) dt +
\int_0^T |f|^2 dt \Big),
\end{cases}
\end{equation*}
 where $W=(w,w^{\prime})$. 
\begin{equation*}\label{observabilityBi}
(A5)_n
\begin{cases}
\exists \ T_{0,n}>0,  \forall \ T>T_{0,n}, \exists \ R_n(T)>0 \\
\mbox{ such that all the solutions } w \mbox{ of }
w'' + A w = 0       
\mbox{ satisfy }  \\
\int_0^T \| \mathcal{\mathbf{B^{\ast}_{n}}}(w,w^{\prime}) \|_{G_{n}}^2 dt \geq R_n(T)e_1(W)(0)\,,
\end{cases}
\end{equation*}
 \begin{Theorem}\label{admissi} $($Admissibility property$)$
 Assume the hypotheses $(A1)$, $(A2)_n$ and $(A4)_{n}$, then for all $T>0$ there exists
 $C(T)>0$ such that for all initial data $U^0 \in \mathcal{H}_n$,
the solution of \eqref{NSHn} satisfies the following direct inequality
 
 \begin{equation}\label{admissineqi}
 \displaystyle{\int_0^T ||\mathcal{\mathbf{B^{\ast}_{n}}}U_n||_{G_{n}}^2 \,dt \le C(T) \Big(
 \sum_{i=1}^{n} e_{1-n+i}(U_{i})(0) \Big) \,,} 
 \end{equation}
  and the following estimates
 \begin{align}
 \label{admiss*}
&
 \displaystyle{\int_0^T |u_{n-1}|^2 \,dt \le C(T) \Big(
 \sum_{i=1}^{n-1} e_{1-n+i}(U_{i})(0) \Big) \,,} 
\\
&
\label{admissx*}
 \displaystyle{\int_0^T |u_{n}|^2 \,dt \le C(T) \Big(
 \sum_{i=1}^{n} e_{1-n+i}(U_{i})(0) \Big) \,.} 
  \end{align}
  \end{Theorem}

  \begin{Remark}
\rm
  We give in Remark~\ref{natspace} the optimal space for well-posedness of \eqref{NSHn} in view of the well-posedness of the corresponding control problem using the duality method HUM (see Lemma~\ref{obsdirn} and Lemma~\ref{obsdirnunbounded}). 
    \end{Remark} 
\begin{Theorem}[Sufficient conditions]\label{obsNSHn}
We assume $(A1)$. Let $n \ge 2$ be an integer.  We assume  
that for all $i=2, \ldots,n$,
the operators $C_{i i-1}$ satisfy the assumption $(A2)_n$ where the operators
$\Pi_i$ satisfy $(A3)_n$. Moreover let $\mathcal{\mathbf{B^{\ast}_n}}$ 
 be any given operator satisfying $(A4)_{n}-(A5)_{n}$. 
Then there exists $T_{n}^{\ast}>0$ such that for all $i=1, \ldots, n$, there exist constants 
$d_{i,n}(T)>0$ such that for all solution $W$ of \eqref{NSHn} and all $T>T_n^{\ast}$ we have
\begin{equation}\label{eqobsk}
e_{1+i-n}(W_i)(0) \le d_{i,n}(T) \int_0^T \| \mathcal{\mathbf{B^{\ast}_n}}(W_n) \|_{G_{n}}^2 dt \,,
\forall \ i=1, \ldots, n \,,
\end{equation}
 with
$ 
d_{i,n}(T) \le \tfrac{K}{T^3}$ for $i \in \{1, \ldots, n-1\}\,, $
$ d_{n,n}(T) \le \tfrac{K}{T} \,,
$ 
where $K$ is bounded with respect to $T$.
\end{Theorem}
We also prove that the above conditions are optimal in the following theorem.
\begin{Theorem}[Necessary conditions]\label{nNEC}
We assume $(A1)$, $(A2)_n$ with $\Pi_i=C_{i i-1}$ for all $i=2, \ldots, n$. Let $\mathcal{\mathbf{B^{\ast}_n}}$ 
 be any given operator satisfying $(A4)_{n}$. Assume that either $\mathcal{\mathbf{B^{\ast}_n}}$
 does not satisfy $(A5)_n$ or that the operators 
 $$
 (C_{i i-1})_{i \in\{2, \ldots, n\}}=(\Pi_i)_{i \in\{2, \ldots, n\}}
 $$
  do not satisfy $(A3)_n$.
Then there does not exist $T_n^{\ast}>0$ such that for all $T>T_n^{\ast}$, 
the following property holds
\begin{equation*}\label{OBSPn}
(OBS)_n \begin{cases}
\exists \ C>0 \mbox{ such that } \forall \ U^0 \in \mathcal{H} \mbox{ the solution of \eqref{NSHn} } \mbox{ satisfies} \quad \\
\displaystyle{C \sum_{i=1}^n e_{1+i-n}(U_i)(0)  \le \int_0^T||\mathcal{\mathbf{B^{\ast}_n}}U_n||_{G_n}^2\,dt \,.}
\end{cases}
\end{equation*}
\end{Theorem}

\begin{Remark}
\rm
We recall that the minimal control time $T_m$ 
under which an observability inequality holds,
 is the minimal time for which the inequality
 holds for all $T >T_{m}$ and does not hold for $T< T_{m}$.
\end{Remark}

We deduce from Theorem~\ref{nNEC} and 
Theorem~\ref{obsNSHn} the following Corollary.

\begin{Corollary}\label{nCNS}
Assume $(A1)$, $(A4)_n$ and $(A2)_n$ with $\Pi_i=C_{i i-1}$ 
for all $i \in \{2,\ldots,n\}$. Then
$(OBS)_n$ holds if and only if $(A3)_n$ and $(A5)_n$ hold. Moreover,
 $T_n^{\ast}$ has to be greater than 
$\max(\max_{2 \le i \le n}(T_{i,c}), T_{0,n})$
 where $T_{i,c}$  for $i=2, \ldots, n$ 
denote for each observability operator $\Pi_i=C_{i i-1}$ the minimal control times for 
which $(A3)_n$ holds, and  $T_{0,n}$ the
 minimal control time for which $(A5)_n$ holds.
\end{Corollary}

\subsubsection{Proofs of the main results for the observability of $n$-coupled cascade systems}

We recall here some of the results of~\cite{alabau2cascade} for the sake of clarity.
The proof of Theorem~\ref{obsNSHn} follows the idea of the two-level energy method~\cite{sicon03}, adapted to $2$-coupled cascade systems in~\cite{alabau2cascade}.
 For systems of order $2$, it consists in using two level of energies, the natural one for the observed component of the state
and the weakened energy of the unobserved component of the state. Here we shall show that we have to use
$n$ levels of weakened (except for the observed one) energies for each of the
$n$ components of the vector-state for coupled cascade systems of order $n$. We shall use, as for the case $n=2$,
a crucial property  for the abstract system (and proved for the applicative examples), that is direct  and observability inequalities for a single equation with a source term, with constants which are uniform with respect to the length $T$ of the time interval $[0,T]$. This property was proved in~\cite{sicon03} thanks to the multiplier method for the usual PDE's (wave, Petrowsky,\ldots) in~\cite{sicon03} and extended  in \cite{alaleaucont} under an abstract form which allows the use of the optimal geometric conditions of Bardos Lebeau and Rauch~\cite{blr92} . We shall use this result in the sequel so we recall that it reads as follows.

\begin{Lemma}[\cite{alaleaucont}, [Lemma 3.3]\label{AL}
We assume hypotheses $(A1)$, $(A4)_n$ and $(A5)_n$. Then, there exist $\eta_0>0$ and $\alpha_0>0$ such that for all $T>T_0$, $f \in L^2([0,T];H)$,  all solutions $P=(p,p^{\prime})$ of  the nonhomogeneous equation
\begin{equation} \label{eqNH} 
p^{\prime\prime} +Ap=f \,,\ \ \
(p,p^{\prime})(0)=(p^0,p^1) \,, 
\end{equation}
  the following uniform observability estimate holds 
\begin{equation}\label{obsabs} 
\eta_0\int_0^T ||\mathcal{\mathbf{B}}^*_nP||_{G_n}^2 \,dt \ge 
\int_0^T e_1(P)(t)\,dt -
\alpha_0\int_0^T|f|^2 \,dt \,.
\end{equation}
\end{Lemma}

\begin{Remark}
\rm
The form of the above inequality for the solutions
 of the abstract second order equation with a source 
term \eqref{eqNH}, is crucial: it is a form which 
respects the invariance by translation in time of 
the equation. Our first results in~\cite{sicon03} 
are based on the following form of the required 
observability inequality:
there exist $\eta_i>0$ for $i=0,\ldots, 3$  
such that for all $T>T_0$, $f \in L^2([0,T];H)$, 
 all solutions $P=(p,p^{\prime})$ of  the 
nonhomogeneous equation \eqref{eqNH} satisfy
\begin{align*}
\eta_0\int_0^T ||\mathcal{\mathbf{B}}^*_nP||_{G_n}^2 \,dt & \ge 
(1-\eta_1 \delta)\int_0^T e_1(P)(t)\,dt 
- \eta_2 (e_1(P)(0)+ e_1(P)(T)) \\
&
-
\eta_3\delta^{-1} \int_0^T|f|^2 \,dt \,,
 \forall \ \delta \in (0, \eta_1^{-1}) .
\end{align*}
 This observability inequality, which is uniform with respect to the time $T$ for sufficiently large $T$, is proved on various examples of PDE's (wave, Petrowsky, \ldots) using the multiplier method.  One can further remark that due to the term
$(e_1(P)(0)+ e_1(P)(T))$ this inequality is not preserved by  time-translations.
However, it naturally arises under this form when using the multiplier method. The drawback is that the resulting geometric conditions on the observation region are then not optimal. On the other hand, the multi-level energy method requires observability constants which are uniform with respect to $T$ for sufficiently large time $T$. A way to handle both constraints is to require an observability inequality in a canonical form which respects the invariance by time-translations. The uniformity of the constants with respect to $T$ is then proved using this invariance property.
This leads  to Lemma 3.3 in \cite{alaleaucont} (recalled in the above Lemma~\ref{AL} under minor changes).
\end{Remark}

\noindent
\textbf{Proof  of Theorem~\ref{admissi}}.
 We have already proved the admissibility property, that is \eqref{admissineqi} for $n=2$ in Lemma 2.5
 in~\cite{alabau2cascade}.  On the other hand,
 the estimate \eqref{admiss*} is trivial for $n=2$. Moreover, the usual energy estimates for the equation in $u_2$ yield
 $$
 \int_0^Te_1(U_2)(t) \,dt \le C(T) \Big( e_1(U_2)(0) + \int_0^T |u_1|^2\,dt\Big)\,.
 $$
  Using \eqref{admiss*} in this inequality, we easily obtain \eqref{admissx*}.  
  Hence, we can assume that $n \ge 3$. We will denote by $C$ generic constants depending in particular on $T$ but not on the initial data.
 Thanks to assumption $(A4)_n$ applied to the last equation of  \eqref{NSHn}, for all
 $T>0$ there exists $C(T)>0$ such that
 \begin{align*} \label{inter1n}
\int_0^T || \mathcal{\mathbf{B^{\ast}_{n}}}U_n||_{G_{n}}^2 \,dt 
&
\le C(T)\Big(e_1(U_n)(0) + e_1(U_n)(T) \\
&
+
\int_0^T e_1(U_n)(t)\,dt + \int_0^T |C_{n n-1}u_{n-1}|^2\,dt
\Big)\,.
 \end{align*}
  Thanks to the usual energy estimates, we obtain 
  \begin{equation}\label{inter2n}
 \int_0^T || \mathcal{\mathbf{B^{\ast}_{n}}}U_n||_{G_{n}}^2 \,dt \le C\Big(e_1(U_n)(0) + 
\int_0^T |u_{n-1}|^2\,dt\Big) \,.
 \end{equation}
  We have
 \begin{equation}\label{inter*n}
 \int_0^T |u_{n-1}|^2\,dt \le 2  \int_0^T e_0(U_{n-1})\,dt \,.
 \end{equation}
  We shall first prove that \eqref{admiss*} holds.
For any $k \in \{1, \ldots, n-2\}$ and $p \in \{1, \ldots, n\}$, we define $W^k$ and $W^k_p=
(w_p^k, (w_p^k)^{\prime})$ as in \eqref{Wi}. We also set $W^0_p=(w_p^0, (w_p^0)^{\prime})=(u_p,
u_p^{\prime})$. We claim that the following inequality
holds for any $k \in \{1, \ldots, n-2\}$ and any $i \in \{0, \ldots, k-1\}$.
\begin{equation*}\label{claimi}
(R_i^k)  \quad
e_{1-i}(W^{k-i}_{n-k}) \le C \Big(e_{-i}(W^{k-i-1}_{n-k}) + |A^{-(i+1)/2} w^{k-i}_{n-k-1}|^2 \Big)\,.
\end{equation*}
  We have by definition of the energy $e_{1-i}$
\begin{align}
\label{inter4n}
2e_{1-i}(W^{k-i}_{n-k})
&
=|A^{(1-i)/2}w^{k-i}_{n-k}|^2 + |A^{-i/2}(w_{n-k}^{k-i})^{\prime}|^2
\\ 
&
\notag
=
|A^{-i/2}w_{n-k}^{k-i-1}|^2 + |A^{(1-i)/2}w^{k-i}_{n-k}|^2\,.
\end{align}
On the other hand by definition of $W^k$, we have
$$
(w_{n-k}^{k-i-1})^{\prime} + Aw_{n-k}^{k-i} + C_{n-k \, n-k-1} w_{n-k-1}^{k-i} =0\,.
$$
 Using this relation in \eqref{inter4n}, together with the property in $(A_2)_n$
that \\
$  
C_{n-k\, n-k-1}^{\ast} \in \mathcal{L}(H_{i+1}) \,,
$ 
 we deduce  that
$$
e_{1-i}(W^{k-i}_{n-k}) \le C \Big(e_{-i}(W^{k-i-1}_{n-k}) + |A^{-(i+1)/2}w_{n-k-1}^{k-i}|^2\Big)\,,
$$
 so that  $(R_{i}^k)$ holds and thus our claim is proved for all $k \in\{1, \ldots, n-2\}$ and for all  $i \in \{0, \ldots, k-1\}$. 
Using now $(R_1^k)$
to estimate $e_0(W^{k-1}_{n-k})$ in the right hand side of $(R_0^k)$, proceeding recursively up
to $(R_{k-1}^k)$ and using that $W^0_{n-k}=U_{n-k}$,
we deduce that
\begin{equation}\label{inter5n}
e_1(W^k_{n-k}) \le C \Big(e_{1-k}(U_{n-k}) +  \sum_{p=1}^k |A^{-p/2}w_{n-k-1}^{k+1-p}|^2 \Big) \,.
\end{equation}
 We now claim that for all $m \in \{1, \ldots, n-2\}$, we have the following inequality
\begin{align*} \label{Qm}
(Q_{m-1}) \quad \int_0^T |u_{n-1}|^2  dt 
& \le C\Big(
\sum_{i=1}^{m-1} e_{1-i}(U_{n-i})(0) + \sum_{p=1}^{m-1} |A^{-p/2} w_{n-m}^{m-p}(0)|^2   \\
&
+
\int_0^T |w_{n-m}^{m-1}|^2 \,dt
\Big) \,.
\end{align*}
 We shall prove $(Q_m)$ by induction on $m$. Let us first prove $(Q_1)$. We have
\begin{equation}\label{Q1}
\int_0^T |u_{n-1}|^2 \,dt = \int_0^T |(w_{n-1}^1)^{\prime}|^2 \,dt \le C\int_0^T e_1(W^1_{n-1}) \,dt \,.
\end{equation}
 On the other hand, we have
$ 
(w^1_{n-1})^{\prime\prime} + A w_{n-1}^1 + C_{n-1 \, n-2} w^1_{n-2}=0 \,,
$ 
 so that classical energy estimates lead to
$$
\int_0^T e_1(W^1_{n-1}) \,dt \le C \Big(e_1(W^1_{n-1})(0) + \int_0^T |w^1_{n-2}|^2 \,dt\Big) \,.
$$
 Using this last estimate in \eqref{Q1}, we  obtain
$$
 \int_0^T |u_{n-1}|^2 \,dt \le C\Big( e_1(W^1_{n-1})(0) + \int_0^T |w^1_{n-2}|^2 \,dt
\Big) \,.
$$
 Using now the inequality \eqref{inter5n} with $k=1$ in the above inequality, we obtain
$(Q_1)$. Assume now that $(Q_{m-1})$ holds. We shall prove $(Q_{m})$. We have
$$
\int_0^T |w_{n-m}^{m-1}|^2 \,dt = \int_0^T |(w_{n-m}^m)^{\prime}|^2 \,dt \le C\int_0^T e_1(W^{m}_{n-m}) \,dt \,.
$$
 As above for the case $m=1$, we deduce that
$$
\int_0^T |w_{n-m}^{m-1}|^2 \,dt \le C \Big(e_1(W^{m}_{n-m})(0) + \int_0^T |w_{n-m-1}^{m}|^2\,dt\Big) \,.
$$
 Using this last estimate in $(Q_{m-1})$, we obtain
\begin{align*}
\quad \int_0^T |u_{n-1}|^2 \,dt  & \le C\Big(
\sum_{i=1}^{m-1} e_{1-i}(U_{n-i})(0) + \sum_{p=1}^{m-1} |A^{-p/2} w_{n-m}^{m-p}(0)|^2 
\\
&
+ e_1(W^{m}_{n-m})(0) + \int_0^T |w_{n-m-1}^{m}|^2\,dt
\Big) \,.
\end{align*}
 We then use \eqref{inter5n} with $k=m$ in the above inequality. This gives
\begin{eqnarray} \label{Qm1} 
\int_0^T |u_{n-1}|^2 \,dt \le C\Big(
\sum_{i=1}^{m} e_{1-i}(U_{n-i})(0)  +
\int_0^T |w_{n-m-1}^{m}|^2\,dt  \nonumber \\
+  \sum_{p=1}^m |A^{-p/2}w_{n-m-1}^{m+1-p}(0)|^2
 + \sum_{p=1}^{m-1} |A^{-p/2} w_{n-m}^{m-p}(0)|^2
\Big) \,.
\end{eqnarray}
 It remains to estimate the last term of this inequality to conclude that $(Q_m)$ holds. 
Let $p \in \{1, \ldots, m-1\}$. We first remark that by the definition of $W_{n-m}^{m-p+1}$, we have
\begin{equation}\label{Qm2}
 |A^{-p/2} w_{n-m}^{m-p}(0)|^2 \le C e_{1-p}(W_{n-m}^{m-p+1})(0) \,.
\end{equation}
We prove in a similar
way to the proof of $(R_i^k)$, that for all $k \in\{1, \ldots, n-2\}$ and for all  $i \in \{0, \ldots, k-1\}$, we have
\begin{equation*}\label{claimii}
(S_i^k) \quad
e_{1-i}(W^{k-i+1}_{n-k}) \le C \Big(e_{-i}(W^{k-i}_{n-k}) + |A^{-(i+1)/2} w^{k-i+1}_{n-k-1}|^2 \Big)\,,
\end{equation*}
  replacing $W_{n-k}^{k-i}$ by $W_{n-k}^{k-i+1}$ in the proof of $(R_i^k)$.

Using $(S_{p+1}^m)$
to estimate $e_{-p}(W^{m-p}_{n-m})$ in the right hand side of $(S_p^m)$, proceeding recursively up
to $(S_m^m)$ and using that $W^0_{n-m}=U_{n-m}$,
we deduce that
\begin{equation*}\label{inter6n}
e_{1-p}(W^{m-p+1}_{n-m}) \le C \Big(e_{-m}(U_{n-m}) +  \sum_{j=p-1}^{m-1} |A^{-(j+2)/2}w_{n-m-1}^{m-j}|^2 \Big) \,.
\end{equation*}
 From this last inequality together with \eqref{Qm2}, and the
inequality $ e_{-m}(U_{n-m}) $ $\le C e_{1-m}(U_{n-m}) , $
 where $C$ depends on $||A^{-1/2}||$, we deduce that
\begin{equation*}\label{inter7n}
\sum_{p=1}^{m-1}  |A^{-p/2} w_{n-m}^{m-p}(0)|^2 \le C \Big(e_{1-m}(U_{n-m}) +  \sum_{p=1}^{m} |A^{-(p+1)/2}w_{n-m-1}^{m-p+1}|^2 \Big) \,.
\end{equation*}
 Using this inequality in \eqref{Qm1}, we deduce that $(Q_{m})$ holds. 

We can now conclude the proof of \eqref{admiss*} as follows. We choose $m=n-2$. 
Thus from the above result, $(Q_{n-2})$ holds,  which means
\begin{align}
 \label{inter8n}
(Q_{n-2}) \qquad 
\int_0^T |u_{n-1}|^2 \,dt & \le C\Big(
\sum_{i=1}^{n-2} e_{1-i}(U_{n-i})(0)  \\
&
\notag 
+ \sum_{p=1}^{n-2} |A^{-p/2} w_{1}^{n-1-p}(0)|^2 +
\int_0^T |w_{1}^{n-2}|^2 \,dt
\Big) \,.
\end{align}
 We will estimate the last two terms of this inequality. By definition of $W_1^{n-1-p}$, we have
$ 
(w_1^{n-2-p})^{\prime} + Aw_1^{n-1-p}=0  , $ $ (w_1^{n-2-p})^{\prime}=w_1^{n-3-p}  .
$ 
 Hence, we have
$ 
 |A^{-p/2} w_{1}^{n-1-p}(0)|^2=|A^{-1-p/2}w_1^{n-3-p}(0)|^2 \,. 
$ 
 Proceeding recursively, we deduce that for all integers $i$ such that $2i+1 \! \le n-p$, we have
$ 
 |A^{- \frac p 2} w_{1}^{n-1-p}(0)|^2 $ $ =  |A^{-p/2-i}w_{1}^{n-1-(2i+1)}(0)|^2\,. 
$ 
 If $n-p$ is even we choose $i=(n-p)/2 -1$, whereas if $n-p$ is odd, we choose
$i=(n-p-1)/2$. Thus, if $n-p$ is even we have
$ 
|A^{-p/2} w_{1}^{n-1-p}(0)|^2=|A^{-n/2}u_1^{\prime}|^2 \le C e_{2-n}(U_1)(0) \,,
$ 
 whereas if $n-p$ is odd, we have
$$ 
|A^{-p/2} w_{1}^{n-1-p}(0)|^2=|A^{(1-n)/2}u_1|^2 \le C e_{2-n}(U_1)(0) \,.
$$ 
 Thus, we have
\begin{equation}\label{inter9n}
\sum_{p=1}^{n-2} |A^{-p/2} w_{1}^{n-1-p}(0)|^2 \le C e_{2-n}(U_1)(0) \,.
\end{equation}
 We now estimate the last term of \eqref{inter8n} as follows. If $n$ is even we have
$|w_1^{n-2}|^2=|A^{(2-n)/2}u_1|^2$, whereas if $n$ is odd we have $|w_1^{n-2}|^2=|A^{(1-n)/2}u_1^{\prime}|^2$. Hence, for the two cases, we have
$ 
|w_1^{n-2}(t)|^2 \le C e_{2-n}(U_1)(t) \,, \, \forall \ t \ge 0\,.
$ 
 Thanks to the conservation of the energy $e_{2-n}(U_1)$, we have
$$
\int_0^T |w_1^{n-2}|^2\,dt \le C e_{2-n}(U_1)(0)\,,
$$
 where the constant $C$ depends on $T$. Using this last estimate together with \eqref{inter9n} in \eqref{inter8n}, we obtain
\eqref{admiss*}. Using this last inequality in \eqref{inter2n}, we obtain \eqref{admissineqi}.
Using classical energy estimates for the equation for $u_n$ in \eqref{NSHn}, we have
$$
\int_0^T|u_n|^2\,dt \le C(e_1(U_n)(0) + \int_0^T |u_{n-1}|^2\,dt) \,.
$$
 Using \eqref{admiss*} in this last estimate, we obtain \eqref{admissx*}.
This concludes the proof.

\smallskip

\noindent
\textbf{Proof of Theorem~\ref{obsNSHn}.}
The proof is proved by induction on $n$. However, \eqref{eqobsk} is not sufficient for a proof by induction, one needs to combine it with other properties to be proved by induction. This is the purpose of the next result. We shall then prove Theorem~\ref{obsNSHn}. Let $n \ge 2$ be given. We consider the following property at order $n$:
\begin{equation*}
(\mathcal{P}_n) 
\begin{cases}
\exists \ K_n>0\,, 
\exists\  T_n^{\ast} >0 \,, 
\mbox{ such that } \forall \ T >T_n^{\ast} , \ 
\exists \ r_{n,n}(T)>0
\mbox{ and } \\ 
 \forall \ i\in \{1,\ldots,n\}, \,  \exists \, d_{i,n}(T)>0  ,  k_{i,n}(T)>0  
\mbox{ such that for all solutions } \\
 W=(w_1, \ldots, w_n, w_1^{\prime}, \ldots, w^{\prime}_n) \mbox{ of }  
W^{\prime}=\mathcal{A}_nW\mbox{ one has }\\
e_{1+i-n}(W_i)(0) \le d_{i,n}(T) \int_0^T ||\mathcal{\mathbf{B^{\ast}_n}}(W_n) \|_{G_{n}}^2 dt \quad
\forall \ i \in \{1,\ldots, n\}  \,, \quad \label{H-Rn}
\\
e_{0}(W_{n-1})(T) \le d_{n-1,n}(T) \int_0^T ||\mathcal{\mathbf{B^{\ast}_n}}(W_n) \|_{G_{n}}^2 dt \,,\\
\int_0^T \langle C_{n n-1} w_{n-1}, w_{n-1}\rangle \,dt \le r_{n,n}(T)  \int_0^T ||\mathcal{\mathbf{B^{\ast}_n}}(W_n) \|_{G_{n}}^2 dt \,,
 \\
\int_0^Te_{1+i-n}(W_i)(t) \,dt \le k_{i,n}(T) \int_0^T ||\mathcal{\mathbf{B^{\ast}_n}}(W_n) \|_{G_{n}}^2 dt \ 
\forall \ i \in \{1,\ldots, n\} \,, \\
\mbox{ where } r_{n,n}(T) \le K_n/T^2 \,, d_{n,n}(T) \le K_n/T 
\mbox{ and } k_{n,n}(T) \le K_n , \\
d_{i,n}(T) \le K_n/T^3 
\mbox{ and } k_{i,n}(T) \le K_n/T^2 \mbox{ for $ i \in \{1, \ldots, n-1\} $ and}  , \\
\mbox{where } K_n \mbox{ is a generic constant which depends on } n \mbox{ but not on } T \,.
 \end{cases}
\end{equation*}
\begin{Theorem}\label{inductionobsNSHn}
Let $n \ge 2$ be an integer.  We assume  
that for all $i \in \{2, \ldots,n+1\}$,
the operators $C_{i i-1}$ satisfy the assumption $(A2)_{n+1}$ where the operators
$\Pi_i$ satisfy $(A3)_{n+1}$.  
We assume that for all $k \in \{2, \ldots, n\}$ the property $(\mathcal{P}_k)$ holds for any observation operator $\mathcal{\mathbf{B^{\ast}_k}}$, satisfying $(A4)_{k}-(A5)_{k}$. Then
the property $(\mathcal{P}_{n+1})$ also holds for any observation operator $\mathcal{\mathbf{B^{\ast}_{n+1}}}$  satisfying $(A4)_{n+1}-(A5)_{n+1}$.
\end{Theorem}

\begin{Remark}\label{obs_cons_T}
\rm
A key point in the multi-level (as well as in the two-level) energy method is to derive observability inequalities for any sufficiently large time $T$ with observability constants which depend in a suitable way on $T$ as stated in $(\mathcal{P}_n)$.  A first step for this is to derive suitable observability estimates for a scalar nonhomogeneous wave equation. In particular, it is crucial to obtain observability constants which are uniform with respect to $T$ for any sufficiently large time $T$. 
This first crucial step has been proved in Lemma 3.3 of \cite{alaleaucont} (we recall this Lemma here as Lemma~\ref{AL} with some minor changes). This Lemma gives observability constants $\eta_0$ and $\alpha_0$ which do not depend on $T$ (for sufficiently large $T$) as a consequence of $(A5)_n$ and the admissibility property $(A4)_n$. This holds even though the observability constants in the assumption $(A5)_n$ depend on an unknown way on the (sufficiently large) time $T$. This result is based on the invariance of the solutions with respect to time translations. Using involved estimates, one can then prove inductively the observability estimates stated in $(\mathcal{P}_n)$
with observability constants which depend on sufficiently large time $T$ with the given explicit  $T$-dependence.
\end{Remark}
Theorem 2.7 in~\cite{alabau2cascade} proves $(\mathcal{P}_n)$ for $n=2$. Let us assume that $(\mathcal{P}_n)$ holds. We shall prove that $(\mathcal{P}_{n+1})$ holds by several intermediate propositions.
We shall keep the same notation for these propositions and their proofs.
\begin{Proposition}\label{PROP1}
Assume the hypotheses of Theorem~$\ref{inductionobsNSHn}$ and that $(\mathcal{P}_n)$ holds.
Then for every $T>T_n^{\ast}$, the following estimate holds
\begin{align}
 \label{EQ13}
&
\int_0^T \langle C_{n+1 n} u_{n}, u_n \rangle \,dt  \le
\frac{\beta_n r_{n,n}(T) \eta T}{2 \alpha_{n+1}} \int_0^T \langle C_{n+1 n} u_{n}, u_n \rangle \,dt  \\
&
\notag 
+
\frac{1}{\eta T} \int_0^T e_1(U_{n+1})(t)\,dt + \frac{c}{\eta} \Big(e_1(U_{n+1})(T)+
e_1(U_{n+1})(0)\Big) \\
&
\notag + 
c\eta\big (e_1(W_{n})(T)+ e_1(W_{n})(0)\big) + c\eta (e_0(W_{n-1})(T)+ e_0(W_{n-1})(0)\Big) \  \forall \ \eta>0 .
\end{align}
\end{Proposition}

\begin{proof}[\bf Proof.]
Let $U^{n+1}=(u_1,\ldots, u_{n+1}, u_1^{\prime}, \ldots, u_{n+1}^{\prime}) \in D(\mathcal{A}_{n+1})$ be a solution of
$ 
(U^{n+1})^{\prime}=\mathcal{A}_{n+1}U^{n+1} \,.
$ 
  We note that  $U^n=(u_1,\ldots, u_{n}, u_1^{\prime}, \ldots, u_{n}^{\prime}) \in D(\mathcal{A}_{n})$
is solution of
$ 
(U^{n})^{\prime}=\mathcal{A}_{n}U^{n} \,.
$ 
We set $W^{n+1}=\mathcal{A}_{n+1}^{-1}U^{n+1}$. Then $W^{n+1}=(w_1,\ldots, w_{n+1}, w_1^{\prime}, \ldots, w_{n+1}^{\prime})
 \in D(\mathcal{A}_{n+1})$ and
 $W^n=(w_1,\ldots, w_{n}, w_1^{\prime}, \ldots, w_{n}^{\prime})$ $ \in D(\mathcal{A}_{n})$ and
 $W^{n+1}$, $W^n=W$ are respective solutions of
 $ 
(W^{n+1})^{\prime}=\mathcal{A}_{n+1}W^{n+1}\,,
$ 
 and
$ 
(W^{n})^{\prime}=\mathcal{A}_{n}W^{n} \,.
$ 
  By definition of $W=W^n$, we have
the relations
$ \label{uiprime}
w_i^{\prime}=u_i \,, i=1, \ldots, n+1 \,.
$
We set $W_n=(w_n, w_n^{\prime})$. We define an observation operator
$\mathcal{\mathbf{B^{\ast}_n}}$
by
$ \label{Bnstar}
\mathcal{\mathbf{B^{\ast}_n}}W_n=\Pi_{n+1}w_n^{\prime}=\Pi_{n+1}u_n \,.
$
 Thus, we have
\begin{equation}\label{EQ1}
\int_0^T |\Pi_{n+1}u_n|^2 \,dt \le \frac{1}{\alpha_{n+1}} \int_0^T \langle C_{n+1 n} u_n,u_n \rangle \,dt \,.
\end{equation} 
The operator $\mathcal{\mathbf{B^{\ast}_n}}$ satisfies $(A4)_n$ and $(A5)_n$.
We apply the property $(\mathcal{P}_n)$ to $W$, solution of $W^{\prime}=\mathcal{A}_{n}W$
with this choice of observation operator $\mathcal{\mathbf{B^{\ast}_n}}$. This, together
with \eqref{EQ1} imply
\begin{equation}\label{EQ2}
e_{1+i-n}(W_i)(0) \le \frac{d_{i,n}(T)}{\alpha_{n+1}} \int_0^T\langle C_{n+1 n} u_n,u_n \rangle \,dt \,,
\quad i \in \{1, \ldots, n\}\,,
\end{equation}
\begin{equation}\label{EQT}
e_0(W_{n-1})(T) \le \frac{d_{n-1,n}(T)}{\alpha_{n+1}} \int_0^T\langle C_{n+1 n} u_n,u_n \rangle \,dt \,,
\end{equation}
\begin{equation}\label{EQ3}
\int_0^T\langle C_{n n -1} w_{n-1},w_{n-1} \rangle \,dt
 \le \frac{r_{n,n}(T)}{\alpha_{n+1}} \int_0^T\langle C_{n+1 n} u_n,u_n \rangle \,dt 
 \,,
\end{equation}
\begin{equation}\label{EQ4}
\int_0^Te_{1+i-n}(W_i)(t)\,dt \le \frac{k_{i,n}(T)}{\alpha_{n+1}}
 \int_0^T\langle C_{n+1 n} u_n,u_n \rangle \,dt  ,\  i \in \{1, \ldots, n\}\,.
\end{equation}
 Since $\mathcal{\mathbf{B^{\ast}_{n+1}}}$
satisfies $(A4)_{i}-(A5)_{i}$ for $i=n+1$, we can apply the uniform estimate \eqref{obsabs} to
$u_{n+1}$ solution of
$$
u_{n+1}^{\prime\prime} + A u_{n+1} + C_{n+1 n} u_{n}=0\,.
$$
 This, together with $(A2)_{n+1}$ give
\begin{equation}\label{EQ5}
\eta_0 \int_0^T ||\mathcal{\mathbf{B^{\ast}_{n+1}}}U_{n+1}||_{G_{n+1}}^2 \ge
\int_0^T e_1(U_{n+1}(t)\,dt - \alpha_0 \beta_{n+1} \int_0^T \langle C_{n+1 n} u_{n}, u_n \rangle \,dt \,.
\end{equation}
 On the other hand, thanks to \eqref{NSHn} for $i=n-1$ and $i=n$, and using the relation $u_{n-1}=w_{n-1}^{\prime}$ with appropriate integration by parts with respect to time we obtain
\begin{align}
 \label{EQ6}
& 
\int_0^T \langle C_{n+1 n} u_{n}, u_n \rangle \,dt=  
\Big[\langle A^{-1/2}u_n^{\prime},  A^{1/2}u_{n+1}
\rangle - \langle u_n, u_{n+1}^{\prime} \rangle  \\
&
\notag
\qquad\qquad + \langle C_{n n-1}w_{n-1}, u_{n+1} \rangle
\Big]_0^T  
 - \int_0^T \langle C_{n n-1}w_{n-1}, u_{n+1}^{\prime}  \rangle \,. 
\end{align}
 We estimate the right hand side of \eqref{EQ6} using $(A2)_n$. This gives
for all $\eta>0$
\begin{align*}
&
\int_0^T \langle C_{n+1 n} u_{n}, u_n \rangle \,dt \\
&
 \le 
\frac{\beta_n \eta T}{2} \int_0^T \langle C_{n n-1} w_{n-1}, w_{n-1} \rangle \,dt +
\frac{1}{\eta T} \int_0^T e_1(U_{n+1})(t)\,dt \\
&
+ 
\frac{\eta}{2}\Big(|w_{n-1}(T)|^2 \! + |w_{n-1}(0)|^2\Big) + \frac{\beta_n^2}{2 \eta}||A^{- \frac 12 }||
\Big(|A^{ \frac 12 }u_{n+1}(T)|^2 + |A^{ \frac 12 }u_{n+1}(0)|^2\Big) \\
&
+\eta\Big((e_0(U_{n})(T)+ e_0(U_{n})(0)\Big) + \frac{1}{\eta} \Big(e_1(U_{n+1})(T)+
e_1(U_{n+1})(0)\Big) \,.
\end{align*}
 Using \eqref{EQ3} in the above inequality, we obtain
\begin{align}\label{EQ9}
& 
\int_0^T \langle C_{n+1 n} u_{n}, u_n \rangle \,dt  \le
\frac{\beta_n r_{n,n}(T) \eta T}{2 \alpha_{n+1}} \int_0^T \langle C_{n+1 n} u_{n}, u_n \rangle \,dt 
\notag
\\
&
+
\frac{1}{\eta T} \int_0^T e_1(U_{n+1})(t)\,dt + \frac{c}{\eta} \Big(e_1(U_{n+1})(T)+
e_1(U_{n+1})(0)\Big)  \\
&
\notag 
\eta\Big(e_0(U_{n})(T)+ e_0(U_{n})(0)\Big) + \eta (e_0(W_{n-1})(T)+ e_0(W_{n-1})(0)\Big) \,,
\end{align}
 where $c>0$ is a generic constant which is independent of $T$.
Since $W^n=\mathcal{A}_n^{-1}U^n$, we have
$u_n=w_n^{\prime}$ and $u_n^{\prime}=-Aw_n -C_{n n-1}w_{n-1}$. Therefore,
we have for any $t \ge 0$
\begin{equation}\label{EQ11}
e_0(U_n)(t) \le c\Big(e_0(W_{n-1})(t) + e_1(W_n)(t)\Big) \,.
\end{equation}
 Using this inequality for $t=0$ and $t=T$ in \eqref{EQ9}, we obtain
\eqref{EQ13}.
\end{proof}
 We shall prove several intermediate results to estimate the right hand side of
\eqref{EQ9}.

\begin{Proposition}\label{PROP2}
Assume the hypotheses of Theorem~$\ref{inductionobsNSHn}$ and that $(\mathcal{P}_n)$ holds.
Then for every $T>T_n^{\ast}$, the following estimates hold for all $\eta>0$
\begin{align} \label{EQ14}
&
e_1(W_n)(T)+e_1(W_n)(0)  \\
&
\notag 
\le 
\frac{1}{\alpha_{n+1}}\Big(2 d_{n,n}(T) + \frac{\beta_n T r_{n,n}(T)}
{2} + \frac{ k_{n,n}(T)}{T}\Big) \int_0^T \langle C_{n+1 n} u_{n}, u_n \rangle \,dt \,,
\end{align}
\begin{align} \label{EQ15}
e_0(W_{n-1})(T) +e_0(W_{n-1})(0)   \le
\frac{1}{\alpha_{n+1}}\Big(2 d_{n-1,n}(T)\Big) \int_0^T
 \langle C_{n+1 n} u_{n}, u_n \rangle \,dt \,,
\end{align}
\begin{align} \label{EQ15b}
e_1(U_{n+1})(T)+e_1(U_{n+1})(0) 
&  \le 2 e_1(U_{n+1})(0) +  
\frac{\beta_{n+1} \eta^2}
{2T} \int_0^T \langle C_{n+1 n} u_{n}, u_n \rangle \,dt  
\notag
\\ 
&
+ \frac{T}{\eta^2}\int_0^Te_1(U_{n+1})(t)\,dt \,.
\end{align}
\end{Proposition}

\begin{proof}[\bf Proof.]
Since we have
$ 
w_n^{\prime\prime}+Aw_n + C_{n n-1} w_{n-1}=0 \,,
$ 
 we deduce that 
$ 
e_1^{\prime}(W_n)(t)$ $=-\langle C_{n n-1} w_{n-1}, w_n^{\prime}\rangle \,,
$ 
 so that
\begin{align*}
&
e_1(W_n)(T) + e_1(W_n)(0)=2 e_1(W_n)(0)
 - \int_0^T \langle C_{n n-1} w_{n-1}, w_n^{\prime}\rangle\\
&
 \le  
2 e_1(W_n)(0) +\frac{T\beta_n}{2} \int_0^T \langle C_{n n-1} w_{n-1}, w_{n-1} \rangle +
\frac{1}{T}\int_0^T e_1(W_n) (t)\,dt \,.
\end{align*}
 We use respectively \eqref{EQ2}, \eqref{EQ3} and \eqref{EQ4} with $i=n$, this gives \eqref{EQ14}.

The proof of \eqref{EQ15b} is similar using the equation for $u_{n+1}$.

We use \eqref{EQ2} with $i=n-1$ and \eqref{EQT}. This gives \eqref{EQ15}.
\end{proof}

\begin{Proposition}\label{PROP3}
Assume the hypotheses of Theorem~$\ref{inductionobsNSHn}$ and that $(\mathcal{P}_n)$ holds.
Then for every $T>T_n^{\ast}$, the following estimate holds
\begin{equation}\label{EQ21}
\int_0^T \langle C_{n+1 n} u_{n}, u_n \rangle \,dt  \le
\frac{K_n}{T^2} \int_0^T e_1(U_{n+1})(t)\,dt + \frac{J_n}{T}e_1(U_{n+1})(0) \,,
\end{equation}
 where $K_n\,, J_n$ are generic constants which depend on $n$ but not on $T$.
\end{Proposition}

\begin{proof}[\bf Proof.]
Using \eqref{EQ14}, \eqref{EQ15} and \eqref{EQ15b} in \eqref{EQ13} we obtain
\begin{align}
\label{EQ17}
&
\Big(1-\eta\theta_n(T)\Big)\int_0^T \langle C_{n+1 n} u_{n}, u_n \rangle \,dt  \\
&
\notag
\le
\Big(\frac{1}{\eta T} + \frac{c T}{\eta^3}\Big)\int_0^T e_1(U_{n+1})(t)\,dt +  
\frac{c}{\eta}e_1(U_{n+1})(0)
\,,
\end{align}
 where $c$ is a generic positive constant which does not
depend on $T$ nor $\eta$ and where the coefficient $\theta_n(T)$ is given by
\begin{align*}
\label{EQ18}
\theta_n(T) & =\frac{\beta_n r_{n,n}(T)T}{2 \alpha_{n+1}} + \frac{c \beta_{n+1}}{2T}+
\frac{c}{\alpha_{n+1}}
\Big(2 d_{n.n}(T) +\frac{T \beta_n r_{n.n}(T)}{2} \\
&
+ \frac{k_{n,n}(T)}{T}\Big)+ 
\frac{c}{\alpha_{n+1}}
\Big(2 d_{n-1.n}(T) \Big)\,.
\end{align*}
 We choose $\eta$ such that
$ \label{EQ20}
\eta=\frac{1}{2 \theta_n(T)}=\eta_n(T) \,. 
$
 Thanks to this choice, the definition of $\theta_n(T)$ and our induction hypothesis
$(\mathcal{P}_n)$, we deduce that
\begin{equation}\label{EQ19}
\frac{1}{\eta_n(T)} \le \frac{K_n}{T} \,.
\end{equation}
  Therefore, using \eqref{EQ19} in \eqref{EQ17}, we
deduce \eqref{EQ21}.
\end{proof}

\begin{Lemma}\label{PROP4}
Assume the hypotheses of Theorem~$\ref{inductionobsNSHn}$
 and that $(\mathcal{P}_n)$ holds. Then for every $T>T_n^{\ast}$,
\begin{equation}\label{EQ27}
\int_0^T e_1(U_{n+1})(t) \ge M_n\,T e_1(U_{n+1})(0) \,,
\end{equation}
where $M_n$ does not depend on $T$ and is defined by
\begin{equation}\label{Mna}
M_n=\frac{\sqrt{a_n^2+a_n+b_n}}{(2a_n+1)\big(
a_n+\sqrt{a_n^2+a_n+b_n}\big)+a_n+2b_n} \,,
\end{equation}
 with
\begin{equation}\label{an}
a_n=\frac{\beta_{n+1} J_n}{2} \,,
\end{equation}
 and
\begin{equation}\label{bn}
b_n=\frac{\beta_{n+1} K_n}{2} \,.
\end{equation}
\end{Lemma}

\begin{proof}[\bf Proof.]
The proof is similar to that of Lemma~2.14 in~\cite{alabau2cascade} in the case $n=2$. For the sake of clarity, we indicate briefly the arguments. We have
$$
e_1^{\prime}(U_{n+1})(t)=-\langle C_{n+1 n} u_{n}, u_{n+1}^{\prime}\rangle \,.
$$
 Integrating twice this equation and using assumption $(A2)_n$, we have for all $\nu>0$
$$
(1+\nu)\int_0^T e_1(U_{n+1})(t) \ge Te_1(U_{n+1})(0)
- \frac{\beta_{n+1} T^2}{2\nu} \int_0^T\langle C_{n+1 n}u_n,u_n\rangle \,.
$$
 We choose $\nu= a_n + \sqrt{a_n^2+a_n+b_n}$. Using \eqref{EQ21} in this last estimate
and the definition of $a_n$, $b_n$ and $\nu$, we obtain \eqref{EQ27} where
$M_n$
is defined by \eqref{Mna}.
\end{proof}

\begin{Proposition}\label{PROP5}
Assume the hypotheses of Theorem~$\ref{inductionobsNSHn}$ and that $(\mathcal{P}_n)$ holds. 
Then $(\mathcal{P}_{n+1})$ holds.

\end{Proposition}

\begin{proof}[\bf Proof.]
We choose $T_{n+1}^{\ast}>\max(T_n^{\ast}, T_n)$ where $T_n^2$ is given by
\begin{equation*}\label{EQTn}
T_n^2= \alpha_0\beta_{n+1}\Big(K_n+ \frac{J_n}{M_n}\Big) \,.
\end{equation*}
 Let $T>T_{n+1}^{\ast}$. In the sequel of this proof, we will denote by $H_{l,n+1}$ for $l=1, \ldots, 6$  generic constants which
do not depend on $T$.

Thanks to \eqref{EQ5} and \eqref{EQ21} and \eqref{EQ27} we have
\begin{equation}\label{EQ28}
\int_0^Te_1(U_{n+1})(t) \le  k_{n+1,n+1}(T)\int_0^T ||\mathcal{\mathbf{B^{\ast}_{n+1}}}(U_{n+1}) \|_{G_{n+1}}^2 dt \,,
\end{equation}
 where 
\begin{equation*}\label{kn+1}
k_{n+1,n+1}(T)= \frac{\eta_0 T^2}{T^2 -T_n^2} \le \frac{\eta_0 T_{n+1}^{\ast \,2}}{T_{n+1}^{\ast \,2} -T_n^2}=H_{1,n+1} \,.
\end{equation*}
 Using now \eqref{EQ28} in \eqref{EQ27}, we obtain
\begin{equation}\label{EQ29}
e_1(U_{n+1})(0) \le  d_{n+1,n+1}(T)\int_0^T ||\mathcal{\mathbf{B^{\ast}_{n+1}}}(U_{n+1}) \|_{G_{n+1}}^2 dt \,,
\end{equation}
 where 
\begin{equation*}\label{dn+1}
d_{n+1,n+1}(T)= \frac{\eta_0 T}{M_n(T^2 -T_n^2)} \le \frac{\eta_0 T_{n+1}^{\ast \,2}}{M_n(T_{n+1}^{\ast \,2} -T_n^2)}\frac{1}{T} =\frac{H_{2,n+1}}{T}\,.
\end{equation*}
 Using similarly \eqref{EQ28} and \eqref{EQ29} in \eqref{EQ21}, we have
\begin{equation}\label{EQ30}
\int_0^T \langle C_{n+1 n} u_{n}, u_n \rangle \,dt  \le  r_{n+1,n+1}(T) \int_0^T ||\mathcal{\mathbf{B^{\ast}_{n+1}}}(U_{n+1}) \|_{G_{n+1}}^2 dt \,,
\end{equation}
 where 
\begin{align*}  \label{rn+1}
r_{n+1,n+1}(T) &
= \frac{\eta_0}{(T^2 -T_n^2)}\Big(K_n+ \frac{J_n}{M_n}\Big)  \\
&
 \le  
\frac{\eta_0 T_{n+1}^{\ast \,2}}{M_n(T_{n+1}^{\ast \,2} -T_n^2)}\Big(K_n+ \frac{J_n}{M_n}\Big) \frac{1}{T^2}=
\frac{H_{3,n+1} }{T^2}\,.
\end{align*}
 Let $i$ be any integer in $\{2,\ldots, n\}$. We estimate $e_{i-n}(U_i)$ as follows. Thanks to the definition of $W_i$,  we have
\begin{align*}
2 e_{i-n}(U_i)  & =|A^{i-n}w_i^{\prime}|^2 + |A^{(i-n-1)/2}u_i^{\prime}|^2 \\
&
\le 
C \Big(e_{1+i-n}(W_i) + |A^{(i-n-1)/2}C_{i\,i-1}w_{i-1}|^2\Big) \,.
\end{align*}
 On the other hand, thanks to $(A2)_n$ we have
$$
|A^{(i-n-1)/2}C_{i\,i-1}w_{i-1}|^2 \le C |A^{(i-n)/2}w_{i-1}|^2 \,.
$$
 Hence, we have
\begin{equation}\label{ein}
e_{i-n}(U_i)\le
C \Big(e_{1+i-n}(W_i) + e_{i-n}(W_{i-1})\Big) \,.
\end{equation}
 Using \eqref{EQ30} in \eqref{EQ2} for $i$ and $i-1$, and inserting the resulting estimate
in \eqref{ein}, we obtain
\begin{equation*}\label{EQn+1}
e_{i-n}(U_i)(0) \le d_{i,n+1}(T) \int_0^T ||\mathcal{\mathbf{B^{\ast}_{n+1}}}(U_{n+1}) \|_{G_{n+1}}^2 dt \,,
\end{equation*}
 where 
$$ \label{din+1}
d_{i,n+1}(T) \! = \! C\Big( \frac{d_{i-1,n}(T) \! + d_{i,n}(T)}{\alpha_{n+1}}\Big) r_{n+1,n+1}(T)  
 \! \le  \! \frac{2C K_n H_{3,n+1}}{\alpha_{n+1}}\frac{1}{T^3}  \! \le  \! \frac{H_{4,n+1}}{T^3} .
$$
 Using similarly \eqref{EQ30} in \eqref{EQ4} for $i$ and $i-1$, and inserting the resulting estimate
in \eqref{ein}, we obtain
\begin{equation*}\label{EQint}
\int_0^T e_{i-n}(U_i)(t)\,dt  \le k_{i,n+1}(T) \int_0^T ||\mathcal{\mathbf{B^{\ast}_{n+1}}}(U_{n+1}) \|_{G_{n+1}}^2 dt \,,
\end{equation*}
 where 
$$
 \label{kin+1}
k_{i,n+1}(T)=C\Big( \frac{k_{i-1,n}(T)+ k_{i,n}(T)}{\alpha_{n+1}}\Big) r_{n+1,n+1}(T)  
 \le \frac{2 C K_n H_{3,n+1}}{\alpha_{n+1} T^2} \le \frac{H_{5,n+1}}{T^2}
.
$$
 We now remark that
$ 
e_{1-n}(U_1)=e_{2-n}(W_1) \,.
$ 
 Thus, we easily deduce using \eqref{EQ30} in respectively \eqref{EQ2} and \eqref{EQ4}
with $i=1$ that
\begin{equation*}\label{EQe1}
e_{1-n}(U_1)(0) \le d_{1,n+1}(T) \int_0^T ||\mathcal{\mathbf{B^{\ast}_{n+1}}}(U_{n+1}) \|_{G_{n+1}}^2 dt \,,
\end{equation*}
 where 
\begin{equation*}\label{d1n+1}
d_{1,n+1}(T)=\frac{C\,d_{1,n}(T)}{\alpha_{n+1}} r_{n+1,n+1}(T) \le
\frac{H_{4,n+1}}{T^3}\,,
\end{equation*}
 and
\begin{equation*}\label{EQ1int}
\int_0^T e_{1-n}(U_1)(t)\,dt  \le k_{1,n+1}(T) \int_0^T ||\mathcal{\mathbf{B^{\ast}_{n+1}}}(U_{n+1}) \|_{G_{n+1}}^2 dt \,,
\end{equation*}
 where 
\begin{equation*}\label{ke1n+1}
k_{1,n+1}(T)=\frac{C\,k_{1,n}(T)}{\alpha_{n+1}}r_{n+1,n+1}(T) 
\le \frac{H_{5,n+1}}{T^2}\,.
\end{equation*}
 We finally estimate $e_0(U_n)(T)$ as follows. Using \eqref{EQ30} in \eqref{EQT} and \eqref{EQ14} and using the resulting estimates in \eqref{EQ11}, we deduce
that
\begin{equation*}\label{EQUn}
e_0(U_n)(T) \le d_{n,n+1}(T) \int_0^T ||\mathcal{\mathbf{B^{\ast}_{n+1}}}(U_{n+1}) \|_{G_{n+1}}^2 dt \,,
\end{equation*}
 where 
\begin{align*} \label{dnn+1}
d_{n,n+1}(T) & =   
C\Big(d_{n-1,n}(T) + 2d_{n,n}(T) + \frac{\beta_n T r_{n,n}(T)}
{2}  \\
&
+ \frac{ k_{n,n}(T)}{T}\Big)\frac{r_{n+1,n+1}(T)}{\alpha_{n+1}} \le
\frac{H_{6,n+1}}{T^3}\,.
\end{align*}
 Thus we prove that $(\mathcal{P}_{n+1})$ holds with $\displaystyle{K_{n+1}=
\max_{1\le l \le 6}(H_{l,n+1})}$. 
\end{proof}

This concludes the proof of Theorem~\ref{inductionobsNSHn}.
We can now prove Theorem~\ref{obsNSHn}.

\smallskip

\noindent
\textbf{Proof of Theorem~\ref{obsNSHn}}. 
This is an easy consequence of $(\mathcal{P}_{n})$.
We shall need the following technical result for the proof of Theorem~\ref{nNEC}.
 \begin{Lemma}\label{tech1n}
 Assume the hypotheses $(A1)$ and $(A2)_n$.  Then there exist $ C_1>0\,, C_2>0$ such that
 for all $W_0 \in \mathcal{H}_n$,  the following properties hold for 
 $W$ the solution of \eqref{NSHn} with initial data $W_0$ and for $Z=\mathcal{A}_n^{-1}W$  
 \begin{equation}\label{equivxN}
 C_1 \displaystyle{\sum_{i=1}^n e_{i-n}(W_i)} \le \displaystyle{\sum_{i=1}^n e_{1+i-n}(Z_i)} \le
 C_2 \displaystyle{\sum_{i=1}^n e_{i-n}(W_i)}\,.
 \end{equation}
 \end{Lemma}

\begin{proof}[\bf Proof.]
 By definition of $Z$ we have $z_i^{\prime}=w_i$ for $i=1, \ldots, n$, and
 $w_i^{\prime} + Az_i + C_{i i-1}z_{i-1}=0$.  Let $ i \in \{1, \ldots,n\}$ be given arbitrarily. 
 Thanks to  assumption $(A2)_n$, we have
 $ 
 e_{1+i-n}(Z_i) \le C\big(e_{i-n}(W_i) + e_{i-n}(Z_{i-1}) 
 \big) .
 $ 
  Moreover, we have $e_{2-n}(Z_1)=e_{1-n}(W_1)$. Thanks to these two properties, we
 easily prove by induction on $j$ that
 $$
 e_{1+i-n}(Z_i) \le C \displaystyle{\sum_{j=1}^i e_{j-n}(W_j)} \,,
 $$
  which in turn implies that the second inequality in \eqref{equivxN} holds. In a similar way, we have
$ 
e_{i-n}(W_i) \le C \big(e_{1+i-n}(Z_i) + e_{i-n}(Z_{i-1})
\big) \,,
$ 
  which in turn implies that the first inequality in \eqref{equivxN} holds.
 \end{proof}

\noindent
\textbf{Proof of Theorem~\ref{nNEC}}.
Let us first assume that $\mathcal{\mathbf{B^{\ast}_n}}$ does not satisfy $(A5)_n$. We argue by contradiction
and assume that there exists $T_n^{\ast}$ such that $(OBS)_n$ holds for all
$T>T_n^{\ast}$. We choose initial data such that $u_i^0=u_i^1=0$ for all $i=1, \ldots, n-1$ whereas $(u_n^0,u_n^1)$ is arbitrary
in $H_1 \times H_0$. Then by uniqueness $u_i=u_i^{\prime} \equiv 0$ for all  $i=1, \ldots, n-1$. Thus, $u_n$
is the solution of
\begin{equation}\label{U2On} 
u_n^{\prime\prime} + A u_n=0 \,,\ \ \
(u_n,u_n^{\prime})(0)=(u_n^0,u_n^1) \,, 
\end{equation}
 whereas $(OBS)_n$ reduces to: there exists $T_n^{\ast}>0$ such that for all $T>T_n^{\ast}$ we have
$$
\begin{cases}
\exists \ C>0 \mbox{ such that } 
 \forall \  (u_n^0,u_n^1)\in  H_1 \times H_0 \mbox{ the solution of \eqref{U2On} } \mbox{ satisfies}   \\
C (e_1(U_n)(0)) \le \int_0^T||\mathcal{\mathbf{B}}^*_nU_n||_{G_{n}}^2\,dt \,.
\end{cases}
$$
 Hence $\mathcal{\mathbf{B}}^{\ast}_n$ satisfies $(A5)_n$ which contradicts our hypothesis.
We shall prove the second assertion of Theorem~\ref{nNEC} by induction on $n$. Our claim is as follows: for any integer $n \ge 2$, if the operators $(C_{i i-1})_{i \in\{2, \ldots, n\}}=(\Pi_i)_{i \in\{2, \ldots, n\}}$ do not satisfy $(A3)_n$ then there does not exist $T_n^{\ast}>0$ such that $(OBS)_n$ holds for all $T>T_n^{\ast}$. This holds true for $n=2$ thanks to Theorem~2.8 in~\cite{alabau2cascade}. Assume that it holds up to the order $n$. We assume that the operators $(C_{i i-1})_{2 \le i \le n+1}=(\Pi_i)_{2 \le i \le n+1}$ do not satisfy $(A3)_{n+1}$.
We argue by contradiction and assume that there exists $T_{n+1}^{\ast}>0$ such that for all $T>T_{n+1}^{\ast}$ all the solutions of 
\begin{equation*}\label{NSHn+1X}
\begin{cases}
u_1^{\prime\prime} + A u_1 = 0 \,,\\
u_i^{\prime\prime} + A u_i+ C_{i i-1}u_{i-1}  = 0 \,, 2 \le i \le n+1\,,\\
(u_i,u_i^{\prime})(0)=(u_i^0,u_i^1) \mbox{ for }
i=1, \ldots n+1\,,
\end{cases}
\end{equation*}
 satisfy $(OBS)_{n+1}$. Two cases are possible: either $C_{n+1 n}=\Pi_{n+1}$ does
not satisfy 
\begin{equation}\label{OUPS1}
\begin{cases}
\exists \ T_{0}>0,\mbox{ such that all the solutions } w \mbox{ of }
w'' + A w = 0       
\mbox{ satisfy}  \\
\int_0^T |C_{n+1 n}w^{\prime}|^2 dt \geq \gamma(T)e_1(W)(0)\,, \forall \ T>T_{0} \,,
\end{cases}
\end{equation}
 or the operators $(C_{i i-1})_{2 \le i \le n}=(\Pi_i)_{2 \le i \le n}$ do not satisfy $(A3)_{n}$. Assume the first alternative. We choose initial data such that $u_i^0=u_i^1=0$ for all
$i \neq n$ and $i \in \{1,\ldots,n+1\}$. Then $(OBS)_{n+1}$ together with the property $(A4)_{n+1}$ with
$e_1(U_{n+1})(0)=0$ imply that for all $T>T_{n+1}^{\ast}$ there exists $C_2>0$ such that
$$
C_2 e_0(U_n)(0) \le \int_0^T |C_{n+1n} u_n|^2 \,\,dt \,,
$$
 for all $u_n$ such that $u_n^{\prime\prime} +A u_n=0$.  We set $w=-A^{-1}u_n^{\prime}$. Then we have $w^{\prime}=u_n$ and
$w_n^{\prime\prime}+ Aw_n=0$. We set $W=(w,w^{\prime})$. Then, thanks to the above inequality and to classical density arguments, we have
for all $T>T_{n+1}^{\ast}$ and for all the solutions of 
$w^{\prime\prime} + Aw=0$ with $W(0) \in H_1 \times H_0$
$$
\int_0^T |C_{n+1n}w^{\prime}|^2 \ge C_2 e_1(W)(0)  \,,
$$
 so that $C_{n+1 n}$ satisfies \eqref{OUPS1}, which contradicts our hypothesis.
 We now consider the second alternative, assuming now that $C_{n+1 n}$ satisfies
the above observability property. We choose initial data of the form $u_{n+1}^0=
u_{n+1}^1=0$. Then $(OBS)_{n+1}$ together with the admissibility assumption $(A4)_{n+1}$ and
$e_1(U_{n+1})(0)=0$ imply that for all $T>T_{n+1}^{\ast}$ there exists $C_2>0$ such that
$$
\displaystyle{C_2 \sum_{i=1}^n e_{i-n}(U_i)(0) \le \int_0^T |C_{n+1n} u_n|^2 \,\,dt \,,}
$$
 where $(u_1, \ldots, u_n)$ is any solution of \eqref{NSHn}. Thanks to Lemma~\ref{tech1n} and
defining $W=\mathcal{A}_n^{-1}U_n$ we deduce that for all $T>T_{n+1}^{\ast}$ there exists $C>0$ such that
\begin{equation}\label{OUPS2}
\displaystyle{C \sum_{i=1}^n e_{1+i-n}(W_i)(0) \le \int_0^T |C_{n+1n} w_n^{\prime}|^2 \,\,dt \,,}
\end{equation}
 for all $W$ solution of \eqref{NSHn}. But $\widehat{\mathcal{\mathbf{B^{\ast}_n}}}$
defined by $\widehat{\mathcal{\mathbf{B^{\ast}_n}}}W_n=C_{n+1n}w_n^{\prime}$ satisfies
$(A4)_n$, and thanks to \eqref{OUPS2},  $(OBS)_n$ holds. But 
by our assumption the operators $(C_{i i-1})_{2 \le i \le n}=(\Pi_i)_{2 \le i \le n}$ do not satisfy $(A3)_{n}$. Thanks to our induction hypothesis,
$(OBS)_n$ cannot hold. We have again a contradiction. This concludes the proof.

The proof of   Corollary~\ref{nCNS}  is a direct consequence of Theorem~\ref{nNEC} and Theorem~\ref{obsNSHn}.

To handle the control problem, we shall
need to prove the admissibility and observability properties under a slightly different form
(mainly for the case $B_n \in \mathcal{L}(G_n,H)$). We have the following results.
 
  \begin{Remark}
\rm
 The operator $\mathcal{A}_n$ defined in \eqref{An} generates a $\mathcal{C}^0$-semigroup on $(H_{1-n})^n\times (H_{-n})^n$.  Hence due to the property of reversibility of time,
the Cauchy problem  $U^{\prime}=\mathcal{A}_n U$, $U(T)=U^T \in (H_{1-n})^n\times (H_{-n})^n$  is well-posed, that is
has a unique solution in $\mathcal{C}^0([0,T];(H_{1-n})^n\times (H_{-n})^n)$.  We set 
\begin{equation*}\label{x-n}
\displaystyle{X_{-(n-1)}=(\Pi_{i=1}^n H_{i-n}) \times  (\Pi_{i=1}^n H_{i-n-1})}\,.
\end{equation*}
 Since $X_{-(n-1)}
 \subset (H_{1-n})^n\times (H_{-n})^n$, we can also solve the Cauchy problem
 with $U^T \in X_{-(n-1)}$. In a similar way, we set
\begin{equation*}\label{xn}
X_{(n-1)}=\displaystyle{(\Pi_{i=1}^{n} H_{i-n+1}) \times (\Pi_{i=1}^{n} H_{i-n})}\,.
\end{equation*}
 Then, since $X_{(n-1)}
 \subset (H_{1-n})^n\times (H_{-n})^n$, we can also solve the Cauchy problem
 with $U^T \in X_{(n-1)}$.
 \end{Remark}

\begin{Lemma}\label{obsdirn}
Assume $(A1)$ and $(A2)_n-(A5)_n$. Let $T>0$ be given.
For $W^T=(w_1^T, \ldots, w_n^T,   q_1^T, 
\ldots, q_n^T) \in X_{-(n-1)}$, we denote by $W=(w_1, \ldots, w_n,$ 
$ w_1^{\prime}, \ldots, w_n^{\prime})$ the unique solution
in $\mathcal{C}^0([0,T]; (H_{1-n})^n\times (H_{-n})^n)$ of 

\begin{equation}\label{WnT}
\begin{cases}
w_1^{\prime\prime} + Aw_1=0 \,, \\
w_i^{\prime\prime} + A w_i + C_{i i-1} w_{i-1}=0 \,,  2 \le i \le n\,, \\
W_{|t=T}=W^T \,.
\end{cases}
\end{equation}
Then $W$ satisfies the following properties

\begin{itemize}
\item[(i)]
 $W \in \mathcal{C}^0([0,T]; X_{-(n-1)})$,

\item[(ii)] There exists $C_1=C_1(T)>0$, such that

\begin{equation}\label{directweak1n}
C_1 \int_0^T || \mathcal{\mathbf{B}}_n^{\ast}Z_n||_{G_n}^2 \,dt \le 
\displaystyle{\sum_{i=1}^n e_{i-n}(W_i)(0)}\,,
\end{equation}
 where $Z=\mathcal{A}_n^{-1}W$.

\item[(iii)]
 For all $T>T_n^{\ast}$, where $T_n^{\ast}$ is given in Theorem~\ref{obsNSHn}, there exists $C_2=C_2(T)>0$ such that

\begin{equation}\label{directweak2n}
\displaystyle{\sum_{i=1}^n e_{i-n}(W_i)(0)} \le C_2  \int_0^T || \mathcal{\mathbf{B}}_n^{\ast}Z_n||_{G_n}^2 \,dt \,,
\end{equation}

\item[(iv)]
 Assume furthermore that $\mathcal{\mathbf{B}}_n^*(w,w^{\prime})=B_n^*w^{\prime}$. Then
properties $(ii)-(iii)$ become

\begin{equation}\label{directweak1bisn}
C_1 \int_0^T ||B_n^{\ast}w_n||_{G_n}^2 \,dt \le\, \displaystyle{\sum_{i=1}^n e_{i-n}(W_i)(0)}
\end{equation}
 and
for all $T>T_n^{\ast}$
\begin{equation}\label{directweak2bisn}
\displaystyle{\sum_{i=1}^n e_{i-n}(W_i)(0)} \le C_2  \int_0^T || B_n^{\ast}w_n||_{G_n}^2 \,dt \,,
\end{equation}
 with the same constants $C_1$ and $C_2$ as in $(ii)-(iii)$.
\end{itemize}
\end{Lemma}

\begin{proof}[\bf Proof.]
Since $W \! \in \!\mathcal{C}^0([0,T];(H_{1-n})^n\times (H_{-n})^n$, $w_1 \in
\mathcal{C}([0,T]; H_{1-n})$. 
Thanks to assumption $(A2)_n$, $C_{21}^{\ast} 
 \in \mathcal{L}(H_{n-1})$, thus we have 
$C_{21} \in \mathcal{L}(H_{1-n})$, thus
$w_2$ is a solution of
$$
\begin{cases}
w_2^{\prime\prime} +Aw_2=-C_{21}w_1 \in \mathcal{C}([0,T]; H_{1-n})\,, \\
(w_{2})_{|t=T}=w_2^T \in H_{2-n} \,,\ \ \
(w_{2}^{\prime})_{|t=T}=q_2^T \in H_{1-n} \,,
\end{cases}
$$
 so that $(w_2,w_2^{\prime}) \in \mathcal{C}([0,T]; H_{2-n} \times H_{1-n})$ by uniqueness.
By induction and thanks to  $(A2)_n$, we prove in a similar way that $(w_i,w_i^{\prime}) \in \mathcal{C}([0,T]; H_{i-n} \times H_{i-n-1})$. This yields $W \in \mathcal{C}^0([0,T]; X_{-(n-1)})$. We set $Z=\mathcal{A}^{-1}W$.

Thanks to \eqref{equivxN} and \eqref{admissineqi}, we easily deduce \eqref{directweak1n}. This proves $(ii)$.

Thanks to \eqref{equivxN} and \eqref{eqobsk}, we easily prove  \eqref{directweak2n}. This proves $(iii)$.

The properties $(iv)$ follow easily from the hypothesis on 
$\mathcal{\mathbf{B}}_n^*(w,w^{\prime})$ and from the definition of $Z$ which implies that
$z_n^{\prime}=w_n$.
\end{proof}

\begin{Lemma}\label{obsdirnunbounded}
Assume $(A1)$ and $(A2)_n-(A5)_n$. Let $T>0$ be given.
For $W^T=(w_1^T, \ldots, w_n^T,  
q_1^T, \ldots, q_n^T) \in X_{(n-1)}$, we denote by
 $W=(w_1, \ldots, w_n,$ $ w_1^{\prime}, \ldots, w_n^{\prime})$ the unique solution
in $\mathcal{C}^0([0,T]; (H_{1-n})^n\times (H_{-n})^n)$ of \eqref{WnT}. Then we have
$W \in \mathcal{C}^0([0,T]; X_{(n-1)})$.
\end{Lemma}

\noindent
{\bf Proof.}
The proof is similar to that of Lemma~\ref{obsdirn} and is left to the reader.

\begin{Remark}\label{natspace}
\rm
The admissibility property gives a hidden regularity result and holds true (by extension) for all $U^0 \in X_{-(n-1)}$ (case of bounded control operator $\mathcal{\mathbf{B^{\ast}_{n}}}$) and for all $U^0 \in X_{n-1}$ (case of unbounded control operator $\mathcal{\mathbf{B^{\ast}_{n}}}$). In a similar way the observability properties given in Theorem~\ref{obsNSHn} below also holds for all initial data in $X_{-(n-1)}$ (resp. in $X_{n-1}$). The extension is a consequence of usual density arguments, of the conservation of the energy for the first equation and continuous dependence with respect to source terms and initial data for a single wave equation with a source term.
Hence, the spaces $X_{-(n-1)}$ (case of bounded control operator) and $X_{n-1}$ (case of unbounded control operator) are in some way the "natural" spaces  to set the admissibility and observability properties given respectively in the Theorem~\ref{admissi} and Theorem~\ref{obsNSHn}. Note also that the assumptions on the coupling operators $C_{i i-1}$ for $i=2, \ldots,n$ are
also the "natural" assumptions for well-posedness of the dual homogeneous observability system in $X_{-(n-1)}$ and $X_{n-1}$. 
\end{Remark}

\subsection{Controllability of bi-diagonal $n$-coupled cascade hyperbolic \ \\
systems by a single control}

We apply the HUM method~\cite{lions} (see also~\cite{dolecruss, LLT}) to deduce from the indirect observability inequality obtained in the previous section an indirect exact controllability result for the dual problem. We refer to~\cite{alabau2cascade} for the definition of transposition solutions for the case of $2$-coupled cascade systems.

We consider the control problem
\begin{equation}\label{CTHn}
\begin{cases}
y_i^{\prime\prime} + A y_i +C_{i+1 i}^{\ast}y_{i+1}= 0 \,,  1 \le i \le n-1\,,\\
y_{n}^{\prime\prime} + A y_{n} = B_n v_n\,,\\
(y_i,y_i^{\prime})(0)=(y_i^0,y_i^1) \mbox{ for }
i=1, \ldots, n \,,
\end{cases}
\end{equation}
 where either $B_n \in  \mathcal{L}(G_n;H)$ (bounded control operator) or
 $B_n \in \mathcal{L}(G_n, H_2^{\prime})$ (unbounded control operator). 
 
  We set for all the sequel
 $Y_0=(y_1^0, \ldots, y_n^0, y_1^1, \ldots, y_n^1)$
and  denote by $Y=(y_1,\ldots, y_n,y_1^{\prime},\ldots, y_n^{\prime})$ the solution of \eqref{CTHn} with initial data $Y_0$.
 
\begin{Theorem}\label{control2n}
Assume the hypotheses $(A1)_n-(A5)_n$. We define $T_n^{\ast}>0$ as in Theorem~$\ref{obsNSHn}.$
\begin{itemize}
\item[(i)] Let $\mathcal{\mathbf{B}}_n^*(w_n,w_n^{\prime})=B_n^*w_n^{\prime}$ 
with $B_n \in \mathcal{L}(G_n,H)$. We set 
\begin{equation}\label{Xstar-n}
X_{-(n-1)}^{\ast}=\displaystyle{(\Pi_{i=1}^{n} H_{n-i+1}) \times (\Pi_{i=1}^{n} H_{n-i})}\,.
\end{equation}
 Then, for  all $T >T_n^{\ast}$
and all $Y_0 \in X_{-(n-1)}^{\ast}$, there exists a control function $v_n \in L^2((0,T);G_n)$ such that the solution $Y$ of \eqref{CTHn} with initial data $Y_0$ satisfies $Y(T)=0$.
\item[(ii)]
 Let $\mathcal{\mathbf{B}}_n^*(w_n,w_n^{\prime})=B_n^*w_n$ with $B_n \in\mathcal{L}(G_n, H_2^{\prime})$. We set 
\begin{equation}\label{Xstar-nbd}
X_{(n-1)}^{\ast}=\displaystyle{(\Pi_{i=1}^{n} H_{n-i}) \times (\Pi_{i=1}^{n} H_{n-i-1})}\,.
\end{equation}
 Then, for all $T >T_n^{\ast}$
and all $Y_0 \in X_{(n-1)}^{\ast}$, there exists a control function $v_n \in L^2((0,T);G_n)$ such that the solution $Y$ of \eqref{CTHn} with initial data $Y_0$ satisfies $Y(T)=0$.
\end{itemize}
\end{Theorem}

\begin{proof}[\bf Proof.]
We first consider the case $(i)$.  Let $Y_0=(y_1^0,\ldots, y_n^0,y_1^1,\ldots, y_n^1) \in 
X_{-(n-1)}^{\ast}$. We consider the bilinear form $\Lambda_n$ on $X_{-(n-1)}$ defined by
\begin{equation*}\label{Lambdan}
\Lambda_n(W^T,\widetilde{W}^T)= \int_0^T \langle B_n^{\ast}w_n, B_n^{\ast}\widetilde{w_n}\rangle_{G_n}\,dt \,,
\forall \ W^T, \widetilde{W}^T \in X_{-(n-1)} \,,
\end{equation*}
 and the linear form on $X_{-(n-1)}$ defined for all $W^T  \in X_{-(n-1)}$
by
\begin{equation*}\label{Ln}
\mathcal{L}_n(W^T)=\displaystyle{
\sum_{k=1}^n \langle y_k^1, w_k(0)\rangle_{H_{n-k},H_{k-n}} -
\sum_{k=1}^n \langle y_k^0, w_k^{\prime}(0)\rangle_{H_{n-k+1},H_{k-n-1}} \,,
}
\end{equation*}
 where $W=(w_1,\ldots, w_n,w_1^{\prime},\ldots, w_n^{\prime})$ and
$\widetilde{W}=(\widetilde{w_1},\ldots, \widetilde{w_n},\widetilde{w_1}^{\prime},\ldots,\widetilde{w_n}^{\prime})$ are respectively solutions of \eqref{WnT} and
\begin{equation*}\label{tildeWn}
\begin{cases}
\widetilde{w_1}^{\prime\prime} + A\widetilde{w_1}=0 \,, \\
\widetilde{w_2}^{\prime\prime} + A \widetilde{w_2} + C_{21} \widetilde{w_1}=0 \,, \\
\vdots \,,\\
\widetilde{w_n}^{\prime\prime} + A \widetilde{w_n} + C_{n n-1} \widetilde{w_{n-1}}=0 \,, 
\end{cases}
\widetilde{W}_{|t=T}=\widetilde{W}^T \,.
\end{equation*}
 Thanks respectively to \eqref{directweak1bisn} and to 
\eqref{directweak2bisn}, $\Lambda_n$ is continuous and coercive on $X_{-(n-1)}$ for $T > T_n^{\ast}$. On the other hand, we prove that $\mathcal{L}_n$ is continuous on $X_{-(n-1)}$
as follows. We shall prove by induction on $n \ge 2$ that
\begin{equation}\label{OKA1}
\sum_{k=1}^n e_{k-n}(W_k)(0) \le C \sum_{k=1}^n e_{k-n}(W_k)(T) \,.
\end{equation}
 We already prove this property for $n=2$ in the proof of $(i)$ of Theorem 2.22 in~\cite{alabau2cascade}. However we shall recall briefly how to proceed. Set $n=2$ and $Z=\mathcal{A}_2^{-1}W$.
From the usual energy estimates for the time reverse problem for $Z=\mathcal{A}_2^{-1}W$, and the conservation of $e_0(Z_1)$ through time we have
\begin{equation*}\label{ZZTOP}
e_0(Z_1)(0) + e_1(Z_2)(0) \le C( e_0(Z_1)(T) + e_1(Z_2)(T)) \,.
\end{equation*}
 We now proceed as in  Lemma~2.17-(iv) in~\cite{alabau2cascade}. We have
$Z=(z_1,z_2,z_1^{\prime} , z_2^{\prime})$ where $z_1=-A^{-1}w_1^{\prime}$, $z_2=
 -A^{-1}w_2^{\prime} +A^{-1}C_{21}A^{-1}w_1^{\prime}$, $z_i^{\prime}=w_i$ for $i=1,2$. Therefore
 $e_0(Z_1)=e_{-1}(W_1)$. On the other hand, we have
  $$
 e_0(W_2)=\tfrac{1}{2}\big(|w_2|^2+ |A^{-1/2}w_2^{\prime}|^2\big) \le
 C\big(e_0(Z_1)+ e_1(Z_2)\big)\,.
$$
 And we also have
$ 
e_1(Z_2)=\tfrac{1}{2}\big(|A^{1/2}z_2|^2+ |z_2^{\prime}|^2\big) \le
 C\big(e_{-1}(W_1)+ e_0(W_2)\big)\,.
$ 
  Therefore, we have $e_0(W_2) \le C \big( e_0(Z_1) + e_1(Z_2) \big)$,
$e_1(Z_2) \le C \big(e_{-1}(W_1)+e_0(W_2)\big)$. Hence we have
 $C_1 \big(e_0(Z_1) +e_1(Z_2)\big) \le e_{-1}(W_1) + e_0(W_2) \le
 C_2 \big( e_0(Z_1) + e_1(Z_2)\big)$.
 This leads to
$$
e_{-1}(W_1)(0) + e_0(W_2)(0) \le C (e_{-1}(W_1)(T) + e_0(W_2)(T)) \,,
$$
 This proves \eqref{OKA1} for $n=2$.
Assume that it holds up to order $n-1$. 
Set $Z=\mathcal{A}_n^{-1}W$.  Thanks to Lemma~\ref{tech1n} and to \eqref{OKA1} for $n-1$, we have
\begin{align}\label{OKAX}
& 
C_1 \sum_{k=1}^n e_{k-n}(W_k)(0) \le \sum_{k=1}^n e_{k-n+1}(Z_k)(0) \\
&
\notag \le 
C (\sum_{k=1}^{n-1} e_{k-n+1}(Z_k)(T) +  e_1(Z_n)(0)) \,.
\end{align}
 The usual energy estimates for $z_n$ which solves $z_n^{\prime\prime} + Az_n
+ C_{n n-1} z_{n-1}=0$ yields
$$
e_1(Z_n)(0) \le C(e_1(Z_n)(T) + \int_0^T |z_{n-1}|^2 \,dt) \,.
$$
 Thanks to \eqref{admiss*} with $Z$ replacing $U$ and $0$ replaced by $T$ (reverse problem) we have
$$
\int_0^T |z_{n-1}|^2 \,dt \le C \sum_{k=1}^{n-1} e_{k-n+1}(Z_k) (T) \,,
$$
 so that
\begin{equation}\label{OKXT}
\sum_{k=1}^n e_{k-n+1}(Z_k)(0) \le C \sum_{k=1}^{n} e_{k-n+1}(Z_k) (T) \,.
\end{equation}
Using this inequality in \eqref{OKAX} and once again
Lemma~\ref{tech1n}, we obtain
$$
\sum_{k=1}^n e_{k-n}(W_k)(0) \le C\sum_{k=1}^{n} e_{k-n+1}(Z_k)(T) \le
C \sum_{k=1}^n e_{k-n}(W_k)(T)\,,
$$
 so that our induction property \eqref{OKA1} is proved. This proves that $\mathcal{L}_n$ is continuous on $X_{-(n-1)}$. Hence, 
thanks to Lax-Milgram Lemma, there exists a unique $W^T \in X_{-(n-1)}$ such that
$  \label{HUMn}
\Lambda_n(W^T,\widetilde{W}^T)=-\mathcal{L}_n(\widetilde{W}^T) \,, 
$
$\forall \widetilde{W}^T \in X_{-(n-1)} \,.
$
 We set $v_n=B_n^{\ast}w_n$. Then $v_n \in L^2((0,T);G_n)$ and we have by definition of
the solution of \eqref{CTHn} by transposition
\begin{align*} 
\int_0^T \langle v_n, B_n^{\ast}\widetilde{w_n}\rangle_{G_n} \,dt  
&
= - 
\mathcal{L}_n(\widetilde{W}^T) + 
\displaystyle{
\sum_{k=1}^n \langle y_k^{\prime}(T), \widetilde{w_k}(T)\rangle_{H_{n-k},H_{k-n}} } \\
&
- 
\displaystyle{\sum_{k=1}^n \langle y_k(T), \widetilde{w_k}^{\prime}(T)\rangle_{H_{n-k+1},H_{k-n-1}}
} , \ \ \forall \ \widetilde{W}^T \in X_{-(n-1)}\,.
\end{align*} 
 On the other hand, we have
$$
\int_0^T \langle v_n, B_n^{\ast}\widetilde{w_n}\rangle_{G_n} \,dt=\Lambda_n(W^T,\widetilde{W}^T)=-\mathcal{L}_n(\widetilde{W}^T) \  \forall \ 
\widetilde{W}^T \in X_{-(n-1)}\,,
$$
 so that, we deduce from these two relations that $$Y(T)=(y_1,\ldots, y_n,y_1^{\prime},\ldots, y_n^{\prime})(T)=0\,.$$
Assume now that $(ii)$ holds. Let $Y_0 \in X_{(n-1)}^{\ast}$.  We consider on
$X_{(n-1)}$ the bilinear form defined by

\begin{equation*}\label{Lambdaunbn}
\widetilde{\Lambda}_n(U^T,\widetilde{U}^T)= \int_0^T \langle B_n^{\ast}u_n, B^{\ast}\widetilde{u_n}\rangle_{G_n}\,dt \,,
\forall \ U^T, \widetilde{U}^T \in X_{(n-1)} \,,
\end{equation*}
 and the linear form on $X_{(n-1)}$ defined for all $U^T  \in X_{(n-1)}$, by
\begin{equation*}\label{Lunbn}
\mathcal{J}_n(U^T)= 
\displaystyle{
\sum_{k=1}^n \langle y_k^1, u_k(0)\rangle_{H_{n-k-1},H_{k+1-n}} 
-\sum_{k=1}^n \langle y_k^0, u_k^{\prime}(0)\rangle_{H_{n-k},H_{k-n}}
}\,.
\end{equation*}
Thanks respectively to \eqref{admissineqi} and to 
\eqref{eqobsk} (applied with $U$), $\widetilde{\Lambda}_n$ is continuous and coercive on $X_{(n-1)}$ for $T > T_n^{\ast}$. The proof of continuity of $\mathcal{J}_n$ on $X_{(n-1)}$
follows
from \eqref{OKXT} with $U$ replacing Z. Hence, 
thanks to Lax-Milgram Lemma, there exists a unique $U^T \in X_{(n-1)}$ such that

\begin{equation*}\label{HUMnn}
\widetilde{\Lambda}_n(U^T,\widetilde{U}^T)=-\mathcal{J}_n(\widetilde{U}^T) \,, \quad \forall \ \widetilde{U}^T \in X_{(n-1)} \,.
\end{equation*}
 We set $v_n=B_n^{\ast} u_n$. We deduce as for the case $(i)$ that $Y(T)=0$.
\end{proof}

\section{Control and observation of mixed bi-diagonal and non bi-diagonal $(n+p)$-coupled cascade hyperbolic systems by $p+1$ controls/observations}

The proofs of most of the results in this section are given in the appendix at the end of the paper.

\subsection{Observability  for mixed $(n+p)$-coupled cascade hyperbolic systems by $p+1$
observations}
Let $n \ge 2$ and $p \ge 1$ be  fixed integers. We will generalize our previous results on bi-diagonal $(n+p)$-coupled cascade systems to  mixed bi-diagonal and non-bidiagonal $(n+p)$-coupled cascade systems. More precisely, we shall discuss cascade systems of the form
\begin{equation}\label{NSHnmixed}
\begin{cases}
u_1^{\prime\prime} + A u_1 = 0 \,,\\
u_i^{\prime\prime} + A u_i+ C_{i i-1}u_{i-1}  = 0 \,, 2 \le i \le n \,, \\
\displaystyle{u_{i}^{\prime\prime} + A u_{i}+ \sum_{k=n-1}^{i-1} C_{i k}\,u_k =0 \,,  n+1 \le i \le n+p \,,}\\
(u_i,u_i^{\prime})(0)=(u_i^0,u_i^1) \mbox{ for }
i=1, \ldots n+p\,,
\end{cases}
\end{equation}

\begin{Remark} \rm
Hence the  cascade systems we discuss now, are bi-diagonal in their $n$ first equations,
and then non bi-diagonal for the next equations ranging from $n+1$ to $n+p$. Furthermore, these
last $p$ equations have a peculiar structure: the equation for $u_{n+1}$ have $2$ non
vanishing coupling terms, the next one $3$, up to for the last equation $p+1$ non-vanishing
coupling terms. This form is required for the extension of our previous results.
\end{Remark}
\subsubsection{Main results for observability of $n+p$-coupled cascade systems by $p+1$ observations}
\begin{Theorem}\label{mixedsystem}
Let $n \ge 2$ be an integer.  We assume 
that for all $i=2, \ldots,n$,
the operators $C_{i i-1}$ satisfy the assumption $(A2)_n$ where the operators
$\Pi_i$ satisfy $(A3)_n$. We assume that the operators $C_{i\, j}$ 
and $C_{i\, j}^{\ast}$ are in $\mathcal{L}(H_k)$ for 
$i \in \{n+1, \ldots, n+p\}$ and $j \in \{n-1, \ldots, i-1\}$ and $k=0,1$.
Moreover, let $\mathcal{\mathbf{B^{\ast}_{n+j}}}$ for $j=0$ to $j=p$ be
any given operators satisfying $(A4)_{n+j}-(A5)_{n+j}$ for all $j$ in $\{0,\ldots,p\}$. 
 Then for all $T>0$ there exists
 $C(T)>0$ such that all 
the solutions of \eqref{NSHnmixed} satisfy the following direct inequality
 \begin{equation}\label{admissineqimixed}
 \displaystyle{\sum_{j=0}^{p}
 \int_0^T ||\mathcal{\mathbf{B^{\ast}_{n+j}}}U_{n+j}||_{G_{n+j}}^2 \,dt \le C(T) \Big(
 \sum_{i=1}^{n} e_{1+i-n}(U_{i})(0) + \sum_{j=1}^p e_1(U_{n+j})(0) \Big) \,.} 
 \end{equation}
 Moreover, there exists $T_{n+p}^{\ast}>0$ such that for all $i=1, \ldots, n$, there exist constants 
$d_{i,n}(T)>0$ such that all $T>T_{n+p}^{\ast}$, all the solutions $U$ of \eqref{NSHnmixed}  satisfy the following observability inequalities
\begin{equation}\label{eqobskmixed}
e_{1+i-n}(U_i)(0) \le d_{i,n}(T) \int_0^T \| \mathcal{\mathbf{B^{\ast}_n}}(U_n) \|_{G_{n}}^2 dt \,,
\forall \ i=1, \ldots, n \,,
\end{equation}
 and
\begin{equation}\label{eqobskmixed2}
 \displaystyle{e_{1}(U_{n+k})(0) \le \rho_{n,k}(T) \sum_{l=0}^k\int_0^T \| \mathcal{\mathbf{B^{\ast}_{n+l}}}(U_{n+l}) \|_{G_{n+l}}^2 dt \,,}
\forall \ k=1, \ldots, p \,,
\end{equation}
 where $d_{i,n}(T)>0$ are the constants given in Theorem~$\ref{obsNSHn}$ and
where   $\rho_{n,k}(T)$ $> 0$ are explicit constants which depend on $n$, $k$ and $T$.
\end{Theorem}

 \begin{Remark}
\rm
 The operator $\mathcal{A}_{n+p}$  generates a $\mathcal{C}^0$-semigroup on $((H_{1-n})^n\times (H_{-1})^p) \times ( (H_{-n})^n \times (H_{-2})^p)$.  Hence due to the property of reversibility of time,
the Cauchy problem  $U^{\prime}=\mathcal{A}_{n+p} U$, $U(T)=U^T \in ((H_{1-n})^n\times (H_{-1})^p) \times ( (H_{-n})^n \times (H_{-2})^p)$  is well-posed, that is
has a unique solution in $\mathcal{C}^0([0,T]; ((H_{1-n})^n\times (H_{-1})^p) \times ( (H_{-n})^n \times (H_{-2})^p))$.  We set 
\begin{equation*}\label{M-np}
\displaystyle{M_{-(n+p-1)}=((\Pi_{i=1}^n H_{i-n})\times H^p) \times ( (\Pi_{i=1}^n H_{i-n-1})}
\times (H_{-1})^p)\,.
\end{equation*}
 Since $M_{-(n+p-1)}
 \subset ((H_{1-n})^n\times (H_{-1})^p) \times ( (H_{-n})^n \times (H_{-2})^p)$, we can also solve the Cauchy problem
 with $U^T \in M_{-(n+p-1)}$.  In a similar way, we set
\begin{equation*}\label{MNP}
M_{(n+p-1)}=\displaystyle{((\Pi_{i=1}^{n} H_{i-n+1}) \times H_1^p) \times ((\Pi_{i=1}^{n} H_{i-n})}
\times H^p)\,.
\end{equation*}
 Then, since $M_{(n+p-1)}
 \subset ((H_{1-n})^n\times (H_{-1})^p) \times ( (H_{-n})^n \times (H_{-2})^p)$, we can also solve the Cauchy problem
 with $U^T \in M_{(n+p-1)}$.
\end{Remark}

  \begin{Remark}
\rm
We set 
\begin{equation*}\label{Nnp}
\displaystyle{N_{(n+p-1)}=((\Pi_{i=1}^n H_{1+i-n})\times H_1^q \times H^{p-q}) \times ( (\Pi_{i=1}^n H_{i-n})}
\times H^q\times (H_{-1})^{p-q})\,.
\end{equation*}
 Since $N_{(n+p-1)}
 \subset ((H_{1-n})^n\times (H_{-1})^p) \times ( (H_{-n})^n \times (H_{-2})^p)$, we can also solve the Cauchy problem
 with $U^T \in N_{(n+p-1)}$.
 \end{Remark}

\subsection{Control of mixed bi-diagonal and non bi-diagonal cascade hyperbolic systems by $p+1$ controls}

We apply the HUM method~\cite{lions} to deduce from our above results,  exact controllability results for the dual $n+p$-coupled cascade systems by either $p+1$ bounded control operators,
$p+1$ unbounded control operators and mixed $p+1$ bounded/unbounded control operators.

We consider the control problem
\begin{equation}\label{CTHn+pq}
\begin{cases}
y_i^{\prime\prime} + A y_i +C_{i+1 i}^{\ast}y_{i+1}= 0  \,, 1\le i \le n-2\,,\\
\displaystyle{y_{n-1}^{\prime\prime} + A y_{n-1} + \sum_{k=n}^{n+p}C_{k i}^{\ast} y_k=0} \,,\\
\displaystyle{y_i^{\prime\prime} + A y_i +\sum_{k=i+1}^{n+p}C_{k i}^{\ast}y_{k}= B_i v_i  \,, n\le i \le n+p-1}\,,\\
y_{n+p}^{\prime\prime} + A y_{n+p}=B_{n+p} v_{n+p}\,\\
(y_i,y_i^{\prime})(0)=(y_i^0,y_i^1) \mbox{ for }
i=1, \ldots, n+p \,,
\end{cases}
\end{equation}
 where we use the convention that the first equation has to disappear  if $n=2$ and where either 
\begin{itemize}
\item for all $k\in \{0, \ldots, p\}$, the operators $B_{n+k} \in  \mathcal{L}(G_{n+k};H)$ (bounded control operators) 

\item or for all $k\in \{0, \ldots, p\}$, the operators 
 $B_{n+k} \in \mathcal{L}(G_{n+k}, H_2^{\prime})$ (unbounded control operators) 
 
 \item or for all $k\in \{0, \ldots, q\}$, the operators 
 $B_{n+k} \in \mathcal{L}(G_{n+k}, H_2^{\prime})$ (unbounded control operators) and for
 all  $k\in \{q+1, \ldots, p\}$, the operators $B_{n+k} \in  \mathcal{L}(G_{n+k};H)$ (bounded control operators). 
 \end{itemize}
 We set for all the sequel $Y_0=(y_1^0, \ldots, y_{n+p}^0, y_1^1, \ldots, y_{n+p}^1)$. 
 
 \subsubsection{The case of either all bounded or all unbounded control operators}
 We shall first consider the case of either all bounded (resp. unbounded) control operators.

\begin{Theorem}\label{control2n+pq}
We assume the hypotheses of Theorem~$\ref{mixedsystem}$ and define $T_{n+p}^{\ast}>0$ as
 in Theorem~$\ref{mixedsystem}.$  We have
\begin{itemize}

\item[(i)] Let $\mathcal{\mathbf{B}}_{n+k}^{\ast}(w_{n+k},w_{n+k}^{\prime})=B_{n+k}^*w_{n+k}^{\prime}$ 
with $B_{n+k} \in \mathcal{L}(G_{n+k},H)$ for all $k \in \{0, \ldots, p\}$. We set 
\begin{equation}\label{Mstar-npq}
M_{-(n+p-1)}^{\ast}=\displaystyle{((\Pi_{i=1}^{n} H_{n-i+1})\times H_1^p) \times ((\Pi_{i=1}^{n} H_{n-i})
\times H^p)}\,.
\end{equation}
 Then, for  all $T >T_{n+p}^{\ast}$ ,
and all $Y_0 \in M_{-(n+p-1)}^{\ast}$, there exist control functions $v_{n+k} \in L^2((0,T);G_{n+k})$
for $k=0, \ldots, p$ such that the solution $Y=(y_1,\ldots, y_{n+p},y_1^{\prime},\ldots, y_{n+p}^{\prime})$ of \eqref{CTHn+pq} with initial data $Y_0$ satisfies $Y(T)=0$.
\item[(ii)] Let $\mathcal{\mathbf{B}}_{n+k}^{\ast}(w_{n+k},w_{n+k}^{\prime})=B_{n+k}^{\ast}w_{n+k}$ with $B_{n+k} \in\mathcal{L}(G_{n+k}, H_2^{\prime})$ for all $k \in \{0, \ldots, p\}$. We set 
\begin{equation*}\label{Xstar-n+p}
M_{(n+p-1)}^{\ast}=\displaystyle{((\Pi_{i=1}^{n} H_{n-i}) \times H^p) \times ((\Pi_{i=1}^{n} H_{n-i-1})
\times (H_{-1})^p)}\,.
\end{equation*}
 Then, for all $T >T_n^{\ast}$,
and all $Y_0 \in M_{(n+p-1)}^{\ast}$, there exist control functions $v_{n+k} \in L^2((0,T);G_{n+k})$ for $k=0, \ldots, p$ such that the solution $Y=(y_1,\ldots, y_{n+p},y_1^{\prime},\ldots, y_{n+p}^{\prime})$ of \eqref{CTHn+pq} with initial data $Y_0$ satisfies $Y(T)=0$.
\end{itemize}
\end{Theorem}

\subsubsection{The case of  mixed bounded and unbounded control operators}

\begin{Theorem}\label{control2n+pqmixed}
We assume the hypotheses of Theorem~$\ref{obsdirnmixed}$ and define $T_{n+p}^{\ast}>0$ as
 in Theorem~$\ref{mixedsystem}. $
We assume that $\mathcal{\mathbf{B}}_{n+k}^{\ast}(w_{n+k},w_{n+k}^{\prime})=B_{n+k}w_{n+k}$ 
with $B_{n+k} \in \mathcal{L}(G_{n+k},H_2^{\prime})$ for all $k \in \{0, \ldots, q\}$. We further assume
that $\mathcal{\mathbf{B}}_{n+k}^{\ast}(w_{n+k},w_{n+k}^{\prime})=B_{n+k}^{\ast}w_{n+k}^{\prime}$ with $B_{n+k} \in \mathcal{L}(G_{n+k}, H)$ for all $k \in \{q+1, \ldots, p\}$.
We set 
\begin{equation*}\label{Nstar-npq}
N_{(n+p-1)}^{\ast}=\displaystyle{((\Pi_{i=1}^{n} H_{n-i})\times H^q \times H_1^{p-q}) \times ((\Pi_{i=1}^{n} H_{n-i-1})
\times H_{-1}^q \times H^{p-q})}\,.
\end{equation*}
 Then, for  all $T >T_{n+p}^{\ast}$ ,
and all $Y_0 \in N_{(n+p-1)}^{\ast}$, there exist control functions $v_{n+k} \in L^2((0,T);G_{n+k})$
for $k=0, \ldots, p$ such that the solution $Y=(y_1,\ldots, y_{n+p},y_1^{\prime},\ldots, y_{n+p}^{\prime})$ of \eqref{CTHn+pq} with initial data $Y_0$ satisfies $Y(T)=0$.
\end{Theorem}

\begin{Remark}
\rm
We can note that the same method does not apply to the case for which the first $q+1$ control operators are bounded, whereas the next $p-q$ ones are unbounded. 
\end{Remark}

\section{Main applicative results}
The main results are proved under a general abstract form given in the successive next sections. However for more clarity we shall give in this section the main consequences on the most well-known examples of applications, namely: wave-type, heat-type and Schr\"odinger cascade coupled systems. Other examples, based for instance on mechanical systems such as plates or Euler-Bernouilli beams can be given, but are not detailed here for length reasons.
\subsection{Geometric preliminaries}
Let $\Omega$ be a bounded open set in $\mathbb{R}^d$ with a sufficiently smooth boundary $\Gamma$. The set $\Omega$ can also be a smooth connected compact Riemannian manifold with or without boundary as in \cite{alaleau11}.  Let $T$ be a given positive time. We recall
the following definition for the Geometric Control Condition of Bardos Lebeau Rauch~\cite{blr92}.
\begin{Definition}
We say that an open subset $\omega$ of $\Omega$ satisfies $(GCC)$ if there exists a time $T>0$ such that every generalized bicharacteristic traveling at speed $1$ in $\Omega$ meets $\omega$ at a time $t<T$. We say that a subset $\Gamma_1$ of the boundary $\Gamma$ satisfies $(GCC)$ if there exists a time $T>0$ such that every generalized bicharacteristic traveling at speed $1$ in $\Omega$ meets $\Gamma_1$ at a time $t<T$ in a non-diffractive point.
\end{Definition}

\begin{Remark}
\rm
In the one dimensional case, $\omega \subset \Omega$ (resp. $\Gamma_1 \subset \Gamma$) satisfies $(GCC)$ as soon as $\omega$ is
any non-empty open subset of $\Omega$ (resp. any non-empty open subset in $\Gamma$).
\end{Remark}

Our abstract results require that the coupling operators $C_{ii-1} \in \mathcal{L}(H)$ satisfy the property that $C_{ii-1}^{\ast} \in \mathcal{L}(H_k)$ for all $k=0, \ldots n-i+1$ for $i=2, \ldots, n$ where $H$ is a given Hilbert space (the pivot space)
and $H_k$ are the domains of some fractional powers of the unbounded coercive operator $A$. In applications, we will be interested by cases for which $H=L^2(\Omega)$, $A$ is the Dirichlet Laplacian and $C_{i i-1}u=c_{i i-1}u$ for all $u\in H$, where the coefficients $c_{i i-1}$ are smooth functions defined on $\Omega$. We shall see that $C_{i i-1}^{\ast} \in \mathcal{L}(H_k)$  if and only if $c_{i i-1}$ satisfies certain compatibility properties when $k \ge 3$, whereas no compatibility conditions will be required for $k \le 2$. 
On the other hand, we are also interested by the situation for which these coefficients are supported in a neighborhood of subsets $\overline{O_i}$ which satisfy $(GCC)$. This is a sufficient (and almost necessary condition) for our abstract result to hold (see assumption $(A5)_n$). 
In general, except in the one-dimensional case, this geometric condition  implies that $\overline{O_i}$ meets a non-empty open subset of the boundary $\Gamma$, so that $c_{i i-1}$ can vanish on $\Gamma$ outside a neighborhood in $\Gamma$ of $\partial O_i\cap \Gamma$ and $c_{i i-1}>0$ on $\partial O_i\cap \Gamma$. So we must describe how these geometric condition can be combined
with the compatibility conditions. Furthermore,  the observation/control region $\omega \subset \Omega$ (resp. $\Gamma_1 \subset \Gamma$) in the locally distributed (resp. boundary) case should also satisfy $(GCC)$. 
Hence if we want to describe possible geometric examples for which the coupling region does not meet the control regions, we should build coefficients which both satisfy the compatibility conditions, are supported in a neighborhood of $\overline{O_i}$ and  for which $\partial O_i\cap \Gamma$ is as small as possible.  

So let us consider the case of an operator $C$ defined as the multiplication operator by $c$ where $c$ is a sufficiently smooth nonnegative function
on $\Omega$, $A$ is the Dirichlet Laplacien, $H_k=D(A^{k/2})$. We shall describe below for $k \ge 3$ and $d \ge 1$, the compatibility conditions for  smooth coefficients $c$ and for certain geometries of $\Omega$.

 We recall that 
\begin{equation}\label{iterateDA}
D(A^{k/2})=\{ u \in H^k(\Omega) \,, u=\Delta u = \ldots \Delta^{[(k-1)/2]}u=0 \mbox{ on } \Gamma\}
\end{equation}
 for all $k \in \{0, 1,\ldots\}$ and where $[x]$ hands for the integer part of the real number $x$.

We shall need the following notation.
For a multi-integer $\alpha=(\alpha_1,
\ldots,\alpha_d)\ \\$$ \in \N^d$, we denote by $|\alpha|=\sum_{i=1}^d \alpha_i$
and $\partial^{\alpha} = \frac{\partial^{|\alpha|}}{\partial_1^{\alpha_1} \ldots \partial_d^{\alpha_d}}$. 
\begin{Definition}\label{1Dcompat}
Let $d=1$ and $\Omega=(L_1,L_2)$ with $-\infty<L_1<L_2<\infty$. We say that $c \in W^{k,\infty}(\Omega)$ satisfies the compatibility condition $(C_{1D})_{[(k-1)/2]}$, if it satisfies
\begin{equation}\label{comp-1D}
(C_{1D})_{[(k-1)/2]} \quad \quad
c^{(2p-1)}(L_i)=0 \quad i=1,2 \,, p=1, \ldots, [(k-1)/2] \,.
\end{equation}
\end{Definition}
\begin{Proposition}\label{1dcomp}
Assume that $\Omega=(L_1,L_2)$ with $-\infty<L_1<L_2<\infty$ and let $k \ge 3$ be a given integer. If $c \in W^{k,\infty}(\Omega)$ satisfies the compatibility condition $(C_{1D})_{[(k-1)/2]}$
then for all $u \in H_k$, $c\,u \in H_k$. Moreover this condition is also necessary.
\end{Proposition}

\begin{Remark}
\rm
Hence in the one-dimensional case,  for all given non-empty open subset $O$ of $\Omega$, we can build smooth nonnegative functions  $c \in W^{k,\infty}(\Omega)$ satisfying the compatibility condition \eqref{comp-1D}, supported in $\overline{O}$ and such that
$c>0$ in a subset of $O$. In particular if $\omega \subset \Omega$ (resp. $\Gamma_1
\subset \Gamma$ with $\Gamma_1 \neq \emptyset$) is a non-empty open set standing for the control region, it then satisfies $(GCC)$ and it can be chosen such that
$\overline{\omega} \cap \overline{O}=\emptyset$ (resp. $\Gamma_1 \cap \overline{O}=\emptyset$), so that the control region does not meet
the coupling region. It can be then generalized to any number of controls and coupling regions.
\end{Remark}
We shall now consider the case $d \ge 2$. We can remark that if $\Gamma_0 \subset \Gamma$ is such that its normal is constant
over this part of the boundary then the compatibility conditions reduce to \eqref{comp-1D} where the derivatives should be taken along
this normal.

\begin{Definition}\label{compD}
Assume $d \ge 2$. Let $k \ge 3$ be a given integer and $c \in W^{k, \infty}(\Omega)$. 
For $k=3$ or $k=4$, we say that $c$ satisfies the compatibility condition $(C)_1$ if
\begin{equation}\label{comp34}
(C)_{1}   \quad \quad
\displaystyle{\frac{\partial c}{\partial \nu}=0  \mbox{ on } \Gamma \,. } 
\end{equation}
 For $k \ge 5$, we say that $c$ satisfies the compatibility condition $(C)_{[(k-1)/2]}$ if
\begin{equation}\label{comp}
(C)_{[(k-1)/2]}
\begin{cases}
(C)_1 \mbox{ holds } \,,\\
\displaystyle{\mbox{ For } \alpha \in \mathbb{N}^d \mbox{ such that } 2 \le |\alpha| \le  2 [(k-1)/2] -1}\,,\\
\displaystyle{\partial^{\alpha} c=0  \mbox{ on } \Gamma}\,.
\end{cases}
\end{equation}
\end{Definition}
\begin{Proposition}\label{compatibilityC}
Assume $d \ge 2$. Let $k \ge 3$ be an integer and  $c \in W^{k, \infty}(\Omega)$. We assume that $c$ satisfies the compatibility condition $(C)_{[(k-1)/2]}$. Then
for all $u \in H_k$, $c\,u \in H_k$.
\end{Proposition}

\begin{Remark}
\rm
If $\Omega$ is a ball of radius $R$ then an open subset which is a neighborhood of any radius will satisfy $(GCC)$. Also, any locally distributed (or boundary) control region $\omega \subset \Omega$ (resp. $\Gamma_1 \subset \Gamma$) such that $\overline{\omega} \supset \{x \in \Gamma\,,
(x-x^0)\cdot \nu(x) >0\}$ (resp. $\Gamma_1 \supset \{x \in \Gamma\,,
(x-x^0)\cdot \nu(x) >0\}$) where $x^0 \in \mathbb{R}^d$ is any given point, satisfies $(GCC)$. Hence if we want to construct a coupling function $c$ such that its support does not meet the control region, it means that $c$ should vanish in some suitable neighborhood of the boundary. On the other hand, this coefficient should satisfy the above compatibility condition.
For $k=3$ or $k=4$, the compatibility conditions reduces to $(C)_1$. Hence one should construct a smooth function $c$ which is nonnegative, strictly positive  in $\overline{O}$, vanishes  outside a neighborhood of $O$  and such that the trace of its normal derivative vanishes on $\Gamma$. We show in the next proposition that this construction is possible. For $(x_1,x_2) \in \mathbb{R}^2$, we denote by $(r,\theta)=(r(x_1,x_2), \theta(x_1,x_2))$ the usual polar coordinates.
\end{Remark}
\begin{Proposition}\label{comp-ball}
Assume that $\Omega \subset \mathbb{R}^2$ is a ball of radius $R>0$. Let $\theta_0$ be a fixed
angle in $[0, 2 \pi)$ and define $\mathcal{R}=\{(r\cos(\theta_0), r\sin(\theta_0))\,, r \in [0,R]\}$. Then for any neighborhood $U$ of $\mathcal{R}$ in
$\Omega$, there exist smooth functions $c$ such that $c \ge 0$ in $\Omega$, $c>0$ in $\overline{U}$, $c$ vanishes outside a neighborhood of $U$ and $c$ satisfies $(C)_1$.
\end{Proposition}

\begin{proof}[\bf Proof.]
We can assume without loss of generality that $U \subset \{(r, \theta), r \in [\varepsilon, R]\,, \theta \in [\theta_0-\delta, \theta_0+\delta]\}
\cup B(0, \varepsilon)$ where $0<\varepsilon<R/2$ and $\delta$. Let $\phi$ be any given smooth nonnegative function on $[\varepsilon, R]$ such
that $\phi^{\prime}(R)=0$, $\phi>0$ on $[\varepsilon, R]$,  $\phi \equiv 0$ in $[0, \varepsilon/2]$. We choose $\psi$ as a smooth nonnegative function on $[0, 2\pi]$ such that $\psi >0$ on $[\theta_0-\delta, \theta_0+\delta]$ and $\psi$ vanishes outside  $[\theta_0-2\delta, \theta_0+2\delta]$. We moreover choose a smooth function $\eta$ on $\overline{\Omega}$ such
that $\eta \equiv 1$ in $B(0,\varepsilon)$, $\eta \equiv 0$ outside $B(0, 2\varepsilon)$ and $\eta \in [0,1]$ elsewhere. We set
$$
c(x,y)= \phi(r(x_1,x_2))\psi(\theta(x_1,x_2)) + \eta(x_1,x_2) \,,\quad  (x_1,x_2) \in B(0,R) \,.
$$
 Then $c$ satisfies the desired properties by construction.
\end{proof}

\begin{Remark}\label{rkcomp-ball}
\rm
The above proof can be generalized to $d \ge 3$ by use of the generalized polar coordinates.
\end{Remark}
\begin{Proposition}
Assume that $d \ge 2$ and $k \ge 5$. We set $\Omega=B(0,R)$ with $R>0$ and $\Gamma=\partial \Omega$.
If $c$ is a smooth function satisfying $(C)_{[(k-1)/2]}$, then $c$ has to be constant
on $\Gamma$. Moreover if $c$ is constant on $\Gamma$ and satisfies 
\begin{equation}\label{suffComp}
\displaystyle{
\frac{\partial ^p c}{\partial \nu^p}=0 \quad p=1, \ldots, 2 [(k-1)/2] -1 \mbox{ on } \Gamma\,,}
\end{equation}
 then $c$ satisfies $(C)_{[(k-1)/2]}$.
\end{Proposition}

\begin{proof}[\bf Proof.]
We will only sketch the proof for the dimension $d=2$.
If $k \ge 5$ then the compatibility condition requires that if $u$ is in $H_5$, then $cu \in H_5$, so that $u=\Delta (cu)= \Delta^2 (cu)=0$
on $\Gamma$ should hold whenever $u \in H^5(\Omega)$ and $u=\Delta u=\Delta ^2u=0$ on $\Gamma$. Using polar coordinates
and the expression of the Laplacian in polar coordinates, we find that this requires
$\frac{\partial c}{\partial \theta} \equiv 0$ on $\Gamma$. Hence $c$
should be constant on $\Gamma$. Conversely if  $c$ is constant on $\Gamma$ and satisfies \eqref{suffComp}, then all the
tangential and normal derivatives of at least order $1$ and up to $2 [(k-1)/2] -1$ are vanishing on $\Gamma$ so that
the compatibility condition is satisfied.
\end{proof}

For $n \ge 4$, we can now give the compatibility conditions on the coefficients $c_{i i-1}$, which implies that when $C_{i i-1}$ stands for the multiplication operator by $c_{i i-1}$ in $H$, then $C_{ii-1}^{\ast} \in \mathcal{L}(H_k)$ for all $k=0, \ldots n-i+1$ for $i=2, \ldots, n$. One can note that no compatibility conditions are required on the coefficients if $n \le 3$. Moreover when $n \ge 4$, no compatibility conditions
are required on the coefficients $c_{n-1 n-2}$ and $c_{n n-1}$. For $n \ge 4$, the sufficient compatibility conditions read as follows.
\begin{Definition}\label{compcii-1}
Let $n \! \ge  \! 4.$  We say that the coefficients $ c_{i i-1}  \! \in  \! W^{n-i+1, \infty}(\Omega) $
 for $ i=2, \ldots, n-2 ,$  
 satisfy the compatibility condition
$(CC)_n$ if  for the dimension $d=1$ $($resp. $d \ge 2 )$ $c_{i i-1}$ satisfies $(C_{1D})_{[(n-i)/2]}$
$($resp. $(C)_{[(n-i)/2]} ) $ for all $i=2, \ldots, n-2$.
\end{Definition}
\subsection{Boundary and localized observability/controllability of $n+p$-coupled cascade wave equations with localized couplings}
\subsubsection{Main observability results for $n+p$-coupled cascade wave systems}
Let $ n \ge 2$ and $p \ge 0$ be two integers. We consider the following $n+p$-coupled mixed bi-diagonal and non bi-diagonal cascade system of wave equations
\begin{equation*}\label{mixedbinondiagonalH}
\begin{cases}
u_{1,tt} -\Delta  u_1 = 0 \quad \mbox{ in } (0,T)\times \Omega\,,\\
u_{i,tt} -\Delta  u_i+ c_{i i-1}(x)u_{i-1}  = 0 \quad \mbox{ in } (0,T)\times \Omega \,, 2 \le i \le n \,, \\
\displaystyle{u_{i,tt} -\Delta u_{i}+ \sum_{k=n-1}^{i-1} c_{i k}(x)u_k  =0 \quad \mbox{ in } (0,T)\times \Omega\,, n+1 \le i \le n+p \,,}\\
u_i=0 \mbox{ for } i=1, \ldots n+p \quad \mbox{ in } (0,T)\times \partial \Omega\,,\\
(u_i,u_{i,t})(0)=(u_i^0,u_i^1) \mbox{ for }
i=1, \ldots n+p \quad \mbox{ in } \Omega\,,
\end{cases}
\end{equation*}
 where the subscript $t$  denotes the partial derivative  with respect to time $t$. This system can be written as an abstract second
order differential system
$$
u^{\prime\prime} + \mathcal{M}_{n+p}^{\ast}u=0\,,
$$
  with the appropriate initial conditions and where $u=(u_1, \ldots, u_{n+p})^t$, the matrix operator $\mathcal{M}_{n+p}$ is given by \eqref{Mn+p} and
where $A$ is the unbounded operator in $H=L^2(\Omega)$ defined by $Au=-\Delta u$
for $u \in D(A)=H^2(\Omega) \cap H^1_0(\Omega)$. 

We set $U^0=
(u_1^0, \ldots, u_{n+p}^0, u_1^1, \ldots, u_{n+p}^1)$ for all the sequel of this section. We make the following assumptions on the coefficients $c_{i,j}$.
$$
(H1)
\begin{cases}
c_{i i-1} \in W^{n-i+1,\infty}(\Omega) \mbox{ for } i=2, \ldots, n\,,\\
c_{i i-1} \ge 0 \mbox{ on } \Omega \mbox{ for } i=2, \ldots, n\,,\\
\{c_{i i-1} >0\} \supset \overline{O_i} \mbox{ for some open subsets } O_i \subset \Omega 
\mbox{ for } i=2, \ldots, n\,,\\
c_{j,k} \in W^{1,\infty}(\Omega) \mbox{ for } j \in \{n+1, \ldots, n+p\}\,, k \in \{n-1, \ldots, j-1\} \,.
\end{cases}
$$

\begin{Remark}\label{localizedregion}
\rm
The couplings terms  $c_{i i-1}$ for $i=2, \ldots, n$ located on the main subdiagonal are
assumed to be strictly positive on $\overline{O_i}$ for $i=2, \ldots, n$, so that the coupling effects due to these terms are effective in a neighborhood of these sets. We will say in all the sequel that the subsets $\overline{O_i}$ for $i=2, \ldots, n$ are the regions on which the couplings are localized.
\end{Remark}
We shall consider $p+1$ observations associated to the equations/compo-\ \\ nents ranked from $n$ up
to $n+p$, the $n-1$ first equations/components being unobserved. These $p+1$ observations can be each either locally distributed or localized on parts of the boundary. We shall denote by
$I_{int}$ (resp. $I_{bd}$) the set of integers $k \in \{0, \ldots, p\}$ such that the corresponding observation is locally distributed (resp. is localized on a part of the boundary). We do not recall below the admissibility property for the coupled system which allows to show that the observations
of the solution are well-defined (hidden regularity result) in a classical way. This property is given in the abstract Theorem~\ref{mixedsystem}.
\begin{Theorem}[Observability estimates]\label{ncoupledwavethm}
We assume that  the hypothesis $(H1)$ holds for some open subsets $O_i
\subset \Omega$ for $i=2, \ldots, n$ that satisfy $(GCC)$ for all $i \in \{2, \ldots, n\}$. 
If $n \ge 4$, we further assume that the coefficients $c_{i i-1}$ for $i=2, \ldots, n-2$ satisfy the compatibility condition
$(CC)_n$.
Let $b_{n+k}$ for $k=0, \ldots, p$ be given functions defined on $\Omega$ (resp.
$\Gamma$) for $k \in I_{int}$ (resp. $k \in I_{bd}$) such that
$$ 
\mbox{ For all } k \in I_{int}:   
b_{n+k} \ge 0 \mbox{ on } \Omega  ,\ \
\{b_{n+k} >0\} \supset \overline{\omega_{n+k}}  ,\  
b_{n+k} \in L^{\infty}(\Omega) , 
$$
$$ 
\mbox{ For all } k \in I_{bd}:   
b_{n+k} \ge 0 \mbox{ on } \Gamma  ,\ \
\{b_{n+k} >0\} \supset \overline{\Gamma_{n+k}}  ,\ 
b_{n+k} \in L^{\infty}(\Gamma) , 
$$
  for some open subsets $\omega_{n+k} \subset \Omega$ and some subsets $\Gamma_{n+k} \subset \Gamma$ such that
$\omega_{n+k}$ satisfy $(GCC)$ for all  $k \in  I_{int}$ and the subsets $\Gamma_{n+k}$ satisfy $(GCC)$ for all $k$ in $I_{bd}$. Then there exists $T^{\ast}>0$ such that for
all $T>T^{\ast}$, there exist constants $c_{i,n}(T)>0$ and $d_{k,n}(T)>0$ such that 
for all $U^0 \in (\Pi_{i=1}^n D(A^{(1+i-n)/2})) \times (H^1_0(\Omega))^p \times
(\Pi_{i=1}^n D(A^{(i-n)/2})) \times (L^2(\Omega))^p$ the following observability inequalities hold
$$
\begin{cases}
\mbox{ Either } 0 \in I_{int} \mbox{ and then  for all } \ i=1, \ldots, n\,:\\
\displaystyle{c_{i,n} (T) ||(u_i^0,u_i^1)||^2_{D(A^{(1+i-n)/2})\times D(A^{(i-n)/2)}} \le 
\int_0^T \int_{\Omega}b_n |u_{n,t}|^2 \,dx\, dt  } \,, \\
\mbox{ Or } 0 \in I_{bd} \mbox{ and then for all } \ i=1, \ldots, n\,,:\\
\displaystyle{c_{i,n} (T) ||(u_i^0,u_i^1)||^2_{D(A^{(1+i-n)/2})\times D(A^{(i-n)/2})} \le 
\int_0^T \int_{\Gamma} b_n\Big|\frac{\partial u_{n}}{\partial \nu}\Big|^2 \,d \sigma\, dt }\,, \\
\mbox{ and }\\ 
\displaystyle{d_{n,k}(T)
 \|(u_{n+k}^0,u_{n+k}^1)\|^2_{H^1_0(\Omega) \times L^2(\Omega)}} 
 \! \le    \! 
\displaystyle{\int_0^T   \!  \int_{\Omega} \sum_{l \in I_{int}, 0 \le l \le k} b_{n+l} |u_{{n+l}, t}|^2 \,dx\, dt}  \\
+ \displaystyle{\int_0^T \int_{\Gamma}\sum_{l \in I_{bd}, 0 \le l \le k} b_{n+l} \Big|\frac{\partial u_{n+l}}{\partial \nu}\Big|^2 \,d \sigma\, dt } 
\quad \forall \ k=1, \ldots, p\,.\\
\end{cases}
$$
\end{Theorem}

\begin{Remark}
\rm
We do not give a characterization of the minimal observability time, and  $T^{\ast}$ is not in general the minimal observatility time. In the case $n=2$, that is for $2$-coupled cascade systems and for locally distributed observation in a $\mathcal{C}^{\infty}$ manifold without boundary, Dehman, Le Rousseau and L\'eautaud in~\cite{DLRL} (see also~\cite{these-leautaud}), give a characterization of the minimal control time by a contradiction argument based on a micro-local analysis. It would be interesting to check if their approach together with the natural observability inequality proved in our paper for $n+p$-coupled
cascade systems, can lead to a characterization of the minimal control time for $n+p$ coupled systems, by contradiction arguments and a micro-local analysis approach.
\end{Remark}

\begin{Remark}
\rm
As in the Remark~\ref{localizedregion}, the subsets $\omega_{n+k}$ (resp. $\Gamma_{n+k}$) of $\omega$ (resp. $\Gamma$) are the regions (indeed in a neighborhood of them) on which the observations are localized. 
In the one-dimensional case, $(GCC)$ is satisfied for any nonempty open subset of $\Omega$ or for any extremity of $\Omega$. Hence, we can exhibit subsets $O_i$ for $i=2, \ldots, n$ such that $(H1)$,
and $(CC)_n$ hold for arbitrary nonempty subsets $O_i$ of $\Omega$ and arbitrary observation region $\omega_n$ or $\Gamma_n$ with
$\overline{O_i}\cap \overline{\omega_n}=\emptyset$ (resp. $\overline{O_i}\cap \overline{\Gamma_n}=\emptyset$) for all $i=2, \ldots, n$. 

Let us now consider the case of dimensions $d \ge 2$. If $n=3$ and $p=0$, we prove that $3$-coupled bi-diagonal cascade systems
with two coupling terms localized respectively on subregions $O_2 \subset \Omega$ and $O_3 \subset \Omega$ with a single observation either  locally distributed on $\omega_3 \subset \Omega$ or distributed on a part $\Gamma_3 \subset \Gamma$ of the boundary, is observable under the geometric condition that both $O_i$ for $i=2, 3$ and $\omega_3$
(resp. $O_i$ for $i=2,3$ and $\Gamma_3$) satisfy (GCC). Hence this covers many situations for which the intersection
$\overline{O_i} \cap \overline{\omega_3}=\emptyset$ for $i=2,3$ (resp. $\overline{O_i} \cap \overline{\Gamma_3}=\emptyset$). This geometric situation still holds true for $n=4$ and $n=5$, since in this case the compatibility condition reduces to
the property that the coefficients $c_{21}$ and $c_{32}$ should satisfy $(C)_1$, but they can be locally supported in a region which meets
only a small part of the boundary (as small as we want for a ball for instance as proved in Proposition~\ref{comp-ball} and
Remark~\ref{rkcomp-ball}). Hence the coupling regions  for $c_{21}\,, c_{32}$ and $c_{54}$
can be chosen in a such a way that they do not meet any other control/observation regions (which should also satisfy $(GCC)$). This also holds
with a larger number of observations ($p \ge 1$).
If $n=6$, then $c_{21}$ should satisfy $(C)_2$ so that for a ball in $\mathbb{R}^2$ it should be constant on the boundary and
have normal derivatives up to $3$ equal to zero on the boundary, whereas the coefficients $c_{i i-1}$ for $i=3, 4$ should satisfy $(C)_1$ but can
be localized in a region which meets only a small part of the boundary. No conditions are required for $c_{5 4}$ and $c_{65}$. 

More generally, if $d \ge 2$ and $n \ge 6$, then the supports of $c_{i i-1}$ for $i=2, \ldots, n-4$ have to contain a neighborhood of the boundary $\Gamma$ so that the support of these coupling coefficients will necessarily meet the support of the control regions. However we can build
examples for which the intersections between the supports of these coupling coefficients and the control regions are non empty but are neighborhoods (as small as we want) of some  parts (as small as we want in the locally distributed observation case) of the boundary (with respect to the appropriate measure on the boundary). 
Moreover
we can exhibit geometric examples for which the supports of the remaining coefficients $c_{i i-1}$ for $i=n-3, \ldots n$ can be chosen so that they do not meet any of the control regions. 
\end{Remark}

\begin{Remark}
\rm
 It would be interesting to determine whether if this regularity hypothesis is necessary for our abstract results to hold, or if it can be weakened, so that the compatibility conditions $(CC)_n$ for $n \ge 6$ for the application to the wave system can be suppressed or relaxed. 
\end{Remark}

\begin{Remark}
\rm
We can also remark that we make no sign assumptions on the coefficients situated away from the main subdiagonal for the equation ranging from $n+1$ up to $n+p$ for which a direct observation is supposed to hold. The significative coefficients for the transmission of the appropriate information are indeed the ones located on the main subdiagonal for equations which are not directly observed. However, it is important, at least for technical reasons, that the coefficients located away from the main subdiagonal are vanishing for columns which are ranging from $1$ up to $n-2$.
\end{Remark}

\subsubsection{Main controllability results for $n+p$-coupled cascade wave systems}
We now consider the $n+p$-coupled dual control cascade system subjected to $p+1$ controls. For this system, we shall consider three different cases: all the $p+1$ controls are locally distributed,
all the $p+1$ controls are localized on parts of the boundary and finally the case of mixed
locally distributed and boundary controls. These three situations are described as follows.
\begin{itemize}
\item[(i)]
 {\bf Locally distributed controls}.  Assume that
$$
(H2)
\begin{cases}
b_{n+k} \ge 0 \mbox{ on } \Omega \mbox{ for } k=0, \ldots, p\,,\\
\{b_{n+k} >0\} \supset \overline{\omega_{n+k}} \mbox{ for some subsets } \omega_{n+k} \subset \Omega
\mbox{ for } k=0, \ldots, p\,,\\
b_{n+k} \in L^{\infty}(\Omega) \mbox{ for all } j \in \{0, \ldots, p\} \,.
\end{cases}
$$ 
\item[(ii)]
 {\bf Boundary controls}. Assume that
$$
(H3)
\begin{cases}
b_{n+k} \ge 0 \mbox{ on } \Gamma \mbox{ for } k=0, \ldots, p\,,\\
\{b_{n+k} >0\} \supset \overline{\Gamma_{n+k}} \mbox{ for some subsets } \Gamma_{n+k} \subset \Gamma
\mbox{ for } k=0, \ldots, p\,,\\
b_{n+k} \in L^{\infty}(\Gamma) \mbox{ for all } j \in \{0, \ldots, p\} \,.
\end{cases}
$$
\item[(iii)]
 {\bf Mixed boundary and locally distributed observability}. 
Assume that
$$
(H4)
\begin{cases}
b_{n+k} \ge 0 \mbox{ on } \Gamma \mbox{ for } k=0, \ldots, q\,,\\
\{b_{n+k} >0\} \supset \overline{\Gamma_{n+k}} \mbox{ for some subsets } \Gamma_{n+k} \subset \Gamma
\mbox{ for } k=0, \ldots, q\,,\\
b_{n+k} \in L^{\infty}(\Gamma) \mbox{ for all } j \in \{0, \ldots, q\} \,,\\
b_{n+k} \ge 0 \mbox{ on } \Omega \mbox{ for } k=q+1, \ldots, p\,,\\
\{b_{n+k} >0\} \supset \overline{\omega_{n+k}} \mbox{ for some subsets } \omega_{n+k} \subset \Omega 
\,, k=q+1, \ldots, p\,,\\
b_{n+k} \in L^{\infty}(\Omega) \mbox{ for all } j \in \{q+1, \ldots, p\} \,,
\end{cases}
$$  
\end{itemize}
where $\omega_{n+k}$ are open subsets. We shall start with the case of $p+1$ locally distributed controls. For the sake of length, we give the control problem under a matrix operator form.
\begin{equation}\label{nmixedwavecont}
\begin{cases}
y^{\prime\prime} + \mathcal{M}_{n+p}y=(0, \ldots, b_n v_n, b_{n+1} v_{n+1}, \ldots, b_{n+p}v_{n+p})^t \,,\\
y(0)=(y_1^0, \ldots, y_{n+p}^0)\,, y^{\prime}(0)=(y_1^1, \ldots, y_{n+p}^1) \,,
\end{cases}
\end{equation}
 where the matrix operator $\mathcal{M}_{n+p}$ is given in \eqref{Mn+p}, $A$ stands for the homogeneous Dirichlet Laplacian,
$y=(y_1,\ldots,y_{n+p})^t$ and $v_n, \ldots, v_{n+p}$ are the $p+1$ locally distributed controls. We also use the convention that the first equation has to disappear  if $n=2$. We set
$Y_0=(y_1^0, \ldots, y_{n+p}^0, y_1^1, \ldots, y_{n+p}^1)$ and denote by $Y=(y_1,\ldots, y_{n+p},y_1^{\prime},\ldots, y_{n+p}^{\prime})$ the solution of \eqref{nmixedwavecont} with initial data $Y_0$. Then we have the following exact controllability result.

\begin{Theorem}\label{contwave}
 We assume that the coefficients $c_{ij}$ satisfy the hypothesis $(H1)$ for some open subsets $O_i
\subset \Omega$ for $i=2, \ldots, n$ that satisfy $(GCC)$ for~all $i \in \{2, \ldots, n\}$. 
If $n \ge 4$, we further assume that the coefficients $c_{i i-1}$ for $i=2, \ldots, n-2$ satisfy the compatibility condition
$(CC)_n$.
We also assume~that the coefficients $b_{n+k}$ and the subsets $\omega_{n+k}$ with $k=0, \ldots,p$  satisfy $(H2)$~where the subsets $\omega_{n+k}$  satisfy $(GCC)$ for all $k \in \{0, \ldots,p\}$. Then for all
$T>T^{\ast}$, and all $Y_0 \in  {(\Pi_{i=1}^{n} D(A^{\frac{n-i+1}2})\times (H^1_0(\Omega))^p) \times (\Pi_{i=1}^{n} D(A^{\frac{n-i}2})} \times  {
 (L^2(\Omega))^p)}$, there exist control functions $v_{n+k} \!  \in \!  L^2((0,T);L^2(\Omega))$
for $k \! = \! 0,$ $\ldots, p$ such that the solution $Y$ of \eqref{nmixedwavecont} 
with initial data $Y_0$ satisfies $Y(T) \! = \! 0$.
\end{Theorem}
We now consider the following $n+p$-coupled dual control cascade system with $p+1$ boundary controls.
\begin{equation}\label{nmixedwavecontbd}
\begin{cases}
y_{i,tt} -\Delta y_i +c_{i+1 i}y_{i+1}= 0  \mbox{ in } (0,T) \times \Omega\,,  1 \le i \le n-2 ,
\\
\displaystyle{y_{i,tt} -\Delta y_{i} + 
 \!  \!  \! 
  \sum_{k=i+1}^{n+p}c_{k\, i} y_k=0 \mbox{ in } (0,T)
 \times \Omega , n-1 \le i \le n+p-1} ,\\
y_{n+p,tt} -\Delta  y_{n+p}=0 \mbox{ in } (0,T) \times \Omega ,\\
y_i=0 \mbox{ in } (0,T) \times \Gamma  \mbox{ for }
i=1, \ldots, n-1  , \\
y_{n+k}=b_{n+k} v_{n+k}  \mbox{ in } (0,T) \times \Gamma  \mbox{ for }
k=0, \ldots, p \,, \\
(y_i,y_{i,t})(0)=(y_i^0,y_i^1) \mbox{ for }
i=1, \ldots, n+p \,,
\end{cases}
\end{equation}
 where we use the convention that the first equation has to disappear  if $n=2$.  Then we have the following exact controllability result.
\begin{Theorem}\label{contwavebd}
We assume that the coefficients $c_{ij}$ satisfy the hypothesis $(H1)$ for some open subsets 
$O_i \!
\subset  \!\Omega$ for $i \!= \!2, \ldots, n$ that satisfy $(GCC)$ for~all
 $i \! \in  \!\{2, \ldots, n\}$. 
If $n  \!\ge \! 4$, we further assume that the coefficients $c_{i i-1}$ for $i \!= \!2, $ $\ldots, n-2$ satisfy the compatibility condition
$(CC)_n$.
We also assume~that the coefficients $b_{n+k}$ and the subsets $\Gamma_{n+k}$ with $k=0, \ldots,p$  satisfy $(H3)$ where the subsets $\Gamma_{n+k}$  satisfy $(GCC)$ for all $k \in \{0, \ldots,p\}$. Then for all
$T>T^{\ast}$, and all $Y_0 \in  {(\Pi_{i=1}^{n} D(A^{\frac{n-i}2})\times (L^2(\Omega))^p) \times (\Pi_{i=1}^{n} D(A^{\frac{n-i-1}2})}   {\times (H^{-1}(\Omega))^p)}$, \\ 
 there exist control functions $v_{n+k} \in L^2((0,T);L^2(\Gamma))$
for $k=0, \ldots, p$ such that the solution $Y$ of \eqref{nmixedwavecontbd} with initial data $Y_0$ satisfies $Y(T)=0$.
\end{Theorem}

We now consider the following $n+p$-coupled dual control cascade system with $p+1$ mixed boundary and locally distributed controls.
$$
\begin{cases}
 \label{nmixedwavecontbdmix} 
&
y_{i,tt} -\Delta y_i +c_{i+1 i}\,y_{i+1}= 0  \mbox{ in } (0,T) \times \Omega\,,  1 \le i \le n-2\,,\\
&
\displaystyle{y_{i,tt} -\Delta y_{i} + \sum_{k=i+1}^{n+p}c_{k\, i}\, y_k=0 \mbox{ in } (0,T) \times \Omega\,, n-1 \le i \le n+q}
\,,\\
&
\displaystyle{y_{i,tt} -\Delta y_{i} +\sum_{k=i+1}^{n+p}c_{k\, i} y_k=}  
\displaystyle{ b_iv_i \mbox{ in } (0,T) \times \Omega , n+q+1,\le i \le n+p-1 ,}\\
&
y_{n+p,tt} -\Delta  y_{n+p}=b_{n+p}v_{n+p} \mbox{ in } (0,T) \times \Omega\,,\\
&
y_i=0 \mbox{ in } (0,T) \times \Gamma  \mbox{ for }
i\in \{1, \ldots, n-1\} \cup \{n+q+1, \ldots, n+p\} \,, \\
&
y_{n+k}=b_{n+k} v_{n+k}  \mbox{ in } (0,T) \times \Gamma  \mbox{ for }
k=0, \ldots, q \,, \\
&
(y_i,y_{i,t})(0)=(y_i^0,y_i^1) \mbox{ for }
i=1, \ldots, n+p \,,
\end{cases}
$$
 where we use the convention that the first equation has to disappear  if $n=2$.  Then we have the following exact controllability result.

\begin{Theorem}\label{contwavebdmixed}
We assume that the coefficients $c_{ij}$ satisfy the hypothesis $(H1)$ for some open subsets $O_i
\subset \Omega$ for $i=2, \ldots, n$ that satisfy $(GCC)$ for all $i \in \{2, \ldots, n\}$. 
If $n \ge 4$, we further assume that the coefficients $c_{i i-1}$ for $i=2, \ldots, n-2$ satisfy the compatibility condition
$(CC)_n$.
We also assume that the coefficients $b_{n+k}$, the subsets $\Gamma_{n+k}$ for $k=0, \ldots,q$  and
the subsets $\omega_{n+k}$ for $k=q+1, \ldots, p$ satisfy $(H4)$ where the subsets $\Gamma_{n+k}$ for $k=0, \ldots,q$ and $\omega_{n+k}$   for $k=q+1, \ldots, p$ satisfy $(GCC)$. Then for all
$T>T^{\ast}$, and all $Y_0 \in \displaystyle{(\Pi_{i=1}^{n} D(A^{(n-i)/2})\times (L^2(\Omega))^q \times (H^1_0(\Omega))^{p-q}) \times (\Pi_{i=1}^{n} D(A^{(n-i-1)/2}) 
\times }\ \\$
$\displaystyle{(H^{-1}(\Omega))^q) \times (L^2(\Omega))^{p-q})}$,
 there exist control functions $v_{n+k} \in L^2((0,T);\ \\$$L^2(\Gamma))$
for $k=0, \ldots, q$ and  $v_{n+k} \in L^2((0,T);\ \\$
$L^2(\Omega))$
for $k=q+1, \ldots, p$ such that the solution $Y$ of \eqref{nmixedwavecontbd} with initial data $Y_0$ satisfies $Y(T)=0$.
\end{Theorem}
\subsection{Boundary and localized controllability of $n+p$-coupled cascade heat and Schr\"odinger equations with localized couplings}
We now consider the following $n+p$-coupled locally control cascade heat-type system. We shall first consider the case of $p+1$ locally distributed controls.
\begin{equation}\label{nmixediffcont}
\begin{cases}
e^{i \theta}y^{\prime} + \mathcal{M}_{n+p}y=(0, \ldots, b_n v_n, b_{n+1} v_{n+1}, \ldots, b_{n+p}v_{n+p})^t \,,\\
y(0)=(y_1^0, \ldots, y_{n+p}^0) \,,
\end{cases}
\end{equation}
 where we use the convention that the first equation has to disappear  if $n=2$. We set
$Y_0=(y_1^0, \ldots, y_{n+p}^0)$. We recover $n+p$-coupled heat (resp. Schr\"odinger) cascade systems when $\theta=0$ (resp. $\theta=\pm \pi/2$) and diffusive coupled cascade systems when
$\theta \in (-\pi/2, \pi/2)$. 
Then we have the following exact controllability result.

\begin{Corollary}\label{contheat}
We assume that the coefficients $c_{ij}$ satisfy the hypothesis $(H1)$ for some open subsets $O_i
\subset \Omega$ for $i=2, \ldots, n$ that satisfy $(GCC)$ for all $i \in \{2, \ldots, n\}$. 
If $n \ge 4$, we further assume that the coefficients $c_{i i-1}$ for $i=2, \ldots, n-2$ satisfy the compatibility condition
$(CC)_n$.
We also assume that the coefficients $b_{n+k}$ and the subsets $\omega_{n+k}$ with $k=0, \ldots,p$  satisfy $(H2)$ where the subsets $\omega_{n+k}$  satisfy $(GCC)$ for all $k \in \{0, \ldots,p\}$. Then, 
the following properties hold
\begin{itemize}
\item[(i)]
 The case $\theta \in (-\pi/2,\pi/2)$ (Heat type systems). We have  for all
$T>0$, and all $Y_0 \in ( L^2(\Omega))^{n+p}$, there exist control functions $v_{n+k} \in L^2((0,T);L^2(\Omega))$
for $k=0, \ldots, p$ such that the solution $Y=(y_1,$ $\ldots, y_{n+p})$ of \eqref{nmixediffcont} with initial data $Y_0$ satisfies $Y(T)=0$.
\item[(ii)] The case $\theta=\pm \pi/2$ (Schr\"odinger systems). We have for all $T>0$ and all
$Y_0 \in \displaystyle{\Pi_{i=1}^n D(A^{(n-i)/2}) \times (L^2(\Omega))^p}$, there exist control functions $v_{n+k} \in L^2((0,T);L^2(\Omega))$
for $k=0, \ldots, p$ such that the solution $Y=(y_1,\ldots, y_{n+p})$ of \eqref{nmixediffcont} with initial data $Y_0$ satisfies $Y(T)=0$.
\end{itemize}
\end{Corollary}

We now consider the following $n+p$-coupled control cascade heat-type system with $p+1$ boundary controls.
\begin{equation}\label{nmixediffcontbd}
\begin{cases}
e^{i \theta}y_{i,t} -\Delta y_i +c_{i+1 i}y_{i+1}= 0  \mbox{ in } (0,T) \times \Omega\,,  1 \le i \le n-2 ,
\\
\displaystyle{e^{i \theta}y_{i,t}  \! - \! \Delta y_{i} 
+ \!   \!\!  \sum_{k=i+1}^{n+p}  \!  \!  c_{k  i} y_k=0 \mbox{ in } (0,T) \times \Omega , n-1 \le i \le n+p-1}
 ,\\
e^{i \theta}y_{n+p,t} -\Delta  y_{n+p}=0 \mbox{ in } (0,T) \times \Omega\,,\\
y_i=0 \mbox{ in } (0,T) \times \Gamma  \mbox{ for }
i=1, \ldots, n-1 \,, \\
y_{n+k}=b_{n+k} v_{n+k}  \mbox{ in } (0,T) \times \Gamma  \mbox{ for }
k=0, \ldots, p \,, \\
y_i(0)=y_i^0 \mbox{ for } i=1, \ldots, n+p \,,
\end{cases}
\end{equation}
 where we use the convention that the first equation has to disappear  if $n=2$.  Then we have the following exact controllability result.

\begin{Corollary}\label{contdiffbd}
We assume that the coefficients $c_{ij}$ satisfy the hypothesis $(H1)$ for some open subsets $O_i
\subset \Omega$ for $i=2, \ldots, n$ that satisfy $(GCC)$ for all $i \in \{2, \ldots, n\}$. 
If $n \ge 4$, we further assume that the coefficients $c_{i i-1}$ for $i=2, \ldots, n-2$ satisfy the compatibility condition
$(CC)_n$.
We also assume that the coefficients $b_{n+k}$ and the subsets $\Gamma_{n+k}$ with $k=0, \ldots,p$  satisfy $(H3)$ where the subsets $\Gamma_{n+k}$  satisfy $(GCC)$ for all $k \in \{0, \ldots,p\}$. Then 
we have
\begin{itemize}
\item[(i)] The case $\theta \in (-\pi/2,\pi/2)$ (Heat type systems). We have  for all
$T>0$, and all $Y_0 \in ( H^{-1}(\Omega))^{n+p}$, there exist control functions $v_{n+k} \in L^2((0,T);L^2(\Gamma))$
for $k=0, \ldots, p$ such that the solution $Y=(y_1,\ldots, y_{n+p})$ of \eqref{nmixediffcontbd} with initial data $Y_0$ satisfies $Y(T)=0$.
\item[(ii)] The case $\theta=\pm \pi/2$ (Schr\"odinger systems). We have for all $T>0$ and all
$Y_0 \in \displaystyle{\Pi_{i=1}^n D(A^{(n-i-1)/2}) \times (H^{-1}(\Omega))^p}$, there exist control functions $v_{n+k} \in L^2((0,T);L^2(\Omega))$
for $k=0, \ldots, p$ such that the solution $Y=(y_1,\ldots, y_{n+p})$ of \eqref{nmixediffcontbd} with initial data $Y_0$ satisfies $Y(T)=0$.
\end{itemize}
\end{Corollary}
We now consider the following $n+p$-coupled control cascade heat-type system with $p+1$ mixed boundary and locally distributed controls.
\begin{equation}\label{nmixediffcontbdmix}
\begin{cases}
e^{i \theta}y_{i,t} -\Delta y_i +c_{i+1 i}\,y_{i+1}= 0  \mbox{ in } (0,T) \times \Omega\,,  1 \le i \le n-2\,,
\\
\displaystyle{e^{i \theta}y_{i,t} \!- \!\Delta y_{i} \! + \!
 \! \! \! \!
  \sum_{k=i+1}^{n+p} \!  \!\! c_{k\, i}  y_k=0 \mbox{ in } (0,T)  \! \times \!  \Omega , n-1 
   \! \le \! i  \! \le n+q} ,\\
\displaystyle{e^{i \theta}y_{i,t} -\Delta y_{i} + \sum_{k=i+1}^{n+p}c_{k\, i} y_k=
}\\
\displaystyle{b_iv_i \mbox{ in } (0,T) \times \Omega\,, n+q+1 \le i \le n+p-1} ,
\\
e^{i \theta}y_{n+p,t} -\Delta  y_{n+p}=b_{n+p}v_{n+p} \mbox{ in } (0,T) \times \Omega ,
\\
y_i \! = \! 0 \mbox{ in } (0,T)  \! \times \!  \Gamma  \mbox{ for }
i \! \in  \! \{1, \ldots, n \! - \! 1\} \!  \cup \!  \{n+q+1, \ldots, n+p\}  , \\
y_{n+k}=b_{n+k} v_{n+k}  \mbox{ in } (0,T) \times \Gamma  \mbox{ for }
k=0, \ldots, q  , \\
(y_i,y_{i,t})(0)=(y_i^0,y_i^1) \mbox{ for }
i=1, \ldots, n+p \,,
\end{cases}
\end{equation}
 where we use the convention that the first equation has to disappear  if $n=2$.  Then we have the following exact controllability result.

\begin{Corollary}\label{contdiffbdmixed}
We assume that the coefficients $c_{ij}$ satisfy the hypothesis $(H1)$ for some open subsets $O_i
\subset \Omega$ for $i=2, \ldots, n$ that satisfy $(GCC)$ for all $i \in \{2, \ldots, n\}$. 
If $n \ge 4$, we further assume that the coefficients $c_{i i-1}$ for $i=2, \ldots, n-2$ satisfy the compatibility condition
$(CC)_n$.
We also assume that the coefficients $b_{n+k}$, the subsets $\Gamma_{n+k}$ for $k=0, \ldots,q$  and
the subsets $\omega_{n+k}$ for $k=q+1, \ldots, p$ satisfy $(H4)$ where the subsets $\Gamma_{n+k}$ for $k=0, \ldots,q$ and $\omega_{n+k}$   for $k=q+1, \ldots, p$ satisfy $(GCC)$. Then 
we have
\begin{itemize}
\item[(i)] The case $\theta \in (-\pi/2,\pi/2)$ (Heat type systems). We have  for all
$T>0$, and all $Y_0 \in (H^{-1}(\Omega))^{n+q}\times ( L^2(\Omega))^{p-q}$, there exist control functions $v_{n+k} \in L^2((0,T);L^2(\Gamma))$
for $k=0, \ldots, q$ and  $v_{n+k} \in L^2((0,T);L^2(\Omega))$
for $k=q+1, \ldots, p$ such that the solution $Y=(y_1,\ldots, y_{n+p})$ of \eqref{nmixediffcontbdmix} with initial data $Y_0$ satisfies $Y(T)=0$.
\item[(ii)] The case $\theta=\pm \pi/2$ (Schr\"odinger systems). We have for all $T>0$ and all
$Y_0 \in \displaystyle{\Pi_{i=1}^n D(A^{(n-i-1)/2}) \times (H^{-1}(\Omega))^q \times (L^2(\Omega))^{p-q}}$, there exist control functions $v_{n+k} \in L^2((0,T);L^2(\Gamma))$
for $k=0, \ldots, q$ and  $v_{n+k} \in L^2((0,T);L^2(\Omega))$
for $k=q+1, \ldots, p$ such that the solution $Y=(y_1,\ldots, y_{n+p})$ of \eqref{nmixediffcontbdmix} with initial data $Y_0$ satisfies $Y(T)=0$.
\end{itemize}
\end{Corollary}

\begin{Remark}
\rm
The condition $(GCC)$ is not natural for the scalar heat equation, so that the results presented above are probably not optimal for $n$-coupled cascade systems. It is also known that $(GCC)$ is not necessary for scalar Schr\"odinger equation in a rectangle (see~\cite{jaffard, tenentucs}). However, the above results are the first ones valid in a multidimensional setting, for localized as well as boundary control, and for control regions which do not meet the coupling regions for $n=3$.
\end{Remark}

\begin{Remark}
\rm
As in~\cite{alaleau11} (see Remark~1.6), the above results apply for $n=3$ to other boundary conditions such as Neumann or Fourier conditions. They also apply more generally for a diffusion operator of the form $-div (c \nabla )$, where $c$ is a sufficiently smooth positive symmetric matrix (see Remark~1.7 in~\cite{alaleau11}). 
\end{Remark}
\section{Proofs of the main applicative results}
\subsection{Proofs of the results for coupled cascade wave systems}

We begin with the proof of Theorem~\ref{ncoupledwavethm}.

\smallskip

\noindent
\textbf{Proof of Theorem~\ref{ncoupledwavethm}}.
This is an application of the abstract Theorem~\ref{mixedsystem}. Here $H=L^2(\Omega)$, and
$A$ is given by $A=-\Delta$ and $D(A)=H^2(\Omega) \cap H^1_0(\Omega)$. The sets $H_k=
D(A^{k/2})$ are given by \eqref{iterateDA}. The operators $C_{ij}$ are the multiplication operators 
in $L^2(\Omega)$ by the corresponding functions $c_{ij}$ and are therefore bounded and self-adjoint. Thanks to the smoothness 
and the compatibility assumptions on the  coefficients $c_{i i-1}$, $C_{i i-1}^{\ast} \in \mathcal{L}(H_k)$ for all
$k=0, \ldots, n-i+1$. Thanks to $(H1)$, the assumption $(A2)_n$ holds with $\Pi_i=
1_{O_i}$ and $\alpha_i=\inf_{O_i}(c_{i i-1})>0$. Morever since the sets $O_i$ satisfies
$(GCC)$,  $(A3)_n$ also holds. We shall check the assumptions on the observability operators.
 First case: $j \in I_{int}$ where $j \in \{0, \ldots, p\}$. Then we have $G_{n+j}=L^2(\Omega)$,
and $B_{n+j}^{\ast}$ is the multiplication in $L^2(\Omega)$ by the bounded function $b_{n+j}$. Therefore $B_{n+j}^{\ast}$ is a bounded symmetric operator in $H$, so that
$(A4)_{n+j}$ holds. Thanks to the assumptions
on $b_{n+j}$ when $j\in I_{int}$ and since $\omega_{n+j}$ satisfies $(GCC)$, we deduce
that $(A5)_{n+j}$ also holds using \cite{blr92}.

Second case: $j \in I_{bd}$ where $j \in \{0, \ldots, p\}$. Then we have $G_{n+j}=L^2(\Gamma)$,
and $B_{n+j}^{\ast} \in \mathcal{L}(H^2(\Omega) \cap H^1_0(\Omega);L^2(\Gamma))$ is defined as
$ 
B_{n+j}^{\ast}u=b_{n+j} \frac{ \partial u}{\partial \nu} $
$ u \in H^2(\Omega) \cap H^1_0(\Omega) .
$ 
 Thanks to the well-known hidden regularity result of~\cite{lions}, $B_{n+j}^{\ast}$ satisfies $(A4)_{n+j}$. Thanks to the assumptions
on $b_{n+j}$ when $j\in I_{bd}$ and since $\Gamma_{n+j}$ satisfies $(GCC)$, we deduce
that $(A5)_{n+j}$ also holds using \cite{blr92}. Hence we can apply Theorem~\ref{mixedsystem}, which gives the desired result.

We now shall prove the controllability results for coupled cascade wave systems.

\smallskip

\noindent
\textbf{Proof of Theorem~\ref{contwave}}.
Thanks to our hypotheses and thanks to the above proof the assumption of
Theorem~\ref{mixedsystem} are satisfied. Hence we apply the part $(i)$ of Theorem~\ref{control2n+pq}. This gives the desired result.

\smallskip

\noindent
\textbf{Proof of Theorem~\ref{contwavebd}}
Thanks to our hypotheses and thanks to the above proof the assumption of
Theorem~\ref{mixedsystem} are satisfied. Hence we apply the part $(ii)$ of Theorem~\ref{control2n+pq}. This gives the desired result.

\smallskip

\noindent
\textbf{Proof of Theorem~\ref{contwavebdmixed}}.
Thanks to our hypotheses and thanks to the above proof the assumption of
Theorem~\ref{mixedsystem} are satisfied. Hence we apply Theorem~\ref{control2n+pqmixed}. This gives the desired result.

\subsection{Proofs of the results for coupled cascade heat and Schr\"odinger equations}
We start with the proof of Corollary~\ref{contheat} using the transmutation method.

\smallskip

\noindent
\textbf{Proof of Corollary~\ref{contheat}}.
We proceed as in~\cite{alaleau11}. 
Proof of $(i)$ We first apply no control on the time interval $(0,T/2)$, so that by
the smoothing effect of the heat equation, the initial data $Y^0\in (L^2(\Omega))^{n+p}$ is driven to
$Y_{|t=T/2} \in  \displaystyle{\Pi_{i=1}^{n} D(A^{(n-i+1)/2})\times (H^1_0(\Omega))^p}$. We then combine Theorem~\ref{contwave} together with the transmutation result  given by Miller in~\cite{miller04} to prove that there exist control $v_{n+k}$ for $k=0, \ldots, p$ such that $Y(T)=0$.
Proof of $(ii)$. It is similar to the case $(i)$ except that we work directly on the time interval $(0,T)$ since no smoothing effect holds in the case of the Schr\"odinger equation. Combining Theorem~\ref{contwave} together with the transmutation result given by Miller in~\cite{miller05} (see Theorem 3.1), we conclude the proof.

\smallskip

\noindent
\textbf{Proof of Corollary~\ref{contdiffbd}}.
We proceed as in~\cite{alaleau11}. 
Proof of $(i)$ We first apply no control on the time interval $(0,T/2)$, so that by
the smoothing effect of the heat equation, the initial data $Y^0 \in (H^{-1}(\Omega))^{n+p}$ is driven to
$Y_{|t=T/2} \in  \displaystyle{\Pi_{i=1}^{n} D(A^{(n-i)/2})\times (L^2(\Omega))^p}$. We then combine Theorem~\ref{contwavebd} together with the transmutation result given by Miller in~\cite{miller} (see Theorem 3.4) to get the desired result.

Proof of $(ii)$. It is similar to the case $(i)$ except that we work directly on the time interval $(0,T)$ since no smoothing effect holds in the case of the Schr\"odinger equation. Combining Theorem~\ref{contwavebd} together with the transmutation result (Theorem 3.1) given by Miller in~\cite{miller05} (see Theorem 3.1), we conclude the proof.

\smallskip

\noindent
\textbf{Proof of Corollary~\ref{contdiffbdmixed}}.
We combine the above results for the boundary and locally distributed controls and use the appropriate version of the transmutation method as above.

\section{Discussion, generalizations and further questions}

We show in this paper that the two-level energy method introduced in~\cite{alacras01,sicon03} and further extended and simplified in~\cite{alaleaucont, alaleau11} and~\cite{alabau2cascade} can be adapted to handle two subclasses of coupled cascade systems: the bi-diagonal
$n$-coupled cascade and mixed $n+p$-coupled cascade systems.  We prove positive general boundary and locally distributed observability and control results  through a generalization of the two-level energy method into a hierarchic multi-level energy method. It is a constructive method, since it does not rely on contradiction arguments to get the desired observability inequalities. We give several applications of these results to $n+p$-coupled cascade wave, heat and Schr\"odinger systems. 

The main features of these results are that they are valid in a multi-dimensional framework, for locally distributed as well as boundary controls/observations, and for localized couplings. Furthermore in the one-dimensional case and if $n \le 5$ in the multi-dimensional case, they are valid in situations for which the control/observations regions do not meet the localized coupling regions.
In dimensions larger than $2$ and for $n \ge 6$,  the supports of the $n-4$ first coefficients will have to meet the control regions. However the intersections between the supports of these coefficients  and the control regions can be made as small as possible (located on some small neighborhood of some parts of the boundary). 

These two subclasses of coupled cascade system are a toy model prior to a more general study. The proof of the result for the bi-diagonal $n$-coupled cascade system is somehow involved and requires a sharp analysis of the way the information is transferred
from the last observed equation to the $n-1$ other unobserved equations. 
Indeed the study of general $n$-coupled cascade systems or of full $n$-coupled systems under sharp
geometric conditions, optimal conditions on the coupling coefficients is a very involved open question which will require further sharp analysis.
The extension of the multi-level energy methods to other examples and to a larger class of $n$-coupled cascade systems is under study.
In particular, the generalization to a full cascade system by a single control involves other hypotheses on the coupling coefficients situated away fom the sup-diagonal. The present generalization to a a full cascade system of order $3$, that is for a non vanishing coefficient $c_{ 31}$ in the $3$-coupled
cascade system, and thus for a matrix operator of the form 
$$
\mathcal{M}_3=
\left(
\begin{matrix}
A &  c_{21}I & c_{31}I\\
0 & A & c_{32} I  \\
0 & 0 &  A 
\end{matrix} \right)
$$ 
 is a work in progress. It requires an extension of the multi-level energy method to transfer the information through a larger
band away from the diagonal.

In former works~\cite{alacras01, sicon03, alaleaucont, alaleau11} we studied another class of coupled systems, namely the $2$-coupled symmetric systems. We may compare this class with the class of cascade systems studied in this paper as follows:
 
\begin{itemize}
\item If the initial data of the first (unobserved) component in a $2$-coupled cascade system is vanishing, then by uniqueness this component is vanishing at all time, so that the coupled system reduces to a scalar wave equation with a usual observability hypothesis. In this case, the desired observability inequality is trivial. The same property holds for $n$-coupled cascade systems if the initial data of the first $n-1$ components are vanishing. This shows in particular that the notion of {\em partial observability} as introduced in~\cite{lions} is trivial for coupled cascade system. This property no longer holds true for $2$-symmetric coupled systems, which are in some way {\em more coupled} than cascade systems. 

\item  A furthermore important difference is that the total energies (weakened and natural) of the solutions of $2$-coupled symmetric systems are conserved. We strongly used this property in~\cite{alacras01, sicon03, alaleaucont, alaleau11}. This property does not hold true for $2$-coupled cascade systems and more generally for $n$-coupled systems, only the energy of the first component is conserved through time. Nevertheless, we could extend the two-levels energy method to handle this loss of conservative properties for the full state variable. 

\item Finally, the results for $2$-coupled symmetric systems hold under a smallness condition on the coupling coefficient which is not longer requested for coupled cascade systems.
\end{itemize}

The results presented in this paper and existing results in the literature lead to several open questions. We shall give some of these open questions.

Our results on coupled cascade wave systems (see also~\cite{alaleaucont, alaleau11}) rely on smoothness assumptions on the coupling coefficients. It is interesting to note that the results on $2$-coupled cascade wave systems in~\cite{RdT11} --valid for one-dimensional domains in the case of $2$-coupled cascade wave systems (and
to multidimensional domains for $2$-coupled cascade Schr\"odinger with periodic boundary conditions)-- are stated for a coupling term which is the
characteristic function of the observation region. Hence it is not a smooth coupling coefficient. This question is also linked to the necessity or not of the compatibility conditions $(CC)_n$ when $n \ge 4$. It would be interesting to explore this question linked to the {\em hierarchic} approach of our multi-level energy method. Such results, if positive, would allow controllability results in dimensions larger than $2$ and for a number of equations larger than $6$, in geometrical situations where none of the coupling regions meet the controls regions. 

Also, 
the necessary and sufficient abstract condition under which our
results are valid, allows us to get results under sharp geometric
 conditions, issued from micro-local analysis, or spectral and 
 frequency approaches (or also multipliers methods) for wave coupled
  systems. These conditions are sharp for wave systems at least for $n \le 3$ ,
 i.e., they are sufficient conditions and are almost necessary
 in general. However, these geometric conditions are probably not
  optimal in the parabolic case and it would be an important 
  step to understand which geometric conditions on the control
   and coupling regions are optimal for coupled cascade
    parabolic systems. 

Generalizations to other classes of coupled 
PDE's --{\em to be identified}--, to nonconservative systems,
 to systems with distinct operators on the main diagonal
 (see~\cite{sicon03} for some results in this direction) \ldots
are challenging questions. 

Many applications involve the control of nonlinear coupled systems by a reduced number of controls
(see e.g. \cite{CGR10}). Thus,  it is a challenging question to understand also nonlinear phenomenon, and in particular to identify classes of nonlinear coupled systems for which  linearization arguments or the Coron's return method~\cite{coronbk07} (when controllability for
the linearized system fails) can be used to obtain positive controllability results for these nonlinear models by a reduced number of controls (see e.g.~\cite{CGR10}). 

Various other questions linked to the cost of controls in the spirit of~\cite{miller, EZ2}, frequency analysis \ldots of such coupled systems are of interest.  Several different mathematical approaches can lead to complementary results, as in the present paper which combines a performant and robust multi-level energy method and
results from micro-local analysis on geometrical aspects. Two key properties for making possible this combination are the invariance by translation in time of the systems and the quadratic dependence of the observation with respect to state variable. We have already used in~\cite{alaamm11} --in a different way and for different purposes-- these properties for nonlinear stabilization, combining an optimal-convexity method with results from micro-local analysis. Already, one can remark that the first property is lost if the coupling coefficients are also depending on time since time invariance not longer holds. Thus dealing with time and space dependent coupling operators is also an important open and challenging question.

\section{Appendix: proof of the controllability results for mixed
 bi-diagonal and non bi-diagonal cascade systems}
 
\subsubsection{Proofs of the main results for observability of $n+p$-coupled cascade systems}
  
\textbf{Proof of Theorem~\ref{mixedsystem}}.
We will denote by $D(T)$ generic positive constants depending on $T$. We prove \eqref{admissineqimixed} as follows. First we shall prove by induction on $j$ that for
any $j \in \{1,\ldots,p\}$, we have
\begin{align}
\label{DIRJ}
\int_0^T |u_{n+j}|^2 \,dt & \le C \int_0^T  e_1(U_{n+j})(t)\,dt \\
&
\notag
\le D(T) \Big(\sum_{i=1}^n e_{1+i-n}(U_i)(0) + 
\sum_{l=1}^j e_1(U_{n+l})(0)
\Big)\,.
\end{align}
 We prove this inequality for $j=1$ as follows. Using classical energy estimates for the equation in $u_{n+1}$, we obtain
$$
\int_0^T  e_1(U_{n+1})(t)\,dt \le D(T) \Big(e_1(U_{n+1})(0) +\int_0^T (|u_{n-1}|^2 + |u_n|^2)\,dt
\Big)\,.
$$
 Using \eqref{admiss*} and \eqref{admissx*} in this last inequality, we obtain
\begin{align*}
\int_0^T   |u_{n+1}|^2  dt  &   \le C \int_0^T e_1(U_{n+1})(t)\,dt \\
&
 \le D(T) \Big(
e_1(U_{n+1})(0) + 
\sum_{i=1}^n e_{1+i-n}(U_i)(0)
\Big) .
\end{align*}
 Hence the claimed property holds for $j=1$. Assume that it holds for $k$, for all
$k \in \{ 1,\ldots, j-1\}$. 
Considering the equation in $u_{n+j}$ and thanks to classical energy estimates, we have
\begin{align*}
&
\int_0^T  e_1(U_{n+j})(t)\,dt \\
 &
\le D(T) \Big(e_1(U_{n+j})(0) +  
\int_0^T (|u_{n-1}|^2 + |u_n|^2
+\sum_{k=1}^{j-1} |u_{n+k}|^2)\,dt
\Big)\,.
\end{align*}
  Using the property \eqref{DIRJ} for $k=1$ up to $j-1$ and  \eqref{admiss*} and \eqref{admissx*} , we obtain
\begin{align*}
\int_0^T |u_{n+j}|^2 \,dt &
\le C \int_0^T  e_1(U_{n+j})(t)\,dt \\
&
\le D(T) \Big(\sum_{i=1}^n e_{1+i-n}(U_i)(0) + 
\sum_{l=1}^j e_1(U_{n+l})(0)
\Big)\,.
\end{align*}
  Hence we proved \eqref{DIRJ} for any $j \in \{1, \ldots, p\}$.
Let $j$ be any integer in $\{1, \ldots, p\}$. Thanks to the assumption $(A4)_{n+j}$ and
using classical energy estimates for the equation in $u_{n+j}$ we have
\begin{align*}
&
\int_0^T ||\mathcal{\mathbf{B^{\ast}_{n+j}}}U_{n+j}||_{G_{n+j}}^2 \,dt
\\
&
 \le D(T)\Big(
e_1(U_{n+j})(0) + 
\int_0^T (|u_{n-1}|^2 + |u_n|^2 + \sum_{k=1}^{j-1} |u_{n+k}|^2)\,dt
\Big)\,.
\end{align*}
 Using \eqref{admiss*}, \eqref{admissx*} and the property \eqref{DIRJ} for $j-1$, we obtain
$$
\int_0^T ||\mathcal{\mathbf{B^{\ast}_{n+j}}}U_{n+j}||_{G_{n+j}}^2 \,dt \le D(T)\Big(
\sum_{i=1}^n e_{1+i-n}(U_i)(0) +
\sum_{l=1}^j e_1(U_{n+l})(0)
\Big)\,.
$$
 We now turn to the proof of the observability inequalities.
The estimate \eqref{eqobskmixed} has already been proved in Theorem~\ref{obsNSHn}
for $T>T_n^{\ast}$. We prove \eqref{eqobskmixed2} as follows.  We set
$T_{n+p}^{\ast}=\max (T_n^{\ast}, \max_{1 \le j \le p}(T_{0,n+j}))$.
We assume from now on that $T>T_{n+p}^{\ast}$.
Thanks to  \eqref{eqobskmixed} and to \eqref{admiss*} and \eqref{admissx*}, we have
\begin{equation}\label{8x}
\int_0^T \big(|u_{n-1}|^2 + |u_n|^2\big)\,dt \le d_n(T) \int_0^T \| \mathcal{\mathbf{B^{\ast}_{n}}}(U_{n}) \|_{G_{n}}^2 dt \,,
\end{equation}
 where $d_n(T)$ depends on $d_{i,n}(T)$ for $i \in\{1, \ldots,n\}$ and the constant $C(T)$
in \eqref{admiss*}-\eqref{admissx*}.

We shall prove by induction on $j \in \{1,\ldots,p\}$, the following property
\begin{equation}\label{x}
\begin{cases}
 \displaystyle{e_{1}(U_{n+k})(0) \le \rho_{n,k}(T) \sum_{l=0}^k\int_0^T \| \mathcal{\mathbf{B^{\ast}_{n+l}}}(U_{n+l}) \|_{G_{n+l}}^2 dt \,,
\forall \ k=1, \ldots, j \,,}\\
 \displaystyle{\int_0^T |u_{n+k}|^2 \,dt \le s_{n,k}(T) \sum_{l=0}^k \int_0^T ||\mathcal{\mathbf{B^{\ast}_{n+l}}}(U_{n+l}) ||_{G_{n+l}}^2 dt \,, \forall \ k=1, \ldots, j \,,}
\end{cases}
\end{equation}
 where $s_{n,k}(T)>0$ are explicit constants.
Let us prove that \eqref{x} holds for $j=1$. Thanks to the equation for $u_{n+1}$ in
\eqref{NSHnmixed}, we obtain as for  the case of $2$ bi-diagonal
cascade systems
$$
(1+T)\int_0^Te_1(U_{n+1})\,dt \ge Te_1(U_{n+1})(0) - CT\int_0^T(|u_{n-1}|^2 +|u_n|^2)\,dt\,.
$$
 Using \eqref{8x} in this estimate, we obtain
\begin{equation}\label{5x}
(1+T)\int_0^Te_1(U_{n+1})\,dt \ge Te_1(U_{n+1})(0) - CTd_n(T)\int_0^T\| \mathcal{\mathbf{B^{\ast}_{n}}}(U_{n}) \|_{G_{n}}^2
\,dt\,.
\end{equation}
 On the other hand, applying the uniform observability estimate given
in Lemma~\ref{AL}  with a second
member for the equation for $u_{n+1}$ in
\eqref{NSHnmixed}, we deduce that there exist $\eta_{n+1}>0$ and $\delta_{n+1}>0$, such that
\begin{align*}
&
\eta_{n+1}\int_0^T \| \mathcal{\mathbf{B^{\ast}_{n+1}}}(U_{n+1}) \|_{G_{n+1}}^2 dt
\\
&
 \ge
\int_0^T e_1(U_{n+1})\,dt -  
\delta_{n+1} \int_0^T|C_{n+1\,n-1}u_{n-1} + 
C_{n+1\, n}u_n|^2\,dt \,.
\end{align*}
 Using \eqref{8x} in this last estimate together with \eqref{5x}, we obtain 
for $j=1$
\begin{equation}\label{9x}
\displaystyle{e_{1}(U_{n+1})(0) \le \rho_{n,1}(T) \sum_{l=0}^1\int_0^T \| \mathcal{\mathbf{B^{\ast}_{n+l}}}(U_{n+l}) \|_{G_{n+l}}^2 dt \,.}
\end{equation}
 On the other hand, thanks to the equation verified by $u_{n+1}$ and as before, we have
$$
\int_0^T |u_{n+1}|^2 \,dt \le C (e_1(U_{n+1})(0) + \int_0^T(|u_{n-1}|^2 +|u_n|^2)\,dt )\,.
$$
 Using \eqref{8x} and \eqref{9x} in this estimate we obtain
$$
\displaystyle{\int_0^T |u_{n+1}|^2 \,dt \le s_{n,1}(T) \sum_{l=0}^1 \int_0^T ||\mathcal{\mathbf{B^{\ast}_{n+l}}}(U_{n+l}) ||_{G_{n+l}}^2 dt \,.}
$$
 Thus, we proved \eqref{x} for $j=1$. Let us now assume that \eqref{x} holds
for $j-1$, we shall prove that it holds for $j$ proceeding as for the case $j=1$.
First the usual estimates for the equation in $u_{n+j}$ lead to
\begin{align*}
&
 (1+T)\int_0^Te_1(U_{n+j})\,dt \\
 &
 \ge Te_1(U_{n+j})(0) -  
\displaystyle{CT\int_0^T(|u_{n-1}|^2 +|u_n|^2
+\sum_{k=1}^{j-1} |u_{n+k}|^2)\,dt\,.}
\end{align*}
 Using \eqref{8x} together with our induction hypothesis for $j-1$, we deduce that
\begin{align}
\label{10x}
&
 (1+T)\int_0^Te_1(U_{n+j})\,dt \\
 &
 \notag \ge Te_1(U_{n+j})(0) -  
\displaystyle{D_{n,j}(T)
\sum_{l=0}^{j-1}\int_0^T \| \mathcal{\mathbf{B^{\ast}_{n+l}}}(U_{n+l}) \|_{G_{n+l}}^2\,dt\,.}
\end{align}
 On the other hand, applying the uniform observability estimate given
in Lemma~\ref{AL}  with a second
member for the equation for $u_{n+j}$ in
\eqref{NSHnmixed}, 
we deduce that there exist $\eta_{n+j}>0$ and $\delta_{n+j}>0$, such that
\begin{align*}
&
\eta_{n+j}\int_0^T \| \mathcal{\mathbf{B^{\ast}_{n+j}}}(U_{n+j}) \|_{G_{n+j}}^2 dt 
\\
&
\ge
\int_0^T e_1(U_{n+j})\,dt -  
C\delta_{n+j} \int_0^T (|u_{n-1}|^2 +|u_n|^2
+\sum_{k=1}^{j-1} |u_{n+k}|^2)\,dt \,.
\end{align*}
 Using \eqref{8x} in this last estimate together with \eqref{10x} and \eqref{x} for $j-1$, we obtain 
\begin{equation*}\label{11x}
 \displaystyle{e_{1}(U_{n+j})(0) \le \rho_{n,j}(T) \sum_{l=0}^j\int_0^T \| \mathcal{\mathbf{B^{\ast}_{n+l}}}(U_{n+l}) \|_{G_{n+l}}^2 dt \,.}
 \end{equation*}
  On the other hand, thanks to the equation verified by $u_{n+j}$ and as before, we have
$$
\displaystyle{\int_0^T |u_{n+j}|^2 \,dt \le C (e_1(U_{n+j})(0) + \int_0^T(|u_{n-1}|^2 +|u_n|^2+
\sum_{k=1}^{j-1} |u_{n+k}|^2)\,dt \,.}
$$
 Using the above estimates, we deduce that
$$
\displaystyle{\int_0^T |u_{n+j}|^2 \,dt \le s_{n,j}(T) \sum_{l=0}^j \int_0^T ||\mathcal{\mathbf{B^{\ast}_{n+l}}}(U_{n+l}) ||_{G_{n+l}}^2 dt \,.}
$$
 Hence the property \eqref{x} is proved for $j$, which concludes the proof.

To handle the control problem, we shall
need to prove the admissibility and observability properties under a slightly different form
(mainly for the case $B_{n+l} \in \mathcal{L}(G_{n+l},H)$ for
$l=0, \ldots, p$). We have the following results.

\begin{Lemma}\label{tech1nmixed}
Let $n \ge 2$ be an integer.  We assume 
that for all $i=2, \ldots,n$,
the operators $C_{i i-1}$ satisfy the assumption $(A2)_n$. We assume that the operators $C_{i\, j}$ 
and $C_{i\, j}^{\ast}$ are in $\mathcal{L}(H_k)$ for 
$i \in \{n+1, \ldots, n+p\}$ and $j \in \{n-1, \ldots, i-1\}$ and $k=0,1$.
 Then there exist $ C_1>0\,, C_2>0$ such that
 for all $U_0 \in \mathcal{H}_{n+p}$,  the following properties hold for 
 $U$ the solution of \eqref{NSHnmixed} with initial data $U_0$ and for $W=\mathcal{A}_{n+p}^{-1}U$.  
 \begin{align}
 \label{equivxNmixed}
 &
 C_1 \Big(\displaystyle{\sum_{i=1}^n e_{i-n}(U_i)} + \sum_{j=1}^p e_0(U_{n+j})
 \Big)  \\
 &
 \notag
 \le \Big(\displaystyle{\sum_{i=1}^n e_{1+i-n}(W_i)} + \sum_{j=1}^p e_1(W_{n+j})
 \Big)\le 
 C_2 \Big(\displaystyle{\sum_{i=1}^n e_{i-n}(U_i)} + \sum_{j=1}^p e_0(U_{n+j})
 \Big) 
 \end{align}
 \end{Lemma}

\begin{proof}[\bf Proof.]
 We first consider the system formed by the $n$ first equations, it is an independent $n$ bi-diagonal cascade system. Hence we can apply Lemma~\ref{tech1n} and in particular the estimates
 \eqref{equivxN} where $W$ is replaced by $U$ and $Z$ by $W$. This gives
 \begin{equation}\label{equivxU}
 C_1  \sum_{i=1}^n e_{i-n}(U_i) \le \sum_{i=1}^n e_{1+i-n}(W_i) \le
 C_2  \sum_{i=1}^n e_{i-n}(U_i)
 \end{equation}
Let $j \in \{1, \ldots, p\}$ be given. By definition of $W$, the equation for $w_{n+j}$ can be
 written as
 $$
 u_{n+j}^{\prime} + A w_{n+j} + \sum_{k=-1}^{j-1} C_{n+j n+k} w_{n+k}=0 \,.
 $$
 Hence we have
\begin{equation}\label{AUX1}
e_1(W_{n+j}) \le C\Big(e_0(U_{n+j}) + e_0(W_{n-1}) + e_0(W_n) +\sum_{k=1}^{j-1} |w_{n+k}|^2
\Big)\,,
\end{equation}
 with the convention that the sum from $k=1$ to $j-1$ vanishes when $j=1$. 
 We shall prove by induction on $j \in \{1,\ldots, p\}$, the property
 \begin{equation}\label{PJ}
 e_1(W_{n+j}) \le \\
 C\Big(
\displaystyle{\sum_{i=1}^n e_{i-n}(U_i)} + \sum_{k=1}^j e_0(U_{n+k})
 \Big) \,.
 \end{equation}
  Let us prove this property for $j=1$. Thanks to \eqref{AUX1} for $j=1$
 and to \eqref{equivxU} we easily deduce \eqref{PJ} for $j=1$. We assume that this
 property holds up to $j-1$. Using \eqref{equivxU} together with \eqref{AUX1} for $j$ and
 the property \eqref{PJ} for $k=1$ up to $k=j-1$, we obtain \eqref{PJ} for $j$. Summing the inequalities \eqref{PJ} for $j=1$ up to $j=p$ and using  the right inequality in \eqref{equivxU}, we
 obtain the right inequality in \eqref{equivxNmixed}. We prove the left inequality in \eqref{equivxU}
 as follows. Thanks to the equation for $w_{n+j}$ as above, we prove
\begin{align*} \label{AUX2}
e_0(U_{n+j})  & \le C\Big(e_1(W_{n+j}) + e_0(W_{n-1}) + e_0(W_n) +\sum_{k=1}^{j-1} |w_{n+k}|^2
\Big) \\
&
 \le 
C \Big(\sum_{k=1}^j e_1(W_{n+k}) + e_0(W_{n-1}) + e_0(W_n)\Big) \,.
\end{align*} 
 Summing these inequalities from $j=1$ up to $j=p$ and using the
left inequality  in \eqref{equivxU}, we obtain the left inequality in \eqref{equivxNmixed}.
\end{proof}

\begin{Lemma}\label{obsdirnmixed}
Let $n \ge 2$ be an integer.  We assume 
that for all $i=2, \ldots,n$,
the operators $C_{i i-1}$ satisfy the assumption $(A2)_n$ where the operators
$\Pi_i$ satisfy $(A3)_n$. We assume that the operators $C_{i\, j}$ 
and $C_{i\, j}^{\ast}$ are in $\mathcal{L}(H_k)$ for 
$i \in \{n+1, \ldots, n+p\}$ and $j \in \{n-1, \ldots, i-1\}$ and $k=0,1$.  Moreover let $\mathcal{\mathbf{B^{\ast}_{n+j}}}$ for $j=0$ to $j=p$ be
any given operators satisfying $(A4)_{n+j}-(A5)_{n+j}$ for all $j$ in $\{0,\ldots,p\}$. 
Let $T>0$ be given.
For $W^T=(w_1^T, \ldots, w_{n+p}^T, q_1^T, \ldots, q_{n+p}^T) \in M_{-(n+p-1)}$, we denote by $W=(w_1, \ldots, w_{n+p}, w_1^{\prime}, \ldots, w_{n+p}^{\prime})$ the unique solution
in $\mathcal{C}^0([0,T]; ((H_{1-n})^n\times (H_{-1})^p) \times ( (H_{-n})^n \times (H_{-2})^p))$ of 
\begin{equation}\label{Wn+p}
\begin{cases}
w_1^{\prime\prime} + A w_1 = 0 \,,\\
w_i^{\prime\prime} + A w_i+ C_{i i-1}w_{i-1}  = 0 \,, 2 \le i \le n \,, \\
\displaystyle{w_{i}^{\prime\prime} + A w_{i}+ \sum_{k=n-1}^{i-1} C_{i k}\,w_k =0 \,,  n+1 \le i \le n+p \,,}\\
W_{|t=T}=W^T \,,
\end{cases}
\end{equation}
 Then $W$ satisfies the following properties

\begin{itemize}
\item[(i)] $W \in \mathcal{C}^0([0,T]; M_{-(n+p-1)})$,

\item[(ii)] There exists $C_1=C_1(T)>0$, such that
\begin{equation*}\label{directweak1nmixed}
C_1 \int_0^T \sum_{l=0}^p || \mathcal{\mathbf{B}}_{n+l}^{\ast}Z_{n+l}||_{G_{n+l}}^2 \,dt \le 
\displaystyle{\sum_{i=1}^n e_{i-n}(W_i)(0)} + \sum_{j=1}^p e_0(W_{n+j}) \,,
\end{equation*}
 where $Z=\mathcal{A}_{n+p}^{-1}W$.

\item[(iii)]
 For all $T>T_{n+p}^{\ast}$, where 
 $T_{n+p}^{\ast}$ is given in Theorem~\ref{mixedsystem}, 
 there exists $C_2=C_2(T)>0$ such that
\begin{equation*}\label{directweak2nmixed}
\displaystyle{\sum_{i=1}^n e_{i-n}(W_i)(0)} + \sum_{j=1}^p e_0(W_{n+j})
\le C_2  \int_0^T \sum_{l=0}^p|| \mathcal{\mathbf{B}}_{n+l}^{\ast}Z_{n+l}||_{G_{n+l}}^2 \,dt \,,
\end{equation*}
\item[(iv)]
 Assume furthermore that $\mathcal{\mathbf{B}}_{n+l}^{\ast}(w,w^{\prime})=B_{n+l}^{\ast}w^{\prime}$ for all $l \in \{0, \ldots, p\}$. Then
properties $(ii)-(iii)$ become
\begin{equation*}\label{directweak1bisnmixed}
C_1 \int_0^T \sum_{l=0}^p ||B_{n+l}^{\ast}w_{n+l}||_{G_{n+l}}^2 \,dt \le\, \displaystyle{\sum_{i=1}^n e_{i-n}(W_i)(0)} + \sum_{j=1}^p e_0(W_{n+j}) \,.
\end{equation*}
For all $T>T_{n+p}^{\ast}$, we have
\begin{equation*}\label{directweak2bisnmixed}
\displaystyle{\sum_{i=1}^n e_{i-n}(W_i)(0)} + \sum_{j=1}^p e_0(W_{n+j})
\le C_2   \int_0^T \sum_{l=0}^p ||B_{n+l}^{\ast}w_{n+l}||_{G_{n+l}}^2 \,dt \,,
\end{equation*}
 with the same constants $C_1$ and $C_2$ than in $(ii)-(iii)$.

\end{itemize}
\end{Lemma}

\begin{proof}[\bf Proof.]
Considering the $n$ first equations of the system in $W$, and using the results of Theorem~\ref{obsdirn}, we deduce that $(w_1, \ldots, w_n, w_1^{\prime}, \ldots, w_n^{\prime}) \in
\mathcal{C}([0,T]; X_{-(n-1)})$.  Considering the equation in $w_{n+1}$ and our assumption on
$C_{n+1n-1}$ and $C_{n+1 n}$ we deduce that $(w_{n+1},w_{n+1}^{\prime}) \in
\mathcal{C}([0,T]; H \times H_{-1})$. In a similar way we prove recursively  that $(w_{n+j},w_{n+j}^{\prime}) \in
\mathcal{C}([0,T]; H \times H_{-1})$ for $j=1$ up to $j=p$. This proves that $W \in \mathcal{C}^0([0,T]; M_{-(n+p-1)})$. The properties $(ii)-(iv)$ follow easily from 
\eqref{admissineqimixed} and \eqref{eqobskmixed}-\eqref{eqobskmixed2} together with\eqref{equivxNmixed} where $U$ is replaced by $W$ and $W$ by $Z$.
\end{proof}

\begin{Lemma}\label{obsdirnmixedide}
Let $n \ge 2$ be an integer.  We assume 
that for all $i=2, \ldots,n$,
the operators $C_{i i-1}$ satisfy the assumption $(A2)_n$ where the operators
$\Pi_i$ satisfy $(A3)_n$. We assume that the operators $C_{i\, j}$ 
and $C_{i\, j}^{\ast}$ are in $\mathcal{L}(H_k)$ for 
$i \in \{n+1, \ldots, n+p\}$ and $j \in \{n-1, \ldots, i-1\}$ and $k=0,1$.  Moreover let $\mathcal{\mathbf{B^{\ast}_{n+j}}}$ for $j=0$ to $j=p$ be
any given operators satisfying $(A4)_{n+j}-(A5)_{n+j}$ for all $j$ in $\{0,\ldots,p\}$. 
Let $T>0$ be given.
For 
$$
W^T=(w_1^T, \ldots, w_{n+p}^T, q_1^T, \ldots, q_{n+p}^T) \in M_{(n+p-1)} ,
$$
 we denote by 
 $$
 W=(w_1, \ldots, w_{n+p}, w_1^{\prime}, \ldots, w_{n+p}^{\prime})
 $$
  the unique solution
in $\mathcal{C}^0([0,T]; ((H_{1-n})^n\times (H_{-1})^p) \times ( (H_{-n})^n \times (H_{-2})^p))$ of 
\eqref{Wn+p}. Then we have $W \in \mathcal{C}^0([0,T]; M_{(n+p-1)})$.
\end{Lemma}

\noindent
{\bf Proof.}
The proof is similar to that of $(i)$ in Lemma~\ref{obsdirnmixed} and is left to the reader.
\qed

To handle the case of mixed unbounded and bounded control operators, we shall need
a different formulation of admissibility and observability properties. We proceed as follows.

\begin{Lemma}\label{tech1nmixedobs}
Let $n \ge 2$ be an integer.  We assume 
that for all $i=2, \ldots,n$,
the operators $C_{i i-1}$ satisfy the assumption $(A2)_n$. We assume that the operators $C_{i\, j}$ 
and $C_{i\, j}^{\ast}$ are in $\mathcal{L}(H_k)$ for 
$i \in \{n+1, \ldots, n+p\}$ and $j \in \{n-1, \ldots, i-1\}$ and $k=0,1$. Let $q$ be a fixed integer 
such that $q \in [1, p]$.
 Then there exist $ C_1>0\,, C_2>0$ such that
 for all $U_0 \in \mathcal{H}_{n+p}$,  the following properties hold for 
 $U$ the solution of \eqref{NSHnmixed} with initial data $U_0$ and for $W=\mathcal{A}_{n+p}^{-1}U$.  
 \begin{align}
 \label{equivxNmixedobs}
 &
 C_1 \Big(\displaystyle{\sum_{i=1}^n e_{1+i-n}(U_i) + \sum_{k=1}^q e_1(U_{n+k}) +
 \sum_{k=q+1}^p e_0(U_{n+k}) }
 \Big) \\
 &
 \notag
 \le  
 \Big(\displaystyle{\sum_{i=1}^n e_{1+i-n}(U_i)+ \sum_{k=1}^q e_1(U_{n+k})}+
 \displaystyle{\sum_{k=q+1}^p e_1(W_{n+k})}
 \Big)\\
 &
 \notag
 \le  
 C_2 \Big(\displaystyle{\sum_{i=1}^n e_{1+i-n}(U_i) + \sum_{k=1}^q e_1(U_{n+k}) +
 \sum_{k=q+1}^p e_0(U_{n+k}) }
 \Big) \,.
 \end{align}
 \end{Lemma}
 
\noindent
{\bf Proof.}
 If $q=p$ \eqref{equivxNmixedobs} trivially holds. Let us assume that $q \le p-1$. Let $j$ be an integer in $[q+1, p]$. The equation for $w_{n+j}$ as before implies that
 \begin{equation}\label{AVX1}
 e_0(U_{n+j}) \le C \Big(
 e_1(W_{n+j}) + |w_{n-1}|^2 +|w_n|^2 + \sum_{k=1}^q |w_{n+k}|^2 +
 \sum_{k=q+1}^{j-1} |w_{n+k}|^2
 \Big) \,.
 \end{equation}
  Thanks to \eqref{equivxNmixed} where $q$ replaces $p$, we have 
 \begin{align}\label{AVX2}
|w_{n-1}|^2 +|w_n|^2 + \sum_{k=1}^q |w_{n+k}|^2  & \le
C\Big(
\sum_{i=1}^n e_{i-n}(U_i) + \sum_{k=1}^q e_0(U_{n+k})
\Big)  
\\
&
\notag
\le
C\Big(
\sum_{i=1}^n e_{1+i-n}(U_i) + \sum_{k=1}^q e_1(U_{n+k})
\Big)\,.
\end{align}
 Using this inequality in \eqref{AVX1}, we obtain the left inequality of
\eqref{equivxNmixedobs}. Let $j$ be an integer in $[q+1, p]$. Thanks to \eqref{PJ}
we have
\begin{align*}
e_1(W_{n+j})  & \le C \Big(\displaystyle{\sum_{i=1}^n e_{i-n}(U_i) + \sum_{k=1}^j e_0(U_{n+k})}
\\
&
\le  
 C_2 \Big(\displaystyle{\sum_{i=1}^n e_{1+i-n}(U_i) + \sum_{k=1}^q e_1(U_{n+k})+
\sum_{k=q+1}^j e_0(U_{n+k})}
 \Big)\,.
\end{align*}
  Thus, we obtain the right inequality of \eqref{equivxNmixedobs}.
\qed

\begin{Theorem}\label{obsdirnmixedobs}
Let $n \ge 2$ be an integer.  We assume 
that for all $i=2, \ldots,n$,
the operators $C_{i i-1}$ satisfy the assumption $(A2)_n$ where the operators
$\Pi_i$ satisfy $(A3)_n$. We assume that the operators $C_{i\, j}$ 
and $C_{i\, j}^{\ast}$ are in $\mathcal{L}(H_k)$ for 
$i \in \{n+1, \ldots, n+p\}$ and $j \in \{n-1, \ldots, i-1\}$ and $k=0,1$.  Moreover let $\mathcal{\mathbf{B^{\ast}_{n+j}}}$ for $j=0$ to $j=p$ be
any given operators satisfying $(A4)_{n+j}-(A5)_{n+j}$ for all $j$ in $\{0,\ldots,p\}$. Let $T>0$ be given.
For $W^T=(w_1^T, \ldots, w_{n+p}^T, q_1^T, \ldots, q_{n+p}^T) \in N_{(n+p-1)}$, we denote by $W=(w_1, \ldots, w_{n+p}, w_1^{\prime}, \ldots, w_{n+p}^{\prime})$ the unique solution
in $\mathcal{C}^0([0,T]; ((H_{1-n})^n\times (H_{-1})^p) \times ( (H_{-n})^n \times (H_{-2})^p))$ of 
\eqref{Wn+p}. Then $W$
satisfies the following properties

\begin{itemize}
\item[(i)] $W \in \mathcal{C}^0([0,T]; N_{(n+p-1)})$,

\item[(ii)] There exists $C_1=C_1(T)>0$, such that
\begin{align}\label{directweak1nmixedobs}
C_1\Big( \int_0^T \sum_{l=0}^q || \mathcal{\mathbf{B}}_{n+l}^{\ast}U_{n+l}||_{G_{n+l}}^2 
+ \sum_{l=q+1}^p || \mathcal{\mathbf{B}}_{n+l}^{\ast}W_{n+l}||_{G_{n+l}}^2\,dt\Big)
\le \\
\displaystyle{\sum_{i=1}^n e_{1+i-n}(U_i)(0)} + \sum_{j=1}^q e_1(U_{n+j})(0)+
 \sum_{j=q+1}^p e_0(U_{n+j})(0)\,,
\end{align}
 where $W=\mathcal{A}_{n+p}^{-1}U$.

\item[(iii)] For all $T>T_{n+p}^{\ast}$ where $T_{n+p}^{\ast}$ is defined in
Theorem~\ref{mixedsystem}, there exists $C_2=C_2(T)>0$ such that
\begin{align}
\label{directweak2nmixedobs}
&
\displaystyle{\sum_{i=1}^n e_{1+i-n}(U_i)(0)} + \sum_{j=1}^q e_1(U_{n+j})(0)+
 \sum_{j=q+1}^p e_0(U_{n+j})(0) \\
 &
 \notag
 \le  
C_2 \Big( \int_0^T \sum_{l=0}^q || \mathcal{\mathbf{B}}_{n+l}^{\ast}U_{n+l}||_{G_{n+l}}^2 
+ \sum_{l=q+1}^p || \mathcal{\mathbf{B}}_{n+l}^{\ast}W_{n+l}||_{G_{n+l}}^2\,dt\Big)
 \,.
\end{align}
\end{itemize}
\end{Theorem}

\begin{proof}[\bf Proof.]
We prove $(i)$ as before. Considering the $n+q$ first equations we can apply
Theorem~\ref{mixedsystem}. Thus we have for all $T>0$ there exists $C_2(T)>0$ such that
\begin{align}
\label{ATX1}
&
 \int_0^T \sum_{l=0}^q || \mathcal{\mathbf{B}}_{n+l}^{\ast}U_{n+l}||_{G_{n+l}}^2 
\, dt \\
&
\notag
\le  
C_2(T)\Big(\displaystyle{\sum_{i=1}^n e_{1+i-n}(U_i)(0)} + \sum_{j=1}^q e_1(U_{n+j})(0)
\Big)\,,
\end{align}
 and \eqref{DIRJ} holds for $j=1, \ldots, q$. Let $j$ be an integer  in $[q+1,p]$. Using the property $(A4)_{n+j}$ 
for the equation in $w_{n+j}$ together with the classical energy estimates and \eqref{equivxNmixed}, we obtain
\begin{align}
\label{ATX3}
&
\int_0^T  || \mathcal{\mathbf{B}}_{n+j}^{\ast}W_{n+j}||_{G_{n+j}}^2 
  dt \\
  &
  \notag
\le
C(T) \Big(e_1(W_{n+j})(0) +  
\int_0^T |w_{n-1}|^2 +|w_n|^2 + \sum_{k=1}^{j-1} |w_{n+k}|^2  dt
\Big) .
\end{align}
 Using \eqref{DIRJ} together with \eqref{admiss*} and \eqref{admissx*}, we obtain
\begin{align}\label{ATX4} 
\int_0^T  || \mathcal{\mathbf{B}}_{n+j}^{\ast}W_{n+j}||_{G_{n+j}}^2 
\, dt  
\le
C(T)
\Big(\displaystyle{\sum_{i=1}^n e_{1+i-n}(W_i)(0)} +  
\sum_{j=1}^q e_1(W_{n+j})(0)
\Big)\,.
\end{align} 
 Using \eqref{equivxNmixed}, we obtain for all $j \in [q+1, \ldots, p]$
\begin{align} \label{ATX5}
&
\int_0^T  || \mathcal{\mathbf{B}}_{n+j}^{\ast}W_{n+j}||_{G_{n+j}}^2 
  dt 
\le
C(T)
\Big(\displaystyle{\sum_{i=1}^n e_{i-n}(U_i)(0)} + \sum_{j=1}^p e_0(U_{n+j})(0)
\Big)  
  \notag  \\ 
  &
  \le
C(T) \Big(\displaystyle{\sum_{i=1}^n e_{1+i-n}(U_i)(0)} + \sum_{j=1}^q e_1(U_{n+j})(0)
+ \sum_{l=q+1}^p e_0(U_{n+j})(0)
\Big)
\,.
\end{align} 
 Hence thanks to \eqref{ATX1} and \eqref{ATX5}, we proved \eqref{directweak1nmixedobs}.
We now consider the proof of the observability inequality \eqref{directweak2nmixedobs}.
Thanks to Theorem~\ref{mixedsystem} applied for the system formed by the first $n+q$
equations, we have for all $T>T_{n+p}^{\ast}$, there exists $C_1(T)>0$ such that
\begin{align}\label{ATX2}
C_1(T)\Big(\displaystyle{\sum_{i=1}^n e_{1+i-n}(U_i)(0)} + \sum_{j=1}^q e_1(U_{n+j})(0)
\Big) \le  
 \int_0^T \sum_{l=0}^q || \mathcal{\mathbf{B}}_{n+l}^{\ast}U_{n+l}||_{G_{n+l}}^2 
\, dt \,.
\end{align}
 Thanks to \eqref{DIRJ} applied to the system of the first $n+q$ equations satisfied
by $W$, together with \eqref{equivxNmixed}, we have 
\begin{align*}
\int_0^T|w_{n+j}|^2\,dt & \le D(T)  \Big(\sum_{i=1}^n e_{1+i-n}(W_i)(0) +
\sum_{k=1}^j e_{1}(W_{n+k})(0)
\Big)   \\
&
\le 
D(T)  \Big(\sum_{i=1}^n e_{i-n}(U_i)(0) +
\sum_{k=1}^q e_0(U_{n+k})(0)
\Big)  \\
\le  
D(T)  \Big(\sum_{i=1}^n e_{1+i-n}(U_i)(0)  & +
\sum_{k=1}^q e_{1}(U_{n+k})(0)
\Big)  \,,\quad \forall
 \ j \in \{1, \ldots, q\} \,.
\end{align*}
 Using \eqref{ATX2} in this last inequality, we obtain
\begin{equation}\label{NIRV0}
\int_0^T|w_{n+j}|^2\,dt \le D(T) \int_0^T \sum_{l=0}^q || \mathcal{\mathbf{B}}_{n+l}^{\ast}U_{n+l}||_{G_{n+l}}^2 \,,\quad \forall
 \ j \in \{1, \ldots, q\} \,.
\end{equation} 
 We consider the equation satisfied by
$w_{n+q+1}$. We shall prove by induction on $k \in \{q+1,\ldots, p\}$ that
the following property holds
\begin{equation} \label{NIRV}
\begin{cases}
\displaystyle{e_1(W_{n+k})(0) \le C(T)\Big( \int_0^T \sum_{l=0}^{q} || \mathcal{\mathbf{B}}_{n+l}^{\ast}U_{n+l}||_{G_{n+l}}^2 + }\\
\displaystyle{\sum_{l=q+1}^{k} || \mathcal{\mathbf{B}}_{n+l}^{\ast}W_{n+l}||_{G_{n+l}}^2
\, dt
\Big)}\,,\\
\displaystyle{\int_0^T |w_{n+k}|^2 \,dt \le C(T)\Big( \int_0^T \sum_{l=0}^{q} || \mathcal{\mathbf{B}}_{n+l}^{\ast}U_{n+l}||_{G_{n+l}}^2 + }\\
\displaystyle{\sum_{l=q+1}^{k} || \mathcal{\mathbf{B}}_{n+l}^{\ast}W_{n+l}||_{G_{n+l}}^2
\, dt
\Big)} \,.
\end{cases}
\end{equation}
 We first prove it for $k=q+1$. Thanks to the assumption $(A5)_{n+q+1}$ and to the uniform observability
estimate with an inhomogeneity given in Lemma~\ref{AL}, there exist
$\eta_{n+q+1}>0$ and $\delta_{n+q+1}>0$ such that for
all $T>T_{n+q}^{\ast}$ we have
\begin{align*}
&
\eta_{n+q+1} \int_0^T  || \mathcal{\mathbf{B}}_{n+q+1}^{\ast}W_{n+q+1}||_{G_{n+q+1}}^2 
\\
&
 \ge
\int_0^T e_1(W_{n+q+1})\,dt - 
\delta_{n+q+1}\int_0^T ( |w_{n-1}|^2 + |w_n|^2 +
\sum_{k=1}^q |w_{n+k}|^2)\,dt\,.
\end{align*}
 Using \eqref{NIRV0} in this last inequality, together with \eqref{admiss*}, \eqref{admissx*}
and \eqref{equivxU} and \eqref{ATX2}, we deduce that there exists $C(T)>0$ such that
\begin{align} \label{ATX3q}
&
C(T) \int_0^T \big(||\mathcal{\mathbf{B}}_{n+q+1}^{\ast}W_{n+q+1}||_{G_{n+q+1}}^2  
+\sum_{l=0}^{q} || \mathcal{\mathbf{B}}_{n+l}^{\ast}U_{n+l}||_{G_{n+l}}^2 \big)
\, dt   \\
&
\ge  
\int_0^T e_1(W_{n+q+1})\,dt\,.
\notag
\end{align}
 Considering once again the equation for $w_{n+q+1}$ and proceeding as
in the proof of Theorem~\ref{mixedsystem}, we have
\begin{align*}
&
 (1+T)\int_0^Te_1(W_{n+q+1})\,dt \\
 &
 \ge Te_1(W_{n+q+1})(0) -   
\displaystyle{CT\int_0^T(|w_{n-1}|^2 +|w_n|^2
+\sum_{k=1}^{q} |w_{n+k}|^2)\,dt\,.}
\end{align*}
 Using once again \eqref{NIRV0} in this last inequality 
 together with \eqref{admiss*}, \eqref{admissx*} and 
 \eqref{equivxU} and \eqref{ATX2}, we have
\begin{align*}
 (1+T) \! 
 \int_0^T \!  \! 
 e_1(W_{n+q+1}) dt  \! \ge  \! Te_1(W_{n+q+1})(0) -  
 C(T) \! 
\int_0^T  \! \sum_{l=0}^{q} \!  \| \mathcal{\mathbf{B}}_{n+l}^{\ast}U_{n+l}\|_{G_{n+l}}^2
 . 
\end{align*}
 Combining \eqref{ATX3q} and this last inequality, we deduce that
\begin{align*}
&
e_1(W_{n+q+1})(0) \\
&
\le C(T) \int_0^T\big( \sum_{l=0}^{q} || \mathcal{\mathbf{B}}_{n+l}^{\ast}U_{n+l}||_{G_{n+l}}^2 + 
 ||\mathcal{\mathbf{B}}_{n+q+1}^{\ast}W_{n+q+1}||_{G_{n+q+1}}^2
\big)  dt  .
\end{align*}
 On the other hand, using once again the equation for $w_{n+q+1}$ the above
observability inequality and
the energy estimates we already used in the proof of Theorem~\ref{mixedsystem}, we have
\begin{align*}
&
\int_0^T |w_{n+q+1}|^2\,dt \\
&
\le C(T) \int_0^T \big( \sum_{l=0}^{q} || \mathcal{\mathbf{B}}_{n+l}^{\ast}U_{n+l}||_{G_{n+l}}^2 + 
||\mathcal{\mathbf{B}}_{n+q+1}^{\ast}W_{n+q+1}||_{G_{n+q+1}}^2
\big)\, dt \,.
\end{align*}
 Assuming that \eqref{NIRV} holds up to $k-1$ for $k \in \{q+1,\ldots, p\}$ and
proceeding in the same way as above, we easily prove that
\eqref{NIRV} holds for $k$. Finally combining \eqref{ATX2} together with \eqref{NIRV} for
all $k \in \{q+1,\ldots, p\}$ and using the left inequality of \eqref{equivxNmixedobs}, we
obtain \eqref{directweak2nmixedobs}.
\end{proof}

The proofs of Theorem ~\ref{control2n+pq} and 
Theorem~\ref{control2n+pqmixed} are similar to 
that of Theorem~\ref{control2n} and are left to the reader.

\smallskip

\noindent 
{\bf Acknowledgments.}
  I would like to thank Piermarco Cannarsa for fruitful 
  discussions and suggestions on this paper. I am also 
  very grateful to the referee for his/her careful reading,
   and valuable comments and suggestions.

\end{document}